\documentclass[10pt]{article}
\usepackage[mathscr]{eucal}

\usepackage{graphicx} 
\usepackage{url}
\usepackage{stmaryrd}

\newcommand\DMO[2]{\DeclareMathOperator{#1}{#2}}

\usepackage{amsmath,amsthm,amsfonts,amscd,amssymb}
\usepackage[usenames,dvipsnames]{color}

\usepackage[all]{xy} 

\usepackage{makeidx}

\newcounter{enumitemp}
\newenvironment{enumeratecontinue}{
  \setcounter{enumitemp}{\value{enumi}}
  \begin{enumerate}
  \setcounter{enumi}{\value{enumitemp}}
}
{
  \end{enumerate}
}

\numberwithin{equation}{section}

\newcommand\pref[1]{(\ref{#1})}

\newcommand\marginparLee[1]{\marginpar{\tiny #1 --- Lee}}

\newcommand\ds\displaystyle

\theoremstyle{plain}

\newtheorem*{KuroshTheorem}{Kurosh Subgroup Theorem}

\newtheorem*{theorem*}{Theorem}
\newtheorem*{proposition*}{Proposition}

\newtheorem{theorem}{Theorem}[section]
\newtheorem{proposition}[theorem]{Proposition}
\newtheorem{lemma}[theorem]{Lemma}

\newtheorem{corollary}[theorem]{Corollary}

\theoremstyle{definition}
\newtheorem{definition}[theorem]{Definition}

\DeclareMathOperator{\Out}{Out}
\DeclareMathOperator{\Aut}{Aut}

\DeclareMathOperator{\rank}{rank}

\DeclareMathOperator{\Stab}{Stab}
\DeclareMathOperator\diam{diam}

\DeclareMathOperator\closure{cl}

\DeclareMathOperator{\MCG}{\mathsf{MCG}}

\DeclareMathOperator\corank{corank}
\DeclareMathOperator\DFF{D_{\text{FF}}}
\DeclareMathOperator\DFS{D_{\text{FS}}}

\DeclareMathOperator\Ax{Ax}
\DeclareMathOperator\core{core}

\newcommand\reals{{\mathbf R}}

\newcommand\Z{{\mathbf Z}}

\newcommand\inject{\hookrightarrow}

\newcommand{\bdy}{\partial}
\newcommand{\from}{\colon}
\newcommand\composed{\circ}
\newcommand\suchthat{\bigm|}
\newcommand\inv{{-1}}
\newcommand\union{\cup}

\newcommand\abs[1]{\left| #1 \right|}

\newcommand\intersect{\cap}

\newcommand\meet{\wedge}

\newcommand\restrict{\bigm|}
\newcommand\subgroup{<}
\newcommand\cross{\times}

\renewcommand\H{{\mathcal H}}

\newcommand\K{{\mathcal K}}
\renewcommand\L{\mathcal L}

\newcommand\V{\mathcal V}

\newcommand\A{\mathscr A}
\newcommand\B{\mathscr B}
\newcommand\C{\mathscr C}
\newcommand\F{\mathscr F}

\newcommand\Fell{\F_{\!ell}}

\newcommand\CFFS{\mathcal{F\!F}}

\newcommand\FFC{{\mathcal{F}}}

\newcommand\<\langle
\renewcommand\>\rangle
\newcommand\wh\widehat
\newcommand\disjunion\sqcup

\DeclareMathOperator\interior{int}
\DeclareMathOperator\frontier{Fr}
\DeclareMathOperator\Fr{Fr}

\newcommand\act\curvearrowright

\newcommand\X{\mathcal{X}}
\newcommand\CV\X

\newcommand\BookOneTag{BFH:TitsOne}
\newcommand\BookOne{\cite{\BookOneTag}}

\newcommand\FSOneTag{HandelMosher:FreeSplittingHyperbolic}
\newcommand\FSOne{\cite{\FSOneTag}}

\newcommand\STLOneTag{HandelMosher:RelComplexHypII}

\newcommand\STLTwoTag{HandelMosher:RelComplexHypIII}

\DMO\Core{Core}
\DMO\ACore{{\hat{\mathcal{C}}}}
\DMO\truss{truss}
\newcommand\wt\widetilde

\newcommand\FS{\mathcal{F\!S}}

\newcommand\collapsesto\succ
\newcommand\collapse\collapsesto
\newcommand\collapses\collapsesto
\newcommand\expandsto\prec
\newcommand\expand\expandsto
\newcommand\expands\expandsto

\newcommand\relA{\emph{rel}~$\A$}

\title{Relative free splitting and free factor complexes I:\\ Hyperbolicity}

\author{Michael Handel and Lee Mosher \thanks{The first author was supported by National Science Foundation grants and by PSC-CUNY grants. The second author was supported by National Science Foundation grants.}}

\makeindex

\begin{document}

\maketitle

\begin{abstract}
In this work, the first of a three part study of free splitting and free factor complexes of a group~$\Gamma$ relative to a free factor system~$\A$, we prove that these complexes are hyperbolic. The proof yields information about coarsely transitive families of quasigeodesics in each of these complexes, expressed in terms of fold paths of free splittings.
\end{abstract}

\section{Introduction to Part I}
\label{SectionIntro}
Masur and Minsky, in their papers \cite{MasurMinsky:complex1} and \cite{MasurMinsky:complex2}, introduced a hierarchy of connected simplicial complexes associated to a finite type surface $S$: at the top of the hierarchy is the curve complex of $S$; and at lower levels are the curve complexes of essential, connected subsurfaces of~$S$. They proved hyperbolicity of the curve complexes of all finite type surfaces, which applies immediately to all levels of the hierarchy of~$S$. This hierarchy of hyperbolic complexes has proved immensely useful in many applications to the large scale geometry of the mapping class group $\MCG(S)$ \,\cite{BestvinaFujiwara:bounded,BehrstockMinsky:rank,Mangahas:UniformUniform,BKMM,BBF:MCGquasitrees}.

For purposes of application to the large scale geometry of the outer automorphism group $\Out(F_n)$ of a rank~$n$ free group $F_n$, several complexes emerged as analogues, in different ways, of the curve complex of a finite type surface: Bestvina and Feighn proved hyperbolicity of the \emph{free factor complex} $\FFC(F_n)$ \cite{BestvinaFeighn:FFCHyp}; we proved hyperbolicity of the \emph{free splitting complex} $\FS(F_n)$ \FSOne, originally introduced as Hatcher's sphere complex \cite{Hatcher:HomStability}; and Mann proved hyperbolicity of the \emph{cyclic splitting complex} \cite{Mann:CyclicSplittingComplex}.

In this paper we study the large scale geometry of ``relative'' free factor and free splitting complexes of $F_n$ and of more general groups, proving their hyperbolicity. Our focus on these relative complexes has two motivations: as potential analogues of subsurface curve complexes for studying the large scale geometry of $\Out(F_n)$; and as analogues of the curve complex itself for studying the large scale geometry of outer automorphism groups of more general groups relative to a choice of free factor system (see below).

In the work of Masur and Minsky, closed surfaces and surfaces with boundary were treated on equal footing in \cite{MasurMinsky:complex1}, and so in \cite{MasurMinsky:complex2} the subsurface curve complexes could be immediately applied to study any compact surface, with or without boundary. Unlike that situation, it seems more appropriate to think of hyperbolicity of relative free factor and free splitting complexes of~$F_n$ as a \emph{generalization} of the absolute cases: stated in Theorems~\ref{TheoremRelFSHyp} and~\ref{TheoremRelFFHyp} below for $F_n$; and in Theorems~\ref{TheoremRelFSGammaHyp} and~\ref{TheoremRelFFGammaHyp} for more general groups. And while the \emph{need} for relativizing the results of \FSOne\ and \cite{BestvinaFeighn:FFCHyp} is perhaps motivated by considering subsurface curve complexes, the \emph{meaning} of relativization only became clear to us after considering deformation spaces of group actions on trees \cite{McCulloughMiller:symmetric,Forester:Deformation,GuirardelLevitt:outer,GuirardelLevitt:DefSpaces} that generalize the outer space of $F_n$ \cite{CullerVogtmann:moduli}. We use a definition of relative free splitting complexes related to deformation spaces of free splittings, as $\FS(F_n)$ is related to the outer space of $F_n$. This motivates Theorems~\ref{TheoremRelFSHyp} and~\ref{TheoremRelFFHyp} for $F_n$ relative to a free factor system, and Theorems~\ref{TheoremRelFSGammaHyp} and~\ref{TheoremRelFFGammaHyp} for any group $\Gamma$ relative to a free factor system. See also the work of Horbez \cite[Appendix A.1]{Horbez:HyperbolicGraphs} for an independent discussion of Theorem~\ref{TheoremRelFSGammaHyp}.

Because of the central interest in $F_n$, we first describe our results in that context.

\medskip\noindent
\textbf{The complex of relative free factor systems of $F_n$.} Free factor systems for $F_n$ were introduced in \BookOne\ to aid analysis of the dynamics of elements of $\Out(F_n)$. Formally a \emph{free factor system} of $F_n$ is a finite set of the form $\A = \{[A_1],\ldots,[A_I]\}$ such that there exists an internal free factorization $F_n = A_1 * \cdots * A_I * B$, $(I \ge 0)$, where each $A_k$ is nontrivial, and $[\cdot]$ denotes conjugacy class of a subgroup of~$F_n$. Elements of the set~$\A$ are referred to as its \emph{components}. The free factor $B$ is called a \emph{cofactor} of $\A$. And while a cofactor is far from unique, not even up to conjugacy, nonetheless its rank and thus its isomorphism type are well-defined (see Lemma~\ref{LemmaPartialFFCofactor}). Note that a cofactor may be trivial. Inclusion of free factors up to conjugacy induces a partial ordering on free factor systems which is denoted $\A \sqsubset \A'$.

Fixing one free factor system~$\A$ of $F_n$, the \emph{complex of free factor systems of $F_n$ rel~$\A$}, denoted $\CFFS(F_n;\A)$ and sometimes also called the \emph{complex of relative free factor systems}, is defined in terms of the partial ordering $\sqsubset$ on the set of free factor systems $\B$ such that $\A \sqsubset \B$. After removing from this poset its unique minimum, namely~$\A$, and its unique maximum, namely~$\{[F_n]\}$, the geometric realization of the resulting poset is, by definition, $\CFFS(F_n;\A)$. For certain \emph{exceptional} free factor systems $\A$ close to the maximum $\{[F_n]\}$, the complex $\CFFS(F_n;\A)$ exhibits the exceptional behavior of being either empty or 0-dimensional (see Section~\ref{SectionDFF}). 

The ``complex of free factor systems'' $\CFFS(F_n)$ is very tightly related to the ``free factor complex'' $\FFC(F_n)$ studied in \cite{BestvinaFeighn:FFCHyp}. See Proposition~\ref{PropF_Into_FF_QI} for details, but in brief: $\CFFS(F_n)$ contains $\FFC(F_n)$ as a quasi-isometrically embedded subcomplex. This same tight relation also holds more generally between $\CFFS(F_n;\A)$ and $\FFC(F_n;\A)$, and still more generally between $\CFFS(\Gamma;\A)$ and $\FFC(\Gamma;\A)$.  

\medskip\noindent
\textbf{Relative free splitting complexes of $F_n$.}
A \emph{free splitting} of $F_n$ is a minimal action of $F_n$ on a nontrivial simplicial tree $T$ with trivial edge stabilizers and with finitely many edge orbits. The set of conjugacy classes of nontrivial vertex stabilizers forms a free factor system of $F_n$ called the \emph{elliptic} free factor system of $T$, and denoted $\Fell T$ (see Section~\ref{SectionFSSucc}). Two free splittings which differ by an equivariant homeomorphism are equivalent. Collapsing invariant subgraphs of free splittings defines a partial ordering on equivalence classes which is denoted $S \collapsesto T$. Fixing a free factor system~$\A$, and restricting the collapse partial order to equivalence classes of those free splittings $T$ such that $\A \sqsubset \Fell T$ (here allowing $\A=\Fell T$), the geometric realization of the resulting partially ordered set is, by definition, the \emph{free splitting complex of~$F_n$ rel~$\A$}. Again, the familiar case $\FS(F_n)=\FS(F_n;\emptyset)$ is the free splitting complex of~$F_n$ as studied in~\FSOne. 

\medskip

\begin{theorem} 
\label{TheoremRelFSHyp}
For any nonfull free factor system $\A$ of $F_n$, the complex $\FS(F_n;\A)$ is nonempty, connected, and hyperbolic.
\end{theorem}

\begin{theorem} 
\label{TheoremRelFFHyp}
For any nonexceptional free factor system $\A$ of $\Gamma$, the complex $\CFFS(F_n;\A)$ is positive dimensional, connected, and hyperbolic.
\end{theorem}

\medskip\noindent
\textbf{Generalizing beyond $F_n$.} The construction by Culler and Vogtmann of the outer space of~$F_n$ \cite{CullerVogtmann:moduli} was extended by McCullough and Miller to construct an outer space of any free product $\Gamma = A_1 * \cdots * A_K$ in which the free factors $A_1,\ldots,A_K$ are freely indecomposable and not infinite cyclic \cite{McCulloughMiller:symmetric}. The further extension to the outer space of a group $\Gamma$ relative to a free factor system $\A$ occurs as a special case of very general constructions of deformation spaces of minimal actions of groups on trees \cite{Forester:Deformation,GuirardelLevitt:outer,GuirardelLevitt:DefSpaces}. 

The close relation between the outer space, free splitting complex, and free factor complex of $F_n$ extends to the context of relative versions of these spaces and complexes, for any group $\Gamma$ relative to any free factor system~$\A$. Our main results, Theorems~\ref{TheoremRelFSGammaHyp} and~\ref{TheoremRelFFGammaHyp} stated below, are generalizations of Theorems~\ref{TheoremRelFSHyp} and~\ref{TheoremRelFFHyp} to that extended context. The proofs of hyperbolicity of these relative complexes do not use any special assumption about the ambient group $\Gamma$, and except for some preliminaries regarding generalizations of very basic facts in the special case~$\Gamma = F_n$ found in Sections~\ref{SectionBasicFFDefs} and~\ref{SectionFFFAndRel}, the proofs are not much different in the general case (see the heading below entitled ``Theorem~\ref{TheoremRelFSGammaHyp}: Comparison of methods''). Theorems~\ref{TheoremRelFSGammaHyp} and~\ref{TheoremRelFFGammaHyp} are therefore intended as a contribution to a growing mathematical study of outer automorphism groups of freely decomposable groups --- both absolute, and relative to a choice of free factor system --- with a goal of developing analogies between theorems about these groups and theorems about $\Out(F_n)$. For other works in this genre see 
\cite{Horbez:BoundaryOfOuterSpace}, 
\cite{Martino:IndexTheorem}, \cite{McCulloughMiller:symmetric}, \cite{CollinsTurner:efficient}. For example we expect that Theorem~\ref{TheoremRelFSGammaHyp} and together with the results of Parts II and III of this work  \cite{\STLOneTag,\STLTwoTag} could be used to advance the study of bounded cohomology of subgroups of $\Out(\Gamma;\A)$, as was done for $\Out(F_n)$ in \cite{HandelMosher:BddCohomology} using \cite{HandelMosher:FreeSplittingHyperbolic,HandelMosher:FreeSplittingLox}.

\medskip\noindent
\textbf{Grushko's Theorem, the Kurosh Subgroup Theorem, and free factor systems.}
Before formulating Theorems~\ref{TheoremRelFSGammaHyp} and~\ref{TheoremRelFFGammaHyp} we first discuss free factor systems and their partial order $\sqsubset$ in the context of a general group. These topics have historical roots in Grushko's Theorem and the Kurosh Subgroup Theorem; see Sections~\ref{SectionFFSystems} and~\ref{SectionKurosh} for a full exposition. Grushko's Theorem can be formulated as follows: any finitely generated group $\Gamma$ has a free product decomposition of the form 
$$(*) \qquad\Gamma = A_1 * \ldots * A_I * B \qquad (I \ge 0)
$$ 
such that each free factor $A_i$ is nontrivial, freely indecomposable, and not infinite cyclic and such that the free factor $B$ is a finite rank free group (possibly trivial); under these conditions $(*)$ is called a \emph{Grushko decomposition} of $\Gamma$. While a Grusko decomposition is not unique, it has certain uniqueness properties as a corollary of the Kurosh Subgroup Theorem (the full statement of which can be found in Section~\ref{SectionFFSystems}): for any other Grushko decomposition $\Gamma = A'_1 * \cdots * A'_{I'} * B'$ \, ($I' \ge 0$), we have $I=I'$ and $\rank(B)=\rank(B')$, and there is a unique index permutation $\sigma$ such that the subgroups $A^{\vphantom{\prime}}_i,A'_{\sigma(i)}$ are conjugate in $\Gamma$ (for $i=1,\ldots,I$). 

One can obtain a different formulation of uniqueness of a Grushko decomposition expressed in terms of ``free factor systems''. Define a \emph{free factor system} of $\Gamma$ to be a set of the form $\A = \{[A_1],\ldots,[A_I]\}$ such that $\Gamma$ has a free product decomposition of the form $(*)$, where $[A_i]$ denotes the conjugacy class in $\Gamma$ of~$A_i$, and the free factors $A_i$ are required only to be nontrivial; again $B$ is a possibly trivial finite rank free group. See \cite{\BookOneTag} for the origin of this definition in the case $\Gamma=F_n$. If each $A_i$ is freely indecomposable and not infinite cyclic then we say that $\A$ is a \emph{Grushko free factor system} of~$\Gamma$. With these concepts in hand the uniqueness property says that the Grushko free factor system $\A$ of a finitely generated group $\Gamma$ is well-defined, as are its two numerical invariants: the cardinality $\abs{\A}=I$; and the quantity~$\rank(B)$ which we call the \emph{corank} of~$\A$ in $\Gamma$, denoted as $\corank(\Gamma;\A)$ or just $\corank(\A)$. 

The Kurosh Subgroup Theorem has another corollary: for any Grushko decomposition~$(*)$, and for any free factor $A' \subgroup \Gamma$, if $A'$ is not a free group then there exists $i \in \{1,\ldots,I\}$ such that $A_i$ is conjugate in $\Gamma$ to a free factor of $A'$. This leads to the following alternate expression of the uniqueness for Grushko decompositions. There is a natural partial order on free factor systems denoted $\A \sqsubset \A'$ meaning that each free factor representing an element of $\A$ is conjugate to a subgroup of some free factor representing an element of~$\A'$. The alternate uniqueness statement says that the Grushko free factor system $\A = \{[A_1],\ldots,[A_K]\}$ of a finitely generated group is the unique minimum of the partial ordering $\sqsubset$ on the set of free factor systems of~$\Gamma$. A converse is also true: if $\A$ is a free factor system of $\Gamma$ that is a minimum of $\sqsubset$, then $\A$ is the Grushko free factor system associated to a Grushko decomposition of $\Gamma$ (this also follows from the Kurosh Subgroup Theorem; see Proposition~\ref{PropGrushkoMinimum}).

We do not actually \emph{apply} Grushko's Theorem in this work, unless it is to conclude (using the notation defined just below) that $\Out(\Gamma) = \Out(\Gamma;\A)$ when $\Gamma$ is finitely generated and $\A$ is the Grushko free factor system of~$\Gamma$. On the other hand, we will be heavily applying the Kurosh Subgroup Theorem throughout this work, via its various corollaries that are developed in Section~\ref{SectionBasicFFDefs}.

\medskip\noindent
\textbf{Relative free splitting and free factor complexes in general.} Fix now an arbitrary group $\Gamma$ and a free factor system~$\A$ of $\Gamma$, not required to be a Grushko free factor system. The \emph{relative outer automorphism group} $\Out(\Gamma;\A)$ is defined to be the subgroup of $\Out(\Gamma)$ which fixes the subset~$\A$ under the action of $\Out(\Gamma)$ on the set of free factor systems of~$\Gamma$. This is the group whose virtual cohomological dimension is studied by Guirardel and Levitt \cite[Theorem~5.2]{GuirardelLevitt:outer} as an application of their construction of the outer space of $\Gamma$ rel~$\A$. Here in Part~I the group $\Out(\Gamma;\A)$ is mostly lurking behind the scenes, but see Section~\ref{SectionRelOut} and Section~\ref{SectionFFCHyp} for a record of basic facts. In Parts II and III  on the other hand \cite{HandelMosher:RelComplexHypII,HandelMosher:RelComplexHypIII},
the study of individual elements of $\Out(\Gamma;\A)$ rises to the~fore.

The \emph{complex of relative free factor systems of $\Gamma$ rel~$\A$}, denoted $\CFFS(\Gamma;\A)$, is defined to be the geometric realization of the partial ordering $\sqsubset$ restricted to the set of  free factor systems $\B$ of $\Gamma$ such that $\A \sqsubset \B$ and such that $\B$ is neither the minimal nor maximal element subject to that restriction, that is, $\B \ne \A$ and $\B \ne \{[\Gamma]\}$. Just as for $\Gamma=F_n$ (see above), there are \emph{exceptional} free factor systems, those closest to the maximum $\{[\Gamma]\}$, for which $\CFFS(\Gamma;\A)$ exhibits exceptional behavior (see Section~\ref{SectionDFF} and Proposition~\ref{PropExceptionalFFS}). 

A \emph{free splitting} of $\Gamma$ is a minimal action on a nontrivial simplicial tree $T$ with trivial edge stabilizers and with finitely many edge orbits. The conjugacy classes of nontrivial vertex stabilizers form a free factor system of $\Gamma$ denoted $\Fell T$. Two free splittings which differ by an equivariant homeomorphism are equivalent. Collapsing equivariant subgraphs of free splittings defines a partial ordering on equivalence classes which is denoted $S \collapsesto T$. To say that $T$ is a free splitting \emph{relative to $\A$} (almost always abbreviated to ``rel~$A$'') means that $\A \sqsubset \Fell T$; here we allow equality $\A=\Fell T$. The \emph{free splitting complex rel~$\A$}, denoted $\FS(\Gamma;\A)$, is the geometric realization of this partial order restricted to the equivalence classes of free splittings $T$ rel~$\A$. When $\A$ is a Grushko free factor system then one may think of $\FS(\Gamma;\A)$ as the \emph{absolute free splitting complex of~$\Gamma$}. Just as happens for $\FS(F_n)$ (c.f.~\cite{Hatcher:HomStability}), in general  $\FS(\Gamma;\A)$ is a kind of ``simplicial completion'' of the Guirardel-Levitt outer space of $\Gamma$ rel~$\A$ considered in \cite{GuirardelLevitt:outer}: that relative outer space is naturally the complement of the subcomplex of $\FS(\Gamma;\A)$ obtained from all $T$ such that the nesting relation $\A \sqsubset \Fell T$ is proper.


\begin{theorem} 
\label{TheoremRelFSGammaHyp}
For any group $\Gamma$ and any nonfull free factor system $\A$ of $\Gamma$, the free splitting complex $\FS(\Gamma;\A)$ is nonempty, connected, and hyperbolic.
\end{theorem}
\noindent
See also \cite{Horbez:HyperbolicGraphs} for a discussion of Theorem~\ref{TheoremRelFSGammaHyp} following \cite{BestvinaFeighn:subfactor}.

\begin{theorem} 
\label{TheoremRelFFGammaHyp}
For any group $\Gamma$ and any nonexceptional free factor system $\A$ of $\Gamma$, the complex of relative free factor systems $\CFFS(\Gamma;\A)$ is nonempty, connected, and hyperbolic.
\end{theorem}

\noindent
Both of these theorems have special outcomes in the case of an exceptional free factor $\A$: for Theorem~\ref{TheoremRelFSGammaHyp} see Section~\ref{SectionFSLow}; and for Theorem~\ref{TheoremRelFFGammaHyp} see Section~\ref{SectionFFHyperbolicProof}.

Theorem~\ref{TheoremRelFSGammaHyp} is proved here in Section~\ref{SectionFFRelAHyp}, and Theorem~\ref{TheoremRelFFGammaHyp} in Section~\ref{SectionFFCHyp}. Outlines of the proofs of those two theorems can be found in the \emph{Overview} \cite{HandelMosher:RelHypComplexIntro}. Also, see below for a comparison of our methods of proof of Theorem~\ref{TheoremRelFSGammaHyp} with methods of earlier works. 

But first we briefly mention two further results arising from the methods of proof of Theorem~\ref{TheoremRelFSGammaHyp}, each of which will be applied in Part III~\cite{\STLTwoTag} where we study the translation lengths of elements of $\Out(\Gamma;\A)$ acting on $\FS(\Gamma;\A)$ and on $\CFFS(\Gamma;\A)$.

\smallskip
\textbf{Theorem~\ref{TheoremRelFSUParams} and Corollary~\ref{CorollaryCompFSU}: Quasigeodesic fold paths.} Theorem~\ref{TheoremRelFSUParams} describes a very explicit coarsely transitive family of uniform quasigeodesics in $\FS(\Gamma;\A)$, namely the family of Stallings fold paths, reparameterized using a measurement of combinatorial change along a fold path which we call \emph{free splitting units}; see Section~\ref{SectionFSU} for details of these units. Also, Corollary~\ref{CorollaryCompFSU} describes a further reparameterization using \emph{component free splitting units}, with less efficient constants but much more easily applicable.

\smallskip
\textbf{Uniformity of constants.} Our main results --- Theorem~\ref{TheoremRelFSGammaHyp}, Theorem~\ref{TheoremRelFFGammaHyp}, Theorem~\ref{TheoremRelFSUParams}, and Corollary~\ref{CorollaryCompFSU} --- each express the existence of various constants: hyperbolicity constants, quasigeodesic constants, coarse Lipschitz constants, etc. These constants each depend only on two numerical invariants of the group $\Gamma$ and its free factor system~$\A$, namely $\corank(\Gamma;\A)$ and the cardinality~$\abs{\A}$. For the foundational special case $\Gamma=F_n$ and $\A=\emptyset$ covered in Theorems~\ref{TheoremRelFSHyp} and~\ref{TheoremRelFFHyp}, these constants depend only on $\corank(F_n;\emptyset) = n$. In the general case an interesting feature arises: the constants are completely independent of the isomorphism classes of subgroups representing the elements of the free factor system~$\A$. In a few cases the constants are independent even of $\Gamma$ and $\A$; see for example Proposition~\ref{PropConnectedLipschitz} describing a $4$-Lipschitz projection map $\FS(\Gamma;\A) \to \CFFS(\Gamma;\A)$.

\medskip\noindent
\textbf{Theorem~\ref{TheoremRelFSGammaHyp}: Comparison of methods.} Here we briefly compare our current methods of proof for Theorem~\ref{TheoremRelFSGammaHyp} for $\Out(\Gamma;\A)$ to methods of the earlier works \cite{\FSOneTag}, \cite{BestvinaFeighn:subfactor} for $\Out(F_n)$, and \cite{Horbez:HyperbolicGraphs} for $\Out(\Gamma;\A)$. While the broad outlines are similar, various refinements and improvements are important for later application in Parts II and III \cite{\STLOneTag,\STLTwoTag}, in particular improved versions of ``free splitting units'' (see Proposition~\ref{PropMMTranslation} and Corollary~\ref{CorollaryCompFSU}). Also important for application in Part~III \cite{\STLTwoTag} is a previously unremarked consequence of the Masur--Minsky axioms, namely the \emph{quasi-closest point property}; see Section~\ref{SectionMMReview}. For these reasons, we mostly give full details of proof of Theorem~\ref{TheoremRelFSGammaHyp}. For occasional steps of proof where precise outlines are practically the same, we will refer to reader to earlier works with just a comment or a sketch.

Theorem~\ref{TheoremRelFSGammaHyp} is proved by applying a hyperbolicity theorem of Masur and Minsky \cite[Theorem 2.3]{MasurMinsky:complex1}. For that application one must supply a \emph{coarsely transitive} family of paths in $\FS(\Gamma;\A)$, and for each path in that family one must supply a \emph{coarse projection function} from $\FS(\Gamma;\A)$ to the given path. One then verifies that these given objects satisfy the three \emph{Masur--Minsky axioms}, and one may then conclude that the given complex is Gromov hyperbolic with respect to the simplicial metric on~$\FS(\Gamma;\A)$.

The path families used here, and in \FSOne, \cite{MasurMinsky:complex1} and \cite{Horbez:HyperbolicGraphs}, are based on Stallings fold paths. In \FSOne\ we used a subset of fold paths in $\FS(F_n)$, namely those satisfying a special ``gate 3 condition'' (see the \emph{Remark on the gate 3 condition} in Section~\ref{SectionFoldSequences}). In~\cite{BestvinaFeighn:subfactor}, Bestvina and Feighn needed to work with optimal paths in outer space, and so they dropped the gate 3 condition and reconfigured the methods of \FSOne\ to work using arbitrary fold paths. Here we take up those reconfigurations, using the family of all fold paths to prove Theorem~\ref{TheoremRelFSGammaHyp}. This requires numerous small changes from the methods of \FSOne; these are commented on throughout the paper, particularly in the narrative of Sections~\ref{SectionFoldPathUnits} and~\ref{SectionFFRelAHyp}. This also results in some improvements and efficiencies, for example free splitting units are more easily defined here than in \FSOne, and hence more easily applicable. In \cite{Horbez:HyperbolicGraphs}, in the general context of $\FS(\Gamma;\A)$, the family of all fold paths is also used.

The coarse projection functions used in this work are described in Definition~\ref{DefProjDiagram}, expressed in terms of a class of commutative diagrams called \emph{projection diagrams}; see the Overview \cite{HandelMosher:RelHypComplexIntro} for an exposition.  Given a fold path in $\FS(\Gamma;\A)$ having the form $S_I \mapsto \cdots \mapsto S_K$, and given a free splitting $T$, a projection diagram from $T$ to the given fold path is a certain commutative diagram which incorporates both the given fold path that starts at $S_I$ and ends at $S_K$, as well as another fold path that starts somewhere near $S_I$ and ends at $T$, and in which there is a certain position $S_J$ (with $J \in \{I,\ldots,K\}$) where the two fold paths appear to diverge. As one varies over all such diagrams, the projection $\pi(T) \in \{I,\ldots,K\}$ is defined to be the maximum value of~$J$. Projection functions used in \FSOne\ and \cite{BestvinaFeighn:subfactor} for the case of $\FS(F_n)$ are described in pretty much the same fashion. For the context of $\FS(\Gamma;\A)$ discussed in \cite{Horbez:HyperbolicGraphs}, after first generalizing to $\FS(\Gamma;\A)$ certain distance bounds that were used in \cite{BestvinaFeighn:subfactor} for $\FS(F_n)$, readers are then directed to follow \cite{BestvinaFeighn:subfactor} in order to generalize projection functions from $\FS(F_n)$ to $\FS(\Gamma;\A)$ and to verify the Masur--Minsky axioms.

The Masur--Minsky axioms are verified in Section~\ref{SectionAxiomReduction}, reducing them to Proposition~\ref{PropMMTranslation} regarding relations between the combinatorial and geometric behavior of projection diagrams. Proposition~\ref{PropMMTranslation}, the technical heart of this work, is needed also in the proofs of Theorem~\ref{TheoremRelFSUParams} and Corollary \ref{CorollaryCompFSU} in Section~\ref{SectionFSUParameterization}. The proof of Proposition~\ref{PropMMTranslation} in Section~\ref{SectionBigDiagrams} is expressed using ``big commutative diagrams'' to compare projections from distinct free splittings to the same fold sequence. This~``big diagram'' argument has some technical differences in comparison with the analogous arguments in the two earlier works \FSOne\ and \cite{BestvinaFeighn:subfactor}; these differences are due to dropping the gate 3 condition and to explicit discussion of free splitting units.

\vfill\break

\setcounter{tocdepth}{2}
\tableofcontents

\vfill\break

\section{Free factor systems} 
\label{SectionBasicFFDefs}

Throughout this paper $\Gamma$ represents an arbitrary freely decomposable group, meaning that $\Gamma$ can be expressed as a nontrivial free product of nontrivial groups (if $\Gamma$ were freely indecomposable then the main objects of study of this paper---relative free factor and free splitting complexes---would be empty). For example this convention rules out the possibility that $\Gamma$ is infinite cyclic; see remarks after the definition of free factor systems in Section~\ref{SectionFFSystems} and after the definition of free splittings in Section~\ref{SectionFSSucc}. 

This section contains basic material regarding the set of free factor systems of $\Gamma$, its partial order $\sqsubset$ known as ``nesting'' or as ``extension'', and its binary operation $\meet$ known as ``meet''. Our focus is on basic applications of the Kurosh Subgroup Theorem. One such application is the Extension Lemma~\ref{LemmaExtension}, regarding the structure of a nested pair of free factor systems $\A \sqsubset \B$. The Extension Lemma and its consequences will be used throughout the rest of the paper. For those interested in the case $\Gamma = F_n$, the contents of this section are mostly well known and/or evident, and need only be skimmed.

An important application of the Extension Lemma is Lemma~\ref{LemmaFFSNorm} which describes a formula for the \emph{depth} of a free factor system with respect to the partial ordering $\sqsubset$, together with various properties of depth (the depth of an element of a partially ordered set is the length of the longest ascending chain starting with the given element). Depth of free factor systems will be applied in several ways, including in a dimension formula for relative free factor complexes (see Proposition~\ref{PropTempDimFF}) and in the construction of a Lipschitz projection $\FS(\Gamma;\A) \mapsto \CFFS(\Gamma;\A)$ (see Section~\ref{SectionFFConnected}). Of more central importance, in Section~\ref{SectionSubgraphComplexity} bounds on depth will be used to derive topological and metric properties of free splittings and their fold paths, and in Section~\ref{SectionFSU} these bounds are translated into properties of free splitting units along fold paths.

\subsection{Free factorizations and free factor systems}
\label{SectionFFSystems}

\newcommand\Id{\text{Id}}

\paragraph{Free factorizations.} Consider any group $\Gamma$ and any indexed set of nontrivial subgroups \hbox{$\H = \{H_l\}_{l \in \L}$}. Recall that for $\H$ to be a \emph{free factorization} of $\Gamma$ means that the following universality property holds: for any group $K$ and any given set of homomorphisms $\{f_l \from H_l \to K \suchthat l \in \L\}$, there exists a unique homomorphism $\Gamma \to K$ extending $f_l$ for each $l \in \L$. Equivalently, every nonidentity element $\gamma \in \Gamma$ can be written as the product of a unique \emph{reduced word over $\H$}, meaning that there exists $M \ge 1$ and a sequence $\gamma_m \in H_{l_m} - \{\Id\}$ indexed by $1 \le m \le M$ such that $l_m \ne l_{m+1}$ for $1 \le m \le M-1$, and $\gamma = \gamma_1 \cdot \ldots \cdot \gamma_M$ (note that $H_l \intersect H_{l'} = \{\Id\}$ if $l \ne l'$). When a free factorization is finite --- which is always true when $\Gamma$ is finitely generated --- we will generally pick a bijection $\L \leftrightarrow \{1,\ldots,L\}$ and write $\Gamma = H_1 * \cdots * H_L$. Meanwhile, as we ponder infinite free factorizations in these early sections of the paper, we shall write $\Gamma = * (H_l)_{l \in \L}$ or just $\Gamma = *\H$. A \emph{free factor} $H \subgroup \Gamma$ is any element of a free factorization; equivalently, $H$ is an element of a two-term free factorization $\Gamma = H * H'$, where $H'$ is obtained by conglomerating the other terms of any free factorization having $H$ as a~term.

For any free factorization $\H$ of $\Gamma$, each conjugacy class in $\Gamma$ is represented by a cyclically reduced word (meaning a reduced word that also satisfies $l_M \ne l_1$) and this representative is unique up to cyclic permutation. This immediately proves the following lemma, which incorporates the well known result that every free factor is malnormal:

\begin{lemma}
Every free factorization $\Gamma = *\H$ is \emph{mutually malnormal}, meaning that for each $H,H' \in \H$ and $\gamma \in \Gamma$, if $\gamma H \gamma^\inv \intersect H'$ is nontrivial then $\gamma \in H = H'$. \qed
\end{lemma}
\noindent
From malnormality of a free factor $H \subgroup \Gamma$ it follows that two subgroups of $H$ are conjugate in $\Gamma$ if and only if they are conjugate in $H$. We shall make tacit use of this equivalence in what follows. 

\newcommand\NC[1]{\<\!\<#1\>\!\>}

A \emph{partial free factorization} of $\Gamma$ is a subset of a free factorization. Every partial free factorization $\H$ has a \emph{cofactor} which is a subgroup $B \subgroup \Gamma$ such that either $\H$ is a free factorization and $B$ is trivial, or $\H \union \{B\}$ is a free factorization; for example, one may take $B$ to be the free product of the complement of $\H$ in any free factorization containing $\H$. Without any assumptions on the cofactor, we have the following result. Let $\NC{\H} \subgroup \Gamma$ be the subgroup normally generated by the union of the subgroups in~$\H$.

\begin{lemma}
\label{LemmaPartialFFCofactor}
For any partial free factorization $\H$ of $\Gamma$ and any cofactor $B$ of $\H$ there exists a short exact sequence $1 \to \NC{\H} \to \Gamma \to B \to 1$ such that the homomorphism $\Gamma \mapsto B$ is a retract, and therefore $B$ is isomorphic to the quotient $\Gamma / \NC{\H}$.
\end{lemma}

\begin{proof} Consider any reduced word $\gamma = \gamma_1 \cdot \ldots \cdot \gamma_l$ over $\H \union \{B\}$. The retraction of $\gamma$ to $B$ is defined by erasing each letter $\gamma_i$ that lies in some element of~$\H$, and multiplying out the surviving letters of $B$ in order. Letting $K$ be the kernel of this retraction, evidently $\NC{\H} \subgroup K$. Conversely, the given word $\gamma$ can be rewritten by moving any letters in $B$ to the front of the word, preserving their order, at the expense of replacing every other letter by a conjugate; so if $w \in K$ then after rewriting one sees that $w \in \NC{\H}$.
\end{proof}

\begin{lemma}
\label{LemmaPartialWeakFFS}
Consider a group $\Gamma$ and two partial free factorizations $\H = \{H_l\}_{l \in \L}$ and $\H' = \{H'_l\}_{l \in \L}$ of $\Gamma$ with the same index set $\L$ and with respective cofactors $B,B'$. If $H_l$ is conjugate to $H'_l$ for all $l \in \L$ then the cofactors $B,B'$ are isomorphic. Furthermore there is an isomorphism $\Gamma \to \Gamma$ which for each $l$ restricts to a conjugation from $H_l$ to $H'_l$ and which restricts to an isomorphism from $B$ to $B'$.
\end{lemma}

\begin{proof} Noting that $\NC{\H} = \NC{\H'}$, apply Lemma~\ref{LemmaPartialFFCofactor} to conclude that each of $B,B'$ is isomorphic to $\Gamma / \NC{\H}$. After choosing conjugations $H_l \mapsto H'_l$ and an isomorphism $B \mapsto B'$, the lemma follows by applying the universality property for free factorizations. 
\end{proof}

\noindent\textbf{Remark.} For an example which determines the extent to which cofactors can fail to be well-defined up to conjugacy, see the discussion of $\Gamma = A * Z$ following Proposition~\ref{PropExceptionalFFS} in which the non-well-definedness of an infinite cyclic cofactor $Z$ is discussed in detail.

\begin{definition}[Free factor systems.] 
\label{DefFFSystems}
A \emph{weak free factor system} of $\Gamma$ is a set of the form $\A = \{[A_l]\}_{l \in \L}$ such that $\{A_l\}_{l \in \L}$ is a partial free factor system of $\Gamma$ having a free cofactor~$B$; in this level of generality we make no assumption on the cardinality of the set $\A$ nor on the rank of the free group $B$, although by Lemma~\ref{LemmaPartialWeakFFS} the rank of $B$ is a well-defined cardinal number. A~\emph{free factor system} of $\Gamma$ is a weak free factor system $\A$ which is finite and has a finite rank cofactor; this allows $\A = \emptyset$ as a possible free factor system, but only if $\Gamma$ is free of finite rank. Also, as usual, the cofactor $B$ may be trivial.  A \emph{realization} of $\A$ is any free factorization of $\Gamma$ having the form $\Gamma = A_1 * \cdots * A_I * B$, where $\A = \{[A_1],\ldots,[A_I]\}$ and $B$ is a cofactor. The individual elements $[A_1],\ldots,[A_I]$ of $\A$ are called its \emph{components}. If $\A = \{[\Gamma]\}$ then we say that $\A$ is full, and otherwise $\A$ is \emph{nonfull}. 
\end{definition}

\noindent
\textbf{Remark.} Recalling our blanket assumption that $\Gamma$ is not infinite cyclic, nevertheless an infinite cyclic group does have a unique nonfull free factor system, namely~$\emptyset$.

\smallskip
\noindent
\textbf{Remark.} In the cases that $\Gamma = F_n$ or $\Gamma$ is finitely generated, Grushko's Theorem combined with the Kurosh Subgroup Theorem implies that every weak free factor system is a free factor system. The reader interested solely in $F_n$ or other finitely generated groups $\Gamma$ may therefore safely ignore the adjective ``weak'', which should cut down on the technical overload of the remaining subsections of Section~\ref{SectionBasicFFDefs}. Also, the Extension Lemma~\ref{LemmaExtension} will provide a relative setting in which we can also ignore ``weak'', which we shall do forever afterwards, once the Extension Lemma is proved.

\subsection{The Kurosh Subgroup Theorem. Extension $\sqsubset$ and meet $\meet$. }
\label{SectionKurosh}

The results obtained in this section by applying the Kurosh Subgroup Theorem are standard in the case $\Gamma = F_n$; see \BookOne. 

The following foundational theorem can be proved using Bass-Serre theory; see for example \cite{ScottWall} and~\cite{Cohen:CombinatorialGroupTheory}. The usual expression of this theorem is in the language of double cosets. We provide a translation into the language of conjugacy of subgroups, as well as a slightly more detailed conclusion, particularly in the case of a free factor~$A \subgroup \Gamma$.

\begin{KuroshTheorem}
For any group $\Gamma$, any free factorization $\Gamma = *(H_l)_{l \in \L}$, and any subgroup $A \subgroup \Gamma$, there exists for each $l \in \L$ a subset $U_l \subset \Gamma$ consisting of representatives $u$ of distinct double cosets $AuH_l$, and there exists a free subgroup $C \subgroup A$, such that the following hold:
\begin{enumerate}
\item \label{ItemKuroshFF}
$\displaystyle A = *\{A \intersect u H_l u^\inv \suchthat l \in \L, \, u \in U_l\} \, * \, C$
\item \label{ItemKuroshTriviality}
For each $(l,v) \in \L \cross \Gamma$:
\begin{enumerate}
\item  \label{ItemKuroshTrivial} The subgroup $C \intersect v H_l v^\inv$ is trivial.
\item \label{ItemKuroshNontrivial}
The subgroup $A \intersect v H_l v^\inv$ is nontrivial $\iff$ there exists $u \in U_l$ such that the subgroups $A \intersect v H_l v^\inv$ and $A \intersect u H_l u^\inv$ are conjugate in~$A$ $\iff$ there exists $u \in U_l$ such that $AvH_l=AuH_l$.
\end{enumerate}
\end{enumerate}
If furthermore $A$ is itself a free factor then:
\begin{enumeratecontinue}
\item \label{ItemKuroshMalnormal}
For each $(l,u),(m,v) \in \L \cross \Gamma$ such that $u \in U_l$ and $v \in U_m$, if the subgroups $A \intersect u H_l u^\inv$ and $A \intersect v H_{m} v^\inv$ are conjugate in $\Gamma$ then $l=m$ and $u=v$ (and so in particular those subgroups are equal).
\end{enumeratecontinue}
\end{KuroshTheorem}

\subparagraph{Remarks.} The statement of the Kurosh Subgroup Theorem found for example in \cite{Cohen:CombinatorialGroupTheory} incorporates only item~\pref{ItemKuroshFF}, but the others are easily proved. Item~\pref{ItemKuroshTrivial} is easily derived from the Bass-Serre theory proof found in \cite{Cohen:CombinatorialGroupTheory}, as is the first equivalence of item~\pref{ItemKuroshNontrivial}. The second equivalence of~\pref{ItemKuroshNontrivial} is a calculation: \, if $AvH_l=AuH_l$ then $u=avh$ for some $a \in A$, $h \in H_l$ and so $a(A \intersect v H_l v^\inv)a^\inv = A \intersect u H_l u^\inv$; \, conversely if $a(A \intersect v H_l v^\inv)a^\inv = A \intersect u H_l u^\inv$ for $a \in A$ then $A \intersect (av) H_l (av)^\inv = A \intersect u H_l u^\inv$ and so, by malnormality of $H_l$, we have $u^\inv av \in H_l$ implying that $AvH_l=AuH_l$. For proving item~\pref{ItemKuroshMalnormal}, the conjugating element must be in $A$ by malnormality of $A$, and $l=m$ by mutual malnormality of $*\{H_l\}_{l \in \L}$; the rest follows from~\pref{ItemKuroshNontrivial}.

\bigskip


One standard consequence of the Kurosh Subgroup Theorem is that for any partial free factorization $\{A_i\}$ of $\Gamma$ and any free factor $A' \subgroup \Gamma$, if each $A_i$ is a subgroup of $A'$ then $\{A_i\}$ is a partial free factorization of~$A'$. The following slight generalization, also an immediate consequence of the Kurosh Subgroup Theorem, is needed for the proof of the Extension Lemma~\ref{LemmaExtension}.

%

\begin{lemma}\label{LemmaKuroshConsequence}
For any group $\Gamma$, any free factor $A' \subgroup \Gamma$, and any partial free factorization $\{A_i\}_{i \in I}$ of~$\Gamma$, if $A_i$ is conjugate in $\Gamma$ to a subgroup of $A'$ then there exists an identically indexed set of subgroups $\{A'_i\}_{i \in I}$ of $A'$, such that $A'_i$ is conjugate in $\Gamma$ to $A_i$ (for each $ i \in I$) and such that $\{A'_i\}_{i \in I}$ is a partial free factorization of $A'$. \qed
\end{lemma}


\paragraph{The extension partial order $\sqsubset$ on weak free factor systems.} Given two subgroups $A,A' \subset \Gamma$ with conjugacy classes $[A],[A']$, let $[A] \sqsubset [A']$ denote the well-defined relation that $A$ is conjugate to a subgroup of $A'$. Define a partial ordering $\A \sqsubset \A'$ on weak free factor systems systems by requiring that for each $[A] \in \A$ there exists $[A'] \in \A'$ such that $[A] \sqsubset [A']$. The fact that this is a partial order follows from item~\pref{ItemKuroshMalnormal} of the Kurosh Subgroup Theorem, which tells us that for any sequence of free factors $A'' \subgroup A \subgroup A'$, if $A''$ and $A'$ are conjugate then $A'' = A = A'$. We express the relation \emph{$\text{(this)} \sqsubset \text{(that)}$} in various ways: \emph{(this) is contained in (that)}; or \emph{(that) is an extension of (this)}; or \emph{$\text{(this)} \sqsubset \text{(that)}$ is an extension}; etc. An extension $\A \sqsubset \A'$ such that $\A \ne \A'$ is called a \emph{proper extension}. 

If $\A,\B$ are free factor systems then we also express the relation $\A \sqsubset \B$ by saying that \emph{$\A$ is nested in $\B$} and that \emph{$\B$ is a free factor system rel~$\A$}.

\paragraph{Meet of free factor systems.} The \emph{meet} $\meet$ is a binary operation on weak free factor systems defined by 
$$\A \meet \B = \{[A \intersect B] \, \,  \suchthat \, \,   [A] \in \A,  \,  \, [B] \in \B,  \,  \, A \intersect B \ne \{\Id\}\}
$$
\noindent
The following, generalizing \cite[Lemma 2.6.2]{\BookOneTag}, will be proved using the Kurosh Subgroup Theorem:

\begin{lemma}[Weak Meet Lemma]
\label{LemmaWeakMeet} In any group $\Gamma$, the meet of any two weak free factor systems is a weak free factor system. 
\end{lemma}
\noindent
We will need to strengthen the conclusion of this lemma by removing the word ``weak'' in various situations. One such situation, for finitely generated groups, is described in Corollary~\ref{CorollaryMeet}. Another ``relativized'' version is given in Proposition~\ref{PropMeet}. 

Before giving the proof of Lemma~\ref{LemmaWeakMeet}, here are two immediate corollaries.

\begin{corollary} 
\label{CorollaryWeakMeetProps}
For any weak free factor systems $\A,\B$ in any group $\Gamma$, their meet $\A \meet \B$ can be characterized as the unique weak free factor system having the following properties: 
\begin{itemize}
\item [(i)] $\A \meet \B \sqsubset \A$; 
\item[(ii)] $\A \meet \B \sqsubset \B$; 
\item[(iii)] For every weak free factor system $\C$, if $\C \sqsubset \A$ and $\C \sqsubset \B$ then $\C \sqsubset \A \meet \B$.  \qed
\end{itemize}
\end{corollary}
\noindent
The next result, well known in the case of free groups from \BookOne, follows immediately by combining Lemma~\ref{LemmaWeakMeet},  Grushko's Theorem, and Corollary~\ref{CorollaryWeakMeetProps}. 

\begin{corollary} 
\label{CorollaryMeet}
In any finitely generated group $\Gamma$, for any two free factor systems $\A,\B$ of $\Gamma$, their meet $\A \meet \B$ is a free factor system. Furthermore if $\A,\B$ are free factor systems relative to a third free factor system $\C$ then $\A \meet \B$ is also a free factor system relative to~$\C$.
\qed
\end{corollary}
\noindent
The second sentence of Corollary~\ref{CorollaryMeet} is true in a general group; see Proposition~\ref{PropMeet}.

\begin{proof}[Proof of the Weak Meet Lemma \ref{LemmaWeakMeet}]
Consider $\A = \{[A_i]\}_{i \in I}$ and $\B=\{[B_j]\}_{j \in J}$ with respective realizations
\begin{align*} 
(\#) \qquad \Gamma &= *(A_i)_{i \in I} * A' \quad\text{and}\quad \Gamma = *(B_j)_{j \in J} * B' \\
\intertext{Applying the Kurosh Subgroup theorem to $A_i$ using the given realization of $\B$, we obtain a free factorization 
}
(\#\#) \qquad A_i &= *(A_{ik})_{k \in K_i} * A'_i
\end{align*}
where $A'_i$ is a free group and the subgroups $A_{ik}$ are representatives of the $\Gamma$-conjugacy classes of all nontrivial intersections of $A_i$ with conjugates of the $B_j$'s. It follows that  
$$\A \meet \B = \{[A_{ik}] \suchthat i \in I, k \in K_i\}
$$
Substituting $(\#\#)$ into $(\#)$ we obtain a free factorization
\begin{align*}
\Gamma &=  *\biggl( *\{A_{ik}\}_{k \in K_i} * A'_i \biggr)_{i \in I} * A' 
\\
&= \biggl( *\{A_{ik}\}_{i \in I, k \in K_i} \biggr) * \biggl[\bigl( * (A'_i)_{i \in I} \bigl) * A' \biggr]
\end{align*}
which, the factor in brackets $[\cdot]$ clearly being free, shows that $\A \meet \B$ is a weak free factor system.
\end{proof}

\subsection{Corank and the structure of extensions of free factor systems}
\label{SectionSqsubsetSuccProps}
In this section we prove the Extension Lemma \ref{LemmaExtension} detailing the structure of any extension $\A \sqsubset \B$ of free factor systems of a group~$\Gamma$. This will be applied in studying the depth of $\sqsubset$ in Section~\ref{SectionDFF}, and when studying free splitting units in Sections~\ref{SectionSubgraphComplexity} and~\ref{SectionFSU}. 

The Extension Lemma~\ref{LemmaExtension} will guarantee that any weak free factor system that is an extension of a free factor system is itself a free factor system. In general, for any extension $\A \sqsubset \B$ of free factor systems we shall say that $\B$ is a \emph{free factor system rel~$\A$}.  In Proposition~\ref{PropMeet}, we will fix a base free factor system~$\A$ and prove that the meet of any two free factor systems rel~$\A$ is also a factor system rel~$\A$, which is how we generalize Corollary~\ref{CorollaryMeet} to non finitely generated groups. These results allow us henceforth to ignore the adjective ``weak'', as long as a base free factor system~$\A$ is specified.

\paragraph{Corank.} Define the \emph{corank} of a free factor system $\A$ of a group $\Gamma$ to be the integer 
$$\corank(\A) = \rank(\Gamma / N(\A)) = \rank(A') \ge 0
$$
where $A'$ is the cofactor of any realization of $\A$. When we wish to emphasize the ambient group~$\Gamma$ we also write $\corank(\Gamma;\A)$. From Bass-Serre theory it follows that $\corank(\A)$ is equal to the topological rank of the underlying graph for any finite graph of groups representation of $\Gamma$ with trivial edge groups and with nontrivial vertex groups $A_1,\ldots,A_I$ so that $\A = \{[A_1],\ldots,[A_I]\}$.

When $\Gamma = F_n$ and $\A = \{[A_1],\ldots,[A_I]\}$, the free factors $A_1,\ldots,A_I$ are all free of finite rank, and we have the following \emph{rank sum formula} for the corank of $\A$:
$$\corank(\A) = n - \sum_{i=1}^I \rank(A_i)
$$
This formula may be useful to the reader for deriving quick proofs of results to follow in the special case $\Gamma=F_n$.

\paragraph{Notation for constants.} Given a group $\Gamma$ and a free factor system~$\A$ of $\Gamma$, the two most important numerical invariants are $\corank(\Gamma;\A)$ and the cardinality $\abs{\A}$. In this work we will encounter several other numerical invariants of $\Gamma$ and $\A$ which depend solely on $\corank(\Gamma;\A)$ and~$\abs{\A}$, and for such a constant we will use notation like $C = C(\Gamma;\A)$. For example, in \cite{CollinsTurner:efficient} the sum $\text{KR}(\Gamma;\A) = \corank(\Gamma;\A) + \abs{\A}$ is called the \emph{Kurosh rank}. 

\medskip

The following lemma defines what we shall call the \emph{containment function} from one free factor system to any of its extensions.

\begin{lemma} Given an extension $\A \sqsubset \B$ of weak free factor systems of a group $\Gamma$, the relation $\sqsubset$ between components of $\A$ and components of $\B$ defines a function $A \mapsto \B$, called the \emph{containment function}.
\end{lemma}

\begin{proof} By definition, for any component $[A] \in \A$ there exists a component $[B] \in \B$ such that $[A] \sqsubset [B]$. By mutual malnormality of any realization of~$\B$, this $[B]$ depends uniquely on $[A]$.
\end{proof}

The following result, in the special case $\Gamma = F_n$, is an evident consequence of the rank sum formula for corank.

\begin{proposition}\label{PropCorankIneq} 
For any nested pair of free factor systems $\A \sqsubset \A'$ of $\Gamma$ we have $\corank(\A) \ge \corank(\A')$. Equality holds if and only if the containment function $\A \mapsto \A'$ is surjective and
for each $[A'_j] \in \A'$ there exists a free factorization with trivial cofactor $A'_j = A_{j1} * \cdots * A_{jk_j}$ so that the preimage of $[A'_j]$ under the containment function is \hbox{$\{[A_{j1}],\ldots,[A_{jk_j}]\} \subset \A$}.
\end{proposition}
\noindent
The proof of Proposition~\ref{PropCorankIneq} in the general case---where rank sum does not make sense---will be given after the statement and proof of the following Extension Lemma. 

For understanding the conclusions of the Extension Lemma we refer the reader to Figure~\ref{FigContainmentFunction} which depicts those conclusions in tabular format. The proof of the Extension Lemma is similar to the proof of the Weak Meet Lemma~\ref{LemmaWeakMeet} but with more care taken regarding cardinalities.


\begin{lemma}[Extension Lemma]
\label{LemmaExtension} Consider a group $\Gamma$ and a free factor system $\A$. If $\A'$ is a weak free factor system such that $\A \sqsubset \A'$, then $\A'$ is a free factor system. Moreover, consider any realization $\Gamma = A'_1 * \cdots * A'_K* B'$ of $\A' = \{[A'_1],\ldots,[A'_K]\}$, with indexing chosen so that the image of the containment function $\A \mapsto \A'$ equals $\{[A'_1],\ldots,[A'_J]\}$, where $0 \le J \le K$.
For $1 \le j \le J$ let $\A_j \subset \A$ be the pre-image of $[A'_j]$ under the containment function, and let $k_j = \abs{\A_j}$. Then there exists a realization of $\A$ of the form
$$\Gamma = A_{11} * \cdots * A_{1k_1} * \cdots\cdots * A_{J1} * \cdots * A_{Jk_J} * \underbrace{(B_1 * \cdots * B_J * A'_{J+1} * \cdots * A'_K * B')}_{B \, = \, \text{cofactor of $\A$}}
$$
such that
$$\A_j = \{[A_{j1}],\ldots,[A_{jk_j}]\} \quad\text{and}\quad A'_j = A_{j1} * \cdots * A_{jk_j} * B_j \qquad \text{($1 \le j \le J$)}
$$
The subgroups $B_1,\ldots,B_J,A'_{J+1},\ldots,A'_K,B'$ are all free of finite rank. By abuse of notation (identifying conjugacy classes in $\Gamma$ with conjugacy classes in $A'_j$) we may regard $\A_j$ as a free factor system of the group $A'_j$ realized with cofactor $B_j$. 
\end{lemma}

\begin{proof} Since $\A$ is finite and $\A \sqsubset \A'$, any realization of the weak free factor system $\A'$ can be listed as
$$(*) \qquad \Gamma = A'_1 * \cdots * A'_J * (*\{A'_k\}_{k \in \K}) * B', \quad J \ge 0
$$
so that $\A' = \{[A'_j] \suchthat 1 \le j \le J\} \union \{\abs{A'_k} \suchthat k \in \K\}$, and $B'$ is a cofactor of $\A'$, and the subset $\{[A'_1],\ldots,[A'_J]\} \subset \A'$ is the image of the containment map $\A \mapsto \A'$ (we assume all free factors of $(*)$ are nontrivial, except perhaps $B'$). For $1 \le j \le J$, let $\A_j \subset \A$ be the preimage of $[A'_j]$ under the containment map $\A \mapsto \A'$, and let $k_j = \abs{\A_j} \ge 1$. By Lemma~\ref{LemmaKuroshConsequence}  we may choose subgroups $A_{j1},\ldots,A_{jk_j} \subgroup A'_j$ so that $\A_j = \{[A_{j1}],\ldots,[A_{jk_j}]\}$ and so that we have a free factorization  
$$(**)_j \qquad A'_j = A_{j1} * \cdots * A_{jk_j} * B_j
$$
Substituting each $(**)_j$ into $(*)$ and rearranging terms we obtain the following free factorization of $\Gamma$, which is clearly a realization of $\A$:
$$\Gamma = A_{11} * \cdots * A_{1k_1} * \cdots\cdots * A_{J1} * \cdots * A_{Jk_J} * \underbrace{(B_1 * \cdots * B_J * (*\{A'_k\}_{k \in \K}) * B')}_{B \, = \, \text{cofactor of $\A$}}
$$
Since $B$ is a finite rank free group, it follows that $\K$ is finite, and that the subgroups $A'_k$ for $k \in \K$ and $B_1,\ldots,B_J,B'$ are all finite rank and free. It follows that $\A'$ is a free factor system of $\Gamma$ with cofactor $B'$, and that $\A_j$ may be regarded as a free factor system of $A'_j$ with cofactor~$B_j$.
\end{proof}

\newcommand\tablecolor[1]{\textcolor{red}{#1}}

\begin{figure}
\begin{center}
\renewcommand\arraystretch{1.2}
\begin{tabular}{| c || c c c | c | c |} \hline
$j$ & \multicolumn{4}{c|}{free factorization of $A'_j$} & cofactor of $\A'$ \\ \hline\hline
$1$ & \tablecolor{$A_{11}$} & $\cdots$ & \tablecolor{$A_{1 k_1}$} & $B_1$ & \\ 
\vdots & \multicolumn{3}{c|}{\vdots} & \vdots & \\ 
$J$ & \tablecolor{$A_{1J}$} & $\cdots$ & \tablecolor{$A_{1 k_J}$} & $B_J$ & $B'$  \\ \cline{1-5} 
$J+1$ & \multicolumn{4}{c|}{$A'_{J+1}$} &  \\ 
\vdots & \multicolumn{4}{c|}{\vdots} &  \\ 
$K$ & \multicolumn{4}{c|}{$A'_{K}$} &  \\ \hline
\end{tabular} 
\end{center}
\caption{The Extension Lemma~\ref{LemmaExtension} shows that for each extension $\tablecolor{\A} \sqsubset \A'$ of free factor systems of~$\Gamma$, and for any realization of $\A'$, there exists a free factorization with terms as depicted which simultaneously incorporates the following: 
the given realization of $\A' = \{[A'_1],\ldots,[A'_K]\}$ with cofactor $B'$; 
for each $j \le J$ a realization of a free factor system of $A'_j$, namely $\A'_j = \{[\tablecolor{A_{j1}}],\ldots,[\tablecolor{A_{jk_j}}]\}$, with cofactor~$B_j$;
and a realization of $\tablecolor{\A} = \{[\tablecolor{A_{11}}],\ldots,[\tablecolor{A_{1k_1}}],\ldots\ldots,[\tablecolor{A_{1J}}],\ldots,[\tablecolor{A_{1k_J}}]\}$ with cofactor $B = B_1 * \cdots * B_J * A'_{J+1} * \cdots * A'_K * B'$. In fact these conclusions can be reformulated even when $\A'$ is only a \emph{weak} free factor system, but one then deduces that $\A'$ is actually a (strong) free factor system.
}
\label{FigContainmentFunction}
\end{figure}

\begin{proof}[Proof of Proposition~\ref{PropCorankIneq}]
This is a quick application of Lemma~\ref{LemmaExtension}. Following the notation of that lemma we have $\corank(\A) = \rank(B) \ge \rank(B') = \corank(\A')$, with equality if and only if and only if none of $B_1,\ldots, B_J, A'_{J+1},\ldots, A'_K$ exist: nonexistence of $A'_{J+1},\ldots,A'_K$ is equivalent to $J=K$ which is equivalent to surjectivity of $\A \mapsto \A'$; and nonexistence of the cofactor $B_j$ is equivalent to existence of the desired free factorization \break $A'_j = A_{j1} * \cdots * A_{jk_j}$ without cofactor.
\end{proof}

Here is the promised relativization of Corollary~\ref{CorollaryMeet}.

\begin{proposition}
\label{PropMeet}
For any group $\Gamma$, any free factor system $\A$, and any two free factor systems $\B,\C$ of $\Gamma$ rel~$\A$, their meet $\B \meet \C$ is a free factor system rel~$\A$. 
\end{proposition}

\begin{proof} Applying Corollary~\ref{CorollaryWeakMeetProps}, $\B \meet \C$ is a weak free factor system, and by item (iii) of that corollary we have $\A \sqsubset \B \meet \C$. By Lemma~\ref{LemmaExtension} it follows that $\B \meet \C$ is a free factor system.
\end{proof}

\subsection{Grushko free factor systems}

Recall Grushko's theorem, which says every finitely generated group has a Grushko decomposition (see Section~\ref{SectionIntro}). Grushko decompositions can also exist naturally outside of the realm of finitely generated groups: any free product of a finite rank free group and finitely many freely indecomposable, non-cyclic groups yields a Grushko decomposition. 

The following proposition describes uniqueness properties of Grushko decompositions, expressed in terms of the $\sqsubset$ relation, thus allowing us to introduce the concept of a Grushko free factor system. Starting in Part~II \cite[Section 2.2.5]{\STLOneTag}, motivated by this proposition we will be extending the terminology to a concept of ``relative'' Grushko free factor systems.

\begin{proposition} 
\label{PropGrushkoMinimum}
For any group $\Gamma$ and any free factor system $\A$ of $\Gamma$, the following are equivalent: 
\begin{enumerate}
\item\label{PropGrushkoRealize}
Some realization $\Gamma = A_1 * \cdots * A_K * A'$ of $\A$ is a Grushko decomposition.
\item\label{PropGrushkoRealizeAny}
Any realization $\Gamma = A_1 * \cdots * A_K * A'$ of $\A$ is a Grushko decomposition.
\item\label{PropGrushkoMin}
$\A$ is a minimum weak free factor system with respect to $\sqsubset$. 
\item\label{PropGrushkoUnique}
For any weak free factor system $\B$ of $\Gamma$ we have $\A \sqsubset \B$. In particular $\A$ is the unique minimum weak free factor system with respect to $\sqsubset$.
\end{enumerate}
\end{proposition}

\noindent
If these properties hold then we say that $\A$ is the \emph{Grushko free factor system} of~$\Gamma$.

\begin{proof} Clearly \pref{PropGrushkoUnique}$\implies$\pref{PropGrushkoMin}$\implies$\pref{PropGrushkoRealizeAny}$\implies$\pref{PropGrushkoRealize}. Assuming \pref{PropGrushkoRealize}, in order to prove~\pref{PropGrushkoUnique} it suffices by Corollary~\ref{CorollaryWeakMeetProps} to prove that $\A = \A \meet \B$.  Let $\C = \A \meet \B \sqsubset \A$. For each $[C] \in \C$, consider the unique component $[A_k] \in \A$ such that $[C] \sqsubset [A_k]$. Applying the Kurosh Subgroup Theorem, after conjugation it follows that $C$ is a nontrivial free factor of $A_k$, but $A_k$ is freely indecomposable, and so $C=A_k$. This proves that $\C$ is a subset of $\A$. If $\C \ne \A$ then there exists $[A_k] \in \A$ such that $[A_k] \not\in \C$, and by the Extension Lemma~\ref{LemmaExtension} it follows that $[A_k]$ is a free factor of some cofactor of $\C$. But cofactors are free and $A_k$ is not free, a contradiction.
\end{proof}

\subsection{Free factor system depth.}
\label{SectionDFF}
In general the \emph{depth} of an element $x$ of a partially ordered set is the cardinality $L$ of the longest ascending chain $x=x_0 \sqsubset \cdots \sqsubset x_L$ of order relations starting with the given element. Given a group $\Gamma$ we compute depth for the set of free factor systems of $\Gamma$ with respect to the partial ordering~$\sqsubset$, and we derive some properties of this depth. These could be immediately applied to define and compute depths of complexes of free factor systems relative to a free factor system, but we shall delay that until Section~\ref{SectionFFCHyp}.

Given a free factor system $\A = \{[A_1],\ldots,[A_I]\}$ of $\Gamma$ define the \emph{free factor system depth} of $\A$ to be
$$\DFF(\A) = 2 \, \corank(\A) + \abs{\A} - 1 = 2 \, \rank(B) + I - 1
$$
where $\abs{\cdot}$ denotes the cardinality, and $B$ is any cofactor of any realization of~$\A$. 

Assuming $\Gamma = F_n$, for any free factor system $\A = \{[A_1],\ldots,[A_I]\}$ we have
$$\DFF(\A) = 2 \bigl(n - \sum_{i=1}^I \rank(A_i)\bigr) + I - 1 = (2n-1) -  \sum_{i=1}^I \bigl( 2 \rank(A_i) - 1 \bigr)
$$
Part of the content of Lemma~\ref{LemmaFFSNorm} below is that $\DFF(\A)$ is indeed the depth of $\A$ with respect to the partial ordering~$\sqsubset$. This is easily checked when $\Gamma=F_n$.

Here are some examples. The \emph{exceptional} free factor systems~$\A$, defined to be those for which $\DFF(\A) \le 2$, can be enumerated as follows:
\begin{itemize}
\item $\DFF(\A)=0$ if and only if $\A$ is the full free factor system $\A=\{[\Gamma]\}$.
\item $\DFF(\A)=1$ if and only if $\corank(\A)=0$ and $\abs{\A}=2$, in which case $\A = \{[A_1],[A_2]\}$ with $\Gamma = A_1 * A_2$. The possibility that $\corank(\A)=1$ and $\abs{\A}=0$ is equivalent to $\Gamma$ being infinite cyclic, which was ruled out.
\item $\DFF(\A) = 2$ if and only if one of the following happens: either $\abs{\A}=1$ and $\corank(\A)=1$, in which case $\A = \{[A]\}$ with realization $\Gamma = A * Z$ where the cofactor $Z$ is infinite cyclic; or $\abs{\A}=3$ and $\corank(\A)=0$ in which case $\A = \{[A_1],[A_2],[A_3]\}$ with realization $\Gamma = A_1 * A_2 * A_3$.
\end{itemize}
As we shall see in Proposition~\ref{PropExceptionalFFS}, the exceptional free factor systems $\A$ are characterized as those for which the complex of free factor systems rel~$\A$ is exceptionally simple, either empty or $0$-dimensional. 


We say that a proper extension $\A \sqsubset \A'$ is \emph{elementary} if one of the following holds:
\begin{enumerate}
\item $\A' = \A \union \{[Z]\}$ where $Z \subgroup \Gamma$ is infinite cyclic; or
\item there exists a realization $\Gamma = A_1 * \cdots * A_I * B$ of $\A$ and two indices $i \ne j \in \{1,\ldots,I\}$ such that
$$\A' = \bigl(\A - \{[A_i],[A_j]\} \bigr) \union \{[A_i * A_j]\}
$$
\end{enumerate}
Another part of Lemma~\ref{LemmaFFSNorm} is that the statement ``$\A \sqsubset \A'$ is elementary'' is equivalent to $\DFF(\A) = \DFF(\A')+1$ which is equivalent to saying that no other free factor system is properly contained between $\A$ and $\A'$. Again this is easily checked when $\Gamma=F_n$.


\begin{lemma}\label{LemmaFFSNorm}
The function $\DFF$ on free factor systems of $\Gamma$ has the following properties:
\begin{enumerate}
\item\label{ItemFFSxIneq}
If $\A \sqsubset \A'$ then $\DFF(\A) \ge \DFF(\A')$ with equality if and only if $\A=\A'$. As a special case, $\DFF(\A) \ge 0$ with equality if and only if $\A=\{[\Gamma]\}$.
\item\label{ItemFFSxElem}
If $\A \sqsubset \A'$ is a proper extension then $\DFF(\A) \ge \DFF(\A') + 1$ with equality if and only if $\A \sqsubset \A'$ is an elementary extension.
\item\label{ItemFFSxIncr}
For any proper extension $\A \sqsubset \A'$ there exists a free factor system $\C$ such that $\A \sqsubset \C \sqsubset \A'$ and such that $\A \sqsubset \C$ is elementary.
\item\label{ItemFFSxChain}
For every chain of proper extensions of the form $\A = \A_0 \sqsubset \cdots \sqsubset \A_K = \{[\Gamma]\}$, its length $K$ satisfies $K \le \DFF(\A)$. Equality holds if only if the chain is maximal, if and only if every extension $\A_{k-1} \sqsubset \A_k$ is an elementary extension.
\end{enumerate}
\end{lemma}

\begin{proof}  Noting that item~\pref{ItemFFSxChain} is a consequence of the earlier items, it remains to prove \pref{ItemFFSxIneq}, \pref{ItemFFSxElem} and~\pref{ItemFFSxIncr}. Assuming $\A \sqsubset \A'$, in items~\pref{ItemFFSxIneq} and~\pref{ItemFFSxElem} we are interested in the difference
$$\DFF(\A)-\DFF(\A') = 2(\corank(\A) - \corank(\A')) + \abs{\A} - \abs{\A'}
$$
Applying Lemma~\ref{LemmaExtension} and adopting its notation, we have
\begin{align*}
\corank(\A) &= \sum_{j=1}^J \rank(B_j) + \sum_{j=J+1}^K \rank(A'_j) + \overbrace{\rank(B')}^{\corank(A')} \\
\corank(A) - \corank(A') &= \sum_{j=1}^J \rank(B_j) + \sum_{j=J+1}^K \rank(A'_j) \\
\abs{\A} - \abs{\A'} &= \sum_{j=1}^J \abs{\A_j} \, - \, K \\
                               &= \sum_{j=1}^J (\abs{\A_j}-1) \, - \, (K-J) \\
\DFF(\A)-\DFF(\A') &= 
          \sum_{j=1}^J \underbrace{2 \rank(B_j)}_{(a)_j}  \,\, + \,\,  \sum_{j=1}^J (\underbrace{\abs{\A_j}-1}_{(b)_j})   \,\, +    \sum_{j=J+1}^K (\underbrace{2\rank(A'_j) - 1}_{(c)_j})                     
\end{align*}
From this it follows that $\DFF(\A) \ge \DFF(\A')$ because each of the quantities $(a)_j$, $(b)_j$, $(c)_j$ is non-negative: for $1 \le j \le J$ the quantity $(a)_j$ is a non-negative even integer, and the quantity $(b)_j$ is a non-negative integer because $\A_j \ne \emptyset$; and for $J+1 \le j \le K$ the quantity $(c)_j$ is an odd positive integer because $A'_j$ is free of rank~$\ge 1$. 
Furthermore: 
\begin{itemize}
\item $(a)_j=0$ if and only if the free factorization $A'_j = A_{j1} * \cdots * A_{j k_j} * B_j$ has trivial cofactor $B_j$ (for $1 \le j \le J$).
\item $(b)_j=0$ if and only if $\abs{\A_j} = k_j = 1$ if and only if $\A_j$ has exactly one component (for $1 \le j \le J$).
\item $(c)_j > 0$ (for $J+1 \le j \le K$). 
\end{itemize}
Thus $\DFF(\A)=\DFF(\A')$ if and only if no $(c)_j$'s exist, i.e.\ $J=K$, and $\A_j = \{[A_{j1}]\}$ for each $1 \le j \le J$, which happens if and only if $\A=\A'$. This completes the proof of~\pref{ItemFFSxIneq}.

We next prove the ``if'' direction of item~\pref{ItemFFSxElem}. Suppose that $\A \sqsubset \A'$ is an elementary extension. In one case we have $\A' = \A \union \{[Z]\}$ where $Z$ is infinite cyclic, and it follows that $K=J+1$, that $(a)_j = (b)_j = 0$ for $1 \le j \le J$, and that $(c)_{J+1}=1$. In the other case, there exists $j_0 \in \{1,\ldots,J\}$ and two components $[A],[A'] \in \A$ such that up to conjugacy we have $A'_{j_0} = A * A'$, and $\A' = (\A - \{[A],[A']\}) \union \{[A'_{j_0}]\}$. It follows that each $(a)_j = 0$, that $(b)_j = 1$ if $j=j_0$ and $(b_j)=0$ otherwise, and that there are no $(c)_j$'s. In either case we have $\DFF(\A)=\DFF(\A') + 1$.

Suppose now that $\A \sqsubset \A'$ is a proper expansion, equivalently $\DFF(\A) - \DFF(\A') > 0$, equivalently at least one of the quantities $(a)_j$, $(b)_j$, $(c)_j$ is positive. In each case we exhibit a free splitting $\C$ such that $\A \sqsubset \C \sqsubset \A'$, and $\A \sqsubset \C$ is elementary. Item~\pref{ItemFFSxIncr} and the remaining contentions of item~\pref{ItemFFSxElem} follow immediately.

\smallskip

\textbf{Case 1:} Some $(a)_j > 0$  ($1 \le j \le J$)  which means  the free factorization  $A'_j = A_{j1} * \cdots * A_{j k_j} * B_j$ has nontrivial cofactor $B_j$. Let $Z$ be rank~$1$ free factor of $B$ and let $\C = \A \union \{[Z]\}$.

\smallskip

\textbf{Case 2:} Some $(b)_j > 0$ ($1 \le j \le J$) which means $\A_j = \{[A_{j1}],[A_{j2}],\ldots,[A_{jk_j}]\}$ has $k_j \ge 2$ components. Let $\C = (\A - \{[A_{j1}],[A_{j2}]\}) \union \{[A_{j1}*A_{j2}]\}$.

\smallskip

\textbf{Case 3:} Some $(c)_j > 0$ exists, which means $J < K$. For $J+1 \le j \le K$ each of the groups $A'_j$ is free of positive rank. Let $Z \subgroup A'_{J+1}$ be a rank~$1$ free factor and let $\C = \A \union \{[Z]\}$. 
\end{proof}

We showed in Corollary~\ref{CorollaryWeakMeetProps} that ``meet'' is a binary operation on the set of free factor systems of a group relative to a given free factor system. As~a corollary to Lemma~\ref{LemmaFFSNorm} we can extend this to a multivariate operation (c.f.\ \cite[Section 2.6]{\BookOneTag} for the case $\Gamma=F_n$, $\A=\emptyset$):

\newcommand\bF{\textbf{F}}

\begin{corollary}
\label{CorollaryFFSupport}
For any group $\Gamma$, any free factor system $\A$ of $\Gamma$, and any set $\bF$ of free factor systems of $\Gamma$ rel~$\A$, there exists a unique free factor system rel~$\A$, denoted $\meet \bF$ such that the following properties hold:
\begin{enumerate}
\item\label{ItemMeetInAll}
For all $\F \in \bF$ we have $\meet\bF \sqsubset \F$.
\item\label{ItemMeetUniversal}
For any free factor system~$\B$ of $\Gamma$ rel~$\A$, if $\B \sqsubset \F$ for all $\F \in \bF$ then $\B \sqsubset \meet \bF$.
\end{enumerate}
Furthermore, choosing an indexing $\bF = \{\F_i \suchthat i \in I\}$, we have:
\begin{enumeratecontinue}
\item $\meet\bF$ is the set of conjugacy classes $[F]$ of all nontrivial subgroups $F \subgroup \Gamma$ that can be expressed in the following form, for some indexed family of subgroups $\{F_i\}_{i \in I}$ such that $[F_i] \in \F_i$:
$$F = \bigcap_{i \in I} F^{\vphantom\inv}_i
$$
\end{enumeratecontinue}
\end{corollary}

\begin{proof} Uniqueness follows from the fact that if $\meet \bF$ and $\meet' \bF$ both satisfy properties~\pref{ItemMeetInAll} and~\pref{ItemMeetUniversal} then $\meet \bF \sqsubset \meet' \bF \sqsubset \meet \bF$, and so $\meet\bF = \meet'\bF$ because the relation $\sqsubset$ is a partial order. 

In the special case when $\bF \subset \{[\Gamma]\}$ --- that is, when $\bF = \emptyset$ or $\{[\Gamma]\}$ --- clearly $\meet\bF = \{[\Gamma]\}$ satisfies ~\pref{ItemMeetInAll} and~\pref{ItemMeetUniversal}.

If $\bF \not\subset \{[\Gamma]\}$, consider a finite sequence $\B_1,\ldots,\B_K \in \bF$ (with $K \ge 1$), and the following inductively defined sequence of free factor systems of $\Gamma$ rel~$\A$:
\begin{align*}
\F_0 &= \{[\Gamma]\} \\
\F_{k} &= \B_k \meet \F_{k-1} \quad (1 \le k \le K)
\end{align*}
and so we have a length~$K$ sequence of extensions
$$\F_K \sqsubset \cdots \sqsubset \F_1 \sqsubset \F_0
$$
We say that $\B_1,\ldots,\B_K$ is a \emph{proper sequence} if the extension $\F_{k} \sqsubset \F_{k-1}$ is proper for each $0 \le k-1 < k \le K$. By choosing any $\B_1 \in \bF - \{[\Gamma]\}$ we obtain a proper sequence of length $K=1$. By Lemma~\ref{LemmaFFSNorm}~\pref{ItemFFSxChain} any proper sequence has \hbox{length~$\le \DFF(\A)$,} and therefore there exists a proper sequence as denoted above whose length~$K$ is maximal. Choosing any such maximal sequence, it remains to check that $\F_K$ satisfies properties~\pref{ItemMeetInAll} and~\pref{ItemMeetUniversal}. For any $\B$ as in property~\pref{ItemMeetUniversal} we have $\B \sqsubset \B_k$ for all $k=1,\ldots,K$, from which it follows by induction that $\B \sqsubset \F_K$. Property~\pref{ItemMeetInAll} must also hold for if not then, choosing $\B_{K+1} \in \bF$ such that $\F_K \not\sqsubset \B_{K+1}$, we obtain a longer proper sequence with $\F_{K+1} = \B_{K+1} \meet \F_K$, violating maximality of~$K$.

For proving the \emph{furthermore} sentence, choose again a maximal, finite, proper sequence of length $K$ with notations as above. By combining the definition of the original binary meet operation with Corollaries~\ref{CorollaryWeakMeetProps} and~\ref{CorollaryMeet} and an induction argument, it follows that $\meet \bF$ is the set of conjugacy classes $[F]$ of all nontrivial subgroups $F \subgroup \Gamma$ which can be written in the form
$$F = B_1 \, \intersect \, \cdots \, \intersect  B_I
$$ 
for some sequence of subgroups $B_1,\ldots,B_K \subgroup \Gamma$ such that $[B_k] \in \B_k$. For any $\B' \in \bF$ and any free factor $B'$ such that $[B'] \in \B'$, the intersection $F \cap B'$ must be either trivial or equal to $F$, for otherwise $\F_{K+1} = \B' \meet \F_K$ again produces a longer proper sequence.
\end{proof}

\section{Relative free splitting complexes}
\label{SectionFFFAndRel}

In Sections~\ref{SectionBasicFSDefs}--\ref{SectionFSRelComplex}, given an arbitrary freely decomposable group $\Gamma$ we define free splittings of $\Gamma$ and their partial ordering $\succ$ called the ``collapse relation''. Also, using these concepts we define free splitting complexes of $\Gamma$, both the ``absolute'' free splitting complex $\FS(\Gamma)$ and the free splitting complex $\FS(\Gamma;\A)$ ``relative to'' a choice of free factor system~$\A$. We also study a function which associates to each free splitting a free factor system called its ``vertex stabilizer system'', and in Section~\ref{SectionSqsubsetSuccRelations} we study how this function relates the partial orderings $\sqsubset$ and $\succ$. In Section~\ref{SectionFSDepthAndDimension} we study the depth of the inverted partial ordering $\prec$. We apply that study to obtain a formula for the dimension of $\FS(\Gamma;\A)$, and to obtain a finer understanding of the partial ordering as it relates to inclusion of simplices. Of particular importance is Proposition~\ref{PropMaximizingSimplices} that explains exactly which free splittings are maximal and minimal with respect to the collapse relation~$\succ$, and which chains of the relation $\succ$ correspond to maximal simplices of $\FS(\Gamma;\A)$.

The proofs in this section are primarily applications of Bass-Serre theory along with basic topological manipulations of graphs and trees, and a few further applications of the Kurosh Subgroup Theorem.

For the case of $\Gamma = F_n$, many of the results of this section, regarding basic concepts of free splittings and the collapse partial ordering may be familiar to a reader of \FSOne. Nonetheless we examine these concepts from new points of view, in order to study relative free splitting complexes. Throughout this section we try to  view these points first from the vantage of the special case $\Gamma = F_n$, before moving on the general formulation. This is done so as to enable the reader interested mostly in $\Gamma = F_n$ to get through this section more quickly.

\subsection{Basic terminology and notation regarding graphs.} 
\label{SectionBasicFSDefs}

A \emph{graph} $G$ is a 1-dimensional $\Delta$-complex, a \emph{tree} is a contractible graph, and a \emph{subgraph} of a graph $G$ is subset that is a subcomplex of some simplicial decomposition of~$G$. Given a subgraph $H \subset G$, its \emph{complementary subgraph}, denoted $G \setminus H$, is the closure of the set theoretic complement $G-H$; equivalently, with respect to a simplicial decomposition in which $H$ is a subcomplex, $G \setminus H$ is the union of those edges of $G$ not contained in $H$. Unless otherwise specified we work in the PL category: every map between graphs is assumed to be PL, meaning that it is simplicial with respect to a choice of subdivision of the domain and range. Occasionally we specify that a map is \emph{simplicial}, meaning simplicial with respect to the \emph{given} graph structures on the domain and range.

Given $p \in G$, let $D_p G$ denote the set of \emph{directions} at $p$, meaning initial germs of locally injective paths with initial point $p$. If $p$ is a vertex then each element of $D_p G$ is uniquely represented by an oriented edge of $G$ with initial vertex~$p$.

\subsection{Free splittings and the partial order $\succ$.}
\label{SectionFSSucc}

A \emph{free splitting} of $\Gamma$ is a simplicial action $\Gamma \act T$ of the group $\Gamma$ on a tree $T$ such that the action is \emph{minimal} meaning that there is no proper $\Gamma$-invariant subtree, $T$ is not a point, the stabilizer of each edge of $T$ is trivial, and there are finitely many edge orbits. It follows that there are finitely many vertex orbits, and so $T/\Gamma$ is a finite graph of groups. It also follows, using minimality, that $T$ has no valence~$1$ vertices. There is an induced action $\Gamma \act \bigcup_p D_p T$ where $p$ ranges over the vertex set of~$T$, and this action is free: if not then there would exist $\gamma \ne \Id \in \Gamma$ and an oriented edge $E \subset T$ such that $\gamma \cdot E = E$, contradicting that $T$ is a free splitting.

Two free splittings $S,T$ of $\Gamma$ are \emph{equivalent}, denoted $S \approx T$, if there exists a $\Gamma$-equivariant homeomorphism $f \from S \mapsto T$. While this homeomorphism need not be simplicial, like any map it is assumed to be PL.

\smallskip
Consider a free splitting $\Gamma \act T$ and let $P \subset T$ be the set of points with nontrivial stabilizer. On $T$ there exists a \emph{natural} graph structure, namely a graph structure satisfying the following: its $0$-skeleton consists of the \emph{natural vertices} of $T$ consisting of the union of $P$ with the set of vertices of valence~$\ge 3$; the action $\Gamma \act T$ is simplicial with respect to the natural graph structure; and the given graph structure on $T$ is a $\Gamma$-equivariant subdivision of the natural graph structure. Edges of a natural graph structure are called \emph{natural edges}, and they are uniquely determined; in fact the natural graph structure is unique up to a $\Gamma$-equivariant PL self-homeomorphism of $T$ that restricts to the identity on the set of natural vertices, thus taking each natural edge to itself. Edges of the given graph structure on $T$ are sometimes called \emph{edgelets} in order to contrast them with the natural edges of $T$; each natural edge is a union of edgelets.

\smallskip
\textbf{Remark.}  In proving the existence of the natural graph structure for a free splitting $\Gamma \act T$ one encounters the following quirk. For a vertex $p \in T$ of valence~$2$, its stabilizer group $\Stab(p)$ is nontrivial if and only if $\Stab(p)$ is cyclic of order~$2$, in which case $\Stab(p)$ acts transitively on the 2-point set $D_p T$. Using this one can show that in the (very special) case where the tree $T$ has an isolated end, the group $\Gamma$ is infinite dihedral and the action $\Gamma \act T$ is equivalent to a standard infinite dihedral action on~$\reals$ generated by reflections across the points of $\Z \subset \reals$; the vertices of $T$ with nontrivial stabilizer therefore form the natural vertex set (recall here our convention, from the opening of Section~\ref{SectionBasicFFDefs}, that $\Gamma$ cannot be infinite cyclic). In the (generic) case that $T$ has no isolated ends, it is evident that the vertices of valence~$\ge 3$ already form the vertex set of a $\Gamma$-equivariant graph structure on $T$, and by further subdividing at the valence~$2$ vertices with nontrivial stabilizer one arrives at the natural graph structure. 

\medskip
\noindent
\textbf{Vertex stabilizer systems.} Associated to each free splitting $\Gamma \act T$ is a nonfull free factor system of $\Gamma$ denoted $\Fell T$, consisting of the conjugacy classes of nontrivial vertex stabilizers and called either the \emph{vertex stabilizer system} of $T$ or the \emph{elliptic subgroup system}. The fact that $\Fell T$ is indeed a free factor system follows from Bass-Serre theory, by using any isomorphism between $\Gamma$ and the fundamental group of the quotient graph of groups $T/\Gamma$. In the converse direction we have the following fact, which will often be invoked silently:

\begin{lemma}\label{LemmaFSToFFOnto} For every nonfull free factor system $\A$ of $\Gamma$ there exists a free splitting $\Gamma \act T$ such that $\A = \Fell T$.
\end{lemma}

\begin{proof} Choose a realization $\Gamma = A_1 * \cdots * A_I * B$ of $\A = \{[A_1],\ldots,[A_I]\}$. Construct a graph of groups with base point $p$, attaching to $B$ a rose with rank equal to $\corank(\A)=\rank(B)$, and attaching $K$ additional edges to $p$ with opposite vertices of valence~$1$ having respective vertex groups $A_1,\ldots,A_I$. The fundamental group of this graph of groups has an isomorphism to $\Gamma = A_1 * \cdots * A_I * B$. Letting $T$ be the Bass-Serre tree of this graph of groups with associated $\Gamma$ action, we obtain a free splitting of $\Gamma$ satisfying $\Fell T=\A$. 
\end{proof}

\smallskip
\noindent
\textbf{Nondegenerate subgraphs, collapse maps, and the partial ordering $\succ$.} 
An \emph{invariant subgraph} of a free splitting $T$ is a $\Gamma$-invariant subgraph $\tau \subset T$ with respect to some subdivision of~$T$. A component of $\tau$ is \emph{degenerate} if it consists of a single point. A \emph{nondegenerate subgraph} of $T$ is a proper invariant subgraph with no degenerate component. 

A \emph{collapse map} $f \from T \to S$ is a map which is simplicial with respect to some subdivisions of $T$ and $S$, such that for each $x \in S$ the inverse image $f^\inv(x)$ is connected. The union of those inverse images $f^\inv(x)$ which are not single points is a nondegenerate subgraph $\sigma \subset T$ called the \emph{collapse forest}. Letting $T \mapsto T/\sigma$ denote the equivariant quotient map under which each component of $\sigma$ is collapsed to a single point, it follows that $T/\sigma$ and $S$ are equivalent free splittings. We sometimes incorporate $\sigma$ into the notation by writing $T \xrightarrow{[\sigma]} S$. 

A collapse $T \xrightarrow{[\sigma]} S$ is \emph{natural} if $\sigma$ is a natural subcomplex of $T$, equivalently $\sigma$ is a union of natural edges of~$T$. Note that for any collapse map $T \xrightarrow{[\sigma]} S$ there exists a natural collapse map $T \xrightarrow{[\sigma']} S$ where $\sigma'$ is the union of natural edges of $T$ contains in $\sigma$.

We define a relation denoted $T \succ S$ to mean that there exists a collapse map $T \mapsto S$. This relation is a well-defined on equivalence classes, and it is a partial ordering (note that a composition of collapse maps is a collapse map). We also express the relation $T \succ S$ as $S \prec T$, and by using various terminologies such as that $T$ \emph{collapses to}~$S$, or that $S$ \emph{expands to}~$T$. If~furthermore $T \not\approx S$ then the collapse or expansion is \emph{proper}, which holds if and only if for some (any) collapse map $T \xrightarrow{[\sigma]} S$ the subgraph $\sigma$ contains a natural edge. 

Note that for any map of free splittings $f \from S \to T$, each element of $\Gamma$ that is elliptic in $S$ is also elliptic in $T$, and therefore $\Fell S \sqsubset \Fell T$ (see Lemma~\ref{LemmaRealCollapse}~\pref{ItemRCFFtoFS}). It follows that if $S \succ T$ then $\Fell S \sqsubset \Fell T$.

\smallskip
\textbf{Remark on abuses of notation.} While a free splitting is formally denoted $\Gamma \act T$, and we often use this notation to emphasize the action, also we often suppress the action from the notation and simply write~$T$. The action is always suppressed from the notation for the equivalence class $[T]$, and sometimes we write just~$T$ for the equivalence class.

\subsection{Absolute and relative free splitting complex.}
\label{SectionFSRelComplex}
We define the (absolute) \emph{free splitting complex} of $\Gamma$, denoted $\FS(\Gamma)$, to be the simplicial complex which is the geometric realization of the  set of equivalence classes of free splittings of $\Gamma$ partially ordered by $\prec$. Thus $\FS(\Gamma)$ has a $0$-simplex for each equivalence class of free splittings $\Gamma \act T$, denoted $[T]$. In general $\FS(\Gamma)$ has a $K$-simplex for each $K+1$-tuple of distinct $0$-simplices $[T_0],[T_1],\ldots,[T_K]$ such that $T_0 \prec T_1 \prec \cdots \prec T_K$; this simplex is denoted $[T_0] \prec [T_1] \prec \cdots \prec [T_K]$. By our convention that $\Gamma$ be freely indecomposable, $\FS(\Gamma)$ is always nonempty.

Consider now a nonfull free factor system $\A$ of $\Gamma$. A \emph{free splitting of $\Gamma$ rel~$\A$} is a free splitting $\Gamma \act T$ with the property that $\A \sqsubset \Fell T$, equivalently  \emph{$\A$ is elliptic with respect to $T$} meaning that each subgroup of $\Gamma$ representing an element of $\A$ fixes some point of~$T$. The \emph{free splitting complex of $\Gamma$ rel~$\A$}, denoted $\FS(\Gamma;\A)$, is the flag subcomplex of $\FS(\Gamma)$ consisting of all simplices $[T_0] \prec\cdots\prec [T_K]$ such that $\A$ is elliptic in each of the free splittings $\Gamma \act T_0,\ldots,T_K$; this is equivalent to requiring simply that $\A$ is elliptic in $T_K$, because $\Fell T_K \sqsubset \cdots \sqsubset \Fell T_0$. The requirement that $\A$ not be full implies that free splittings rel~$\A$ exist (by Lemma~\ref{LemmaFSToFFOnto}) and so $\FS(\Gamma;\A)$ is nonempty. In Corollary~\ref{CorollaryConnected} below we will see that $\FS(F_n;\A)$ is connected.

Note that if $\Gamma$ has a proper Grushko decomposition, equivalently if there exists a free factor system~$\A$ which is minimal with respect to $\sqsubset$ (see Proposition~\ref{PropGrushkoMinimum}), then $\FS(\Gamma)=\FS(\Gamma;\A)$; this holds for example whenever $\Gamma$ is finitely generated.

\smallskip
\textbf{Remarks on terminology and notation.} The notation $[T]$ is used both for the equivalence class of a free splitting $\Gamma \act T$ and for the corresponding $0$-simplex of $\FS(\Gamma)$. Sometimes we abuse notation by writing things like ``$T \in \FS(\Gamma;\A)$'' which can be read formally either as ``$T$ is a free splitting of $\Gamma$ rel~$\A$'' or as ``$[T]$ is a $0$-simplex of $\FS(\Gamma;\A)$''.

In \FSOne\ we used the notation $\FS(F_n)$ a little differently, namely the complex with one $k$-simplex for each equivalence class of free splittings $T$ having $k+1$-orbits of natural edges, where the face inclusion is defined by the relation $S \prec T$. Also, we used the notation $\FS'(F_n)$ for the first barycentric subdivision of $\FS(F_n)$ which is equivalent to the free splitting complex as defined in this section. But even in \FSOne\ we worked primarily with this first barycentric subdivision, and since relative free splitting complexes live naturally as subcomplexes of this first barycentric subdivision, in this current work we switch the notation and we hope that this does not cause confusion.

\subsection{Relations between the partial orders $\sqsubset$ and $\succ$.} 
\label{SectionSqsubsetSuccRelations}
In the following lemma we collect properties relating the partial order $\sqsubset$ on free factor systems to the partial order $\succ$ on (equivalence classes of) free splittings. These properties are all true as well when they are specialized by choosing a free factor system $\A$ and putting in the qualifier ``rel~$\A$''.


\begin{lemma} \label{LemmaRealCollapse} For any $\Gamma$ the following hold:
\begin{enumerate}
\item\label{ItemRCFStoFF}
For any map of free splittings $f \from T \to S$ we have an extension $\Fell T \sqsubset \Fell S$ of free factor systems. In particular if $S \prec T$ then $\Fell T \sqsubset \Fell S$.
\item\label{ItemRCFMeet}
For any free factor system $\A$ of $\Gamma$ and any two free splittings $\Gamma \act S,T$ rel~$\A$ there exists a free splitting $\Gamma \act U$ rel~$\A$ and a natural collapse map $f \from U \to T$ such that $\Fell U = \Fell S \meet \Fell T$ and such that for each $x \in T$, if the subgroup $\Stab_T(x)$ is nontrivial then its action on $f^\inv(x) \subset U$ is equivalent to its action on its minimal subtree in $S$.
\item\label{ItemRCFFtoFS}
For any free splitting $\Gamma \act T$ and any free factor system $\B \sqsubset \Fell T$ there exists a free splitting $U$ and a collapse map $U \mapsto T$ such that $\Fell U=\B$.
\item\label{ItemRCFSequence}
More generally, for each sequence of extensions $\A_0 \sqsubset \A_1 \sqsubset \cdots \sqsubset \A_K$ of free factor systems rel~$\A$, and each free splitting $S_K$ such that $\Fell S_K=\A_K$ there exists a sequence of free splittings and collapses $S_0 \collapses S_1 \collapses \cdots \collapses S_K$ such that $\Fell S_k=A_k$ for each $k=0,\ldots,K$.
\end{enumerate}
\end{lemma}

\begin{proof} Item~\pref{ItemRCFStoFF} is evident since $\Stab(x) \subgroup \Stab(f(x))$. Clearly \pref{ItemRCFMeet}$\implies$\pref{ItemRCFFtoFS} by taking $S$ to be any free splitting such that $\Fell S=\B$ and using that $\B \sqsubset \C$ implies $\B \meet \C = \B$. Also clearly \pref{ItemRCFFtoFS}$\implies$\pref{ItemRCFSequence}. 

Item~\pref{ItemRCFMeet} says intuitively that one can always ``blow up'' $T$ to get some $U$ so that for each $x \in T$ the actions of $\Stab(x)$ on its blowup in $U$ is a copy of its action on its minimal subtree in~$S$. The proof of item~\pref{ItemRCFMeet} is an elaboration of the Bass-Serre theory proof of the Kurosh Subgroup Theorem (see e.g.\ \cite{Cohen:CombinatorialGroupTheory}); here are a few details. First apply Proposition~\ref{PropMeet} to conclude that $\B = \Fell S \meet \Fell T$ is a free factor system rel~$\A$. Consider $x \in T$ such that $\Stab_T(x)$ is nontrivial, let $S^x \subset S$ be the minimal subtree for the action $\Stab_T(x) \act S$, and suppose that $S^x$ is not a point. Blow up the vertex $x$ using $S^x$: detach each of the directions of $D_x T$ from $x$, then remove $x$, then reattach the directions of $D_x T$ to a copy of the tree $S^x$ in a $\Stab_T(x)$-equivariant manner. Now extend this ``detachment--attachment'' operation over the whole orbit of~$x$, reattaching the directions in a $\Gamma$-equivariant manner. Doing this for each orbit of such points $x$ results in the desired free splitting $\Gamma\act U$. 
\end{proof}

\subsection{Free splitting depth of free factor systems and dimensions of relative free splitting complexes.}
\label{SectionFSDepthAndDimension}

The absolute free splitting complex of a rank~$n$ free group $\FS(F_n)$ has the following easily proved properties. Define a free splitting $F_n \act T$ to be \emph{generic} if every vertex has valence~$3$. First, $T$ has at most $3n-3$ natural edge orbits, the maximum being attained if and only if $T$ is generic. Also, the maximal number of natural vertex orbits is the number attained for generic $T$ which is $2n-2$. These are proved by simple Euler characteristic calculations taking place in the quotient graph of groups $T/F_n$. Next, given a $D$-simplex $[T_0] \prec [T_1] \prec \cdots \prec [T_D]$ with corresponding sequence of natural collapse maps $T_D \mapsto \cdots \mapsto T_1 \mapsto T_0$, the following are easily proved to be equivalent:
\begin{enumerate}
\item $D = 3n-4$.
\item $T_D$ is generic, each map $T_d \mapsto T_{d-1}$ collapses exactly one orbit of natural edges, and $T_0$ has exactly one orbit of natural edges.
\item The simplex $[T_0] \prec [T_1] \prec \cdots \prec [T_D]$ is maximal, meaning it is not a proper face of any other simplex.
\end{enumerate}
As a consequence, the dimension of $\FS(F_n)$ equals $3n-4$ and every simplex is a face of some simplex of maximal dimension $3n-4$.

We now generalize, stating and proving analogous results for relative free splitting complexes. 

\begin{definition} 
\label{DefFSDepth}
Let $\Gamma$ be a group and $\A$ any free splitting of $\Gamma$.
\begin{enumerate}
\item The \emph{free splitting depth} of~$\A$ is defined to be the number
$$\DFS(\A) = 3 \corank(\A) + 2 \abs{\A} - 4
$$
\item A free splitting $\Gamma \act T$ rel~$\A$ is \emph{generic} if $\Fell T=\A$ and for each vertex $v$ the following holds: if $\Stab(v)$ is trivial then $v$ has valence~$\le 3$; whereas if $\Stab(v)$ is nontrivial then $\Stab(v)$ acts transitively on $D_v T$. 
\end{enumerate}
\end{definition}
\noindent
Note that for $\Gamma \act T$ to be generic, it is equivalent that in the quotient graph of groups $G = T / F_n$ the following hold: the nontrivial vertex groups are of the form $A_1,\ldots,A_I$ where $\A = \{[A_1],\ldots,[A_I]\}$; and for every vertex $V$ of $G$, if $V$ has trivial vertex group then $V$ has valence~$2$ or $3$, whereas if $V$ has nontrivial vertex group then $V$ has valence~$1$. One can always choose the vertex groups to fit into a realization of $\A$ of the form $\Gamma = A_1 * \cdots * A_I * B$ in such a way that $B$ is identified with a lift to $\Gamma$ of the fundamental group of the underlying graph of $G$.





\begin{proposition}\label{PropGenericFS}
For any free splitting $\Gamma \act T$ rel~$\A$ the following hold:
\begin{enumerate}
\item\label{ItemNatEBound}
The number of natural edge orbits of $T$ satisfies 
$$E(T) \le \DFS(\A)+1 = 3 \corank(\A) + 2 \abs{\A} - 3
$$
\item\label{ItemGenericEquivalencies}
The following are equivalent:
\begin{enumerate}
\item\label{ItemNatECount}
$E(T) = \DFS(\A)+1$.
\item\label{ItemTGeneric}
$T$ is generic.
\item\label{ItemTMaximal}
$[T]$ is maximal with respect to the partial ordering $\prec$, that is, for every free splitting $\Gamma \act U$, if there exists a collapse map $U \mapsto T$ then $[U]=[T]$.
\end{enumerate}
\item\label{ItemVMax}
The number of natural vertex orbits of $T$ satisfies 
$$V(T) \le \DFS(\A) + 2 - \corank(\A) = 2 \corank(\A) + 2\A - 2
$$
with equality if and only if $T$ is generic.
\end{enumerate}
\end{proposition}

\begin{proof} In this proof we assume that all vertices and all edges of free splittings are natural, equivalently no valence~$2$ vertex has nontrivial stabilizer; if any such vertices exist, just remove them from the $0$-skeleton. Thus every vertex and every edge of the quotient graph of groups is natural, meaning that no valence~$2$ vertex has trivial vertex group. Also, all collapse maps are natural and are nontrivial if and only if they are not homeomorphisms. Having done this, for any such free splitting $T$ with quotient $G = T / \Gamma$ the numbers $E=E(T)$ and $V=V(T)$ are just the counts of edge and vertex orbits of~$T$, equivalent of edges and vertices of $G$. Let $V_k=V_k(T)$ be the number of valence~$k$ vertices of $G$, equivalently the number of $\Gamma$-orbits of vertices $v \in T$ at which the set $D_v \Gamma$ has exactly $k$ orbits under the action of $\Stab(v)$.

We first prove \pref{ItemTGeneric}$\implies$\pref{ItemNatECount}. Assuming $T$ is generic we have $V = V_1 + V_3$ and $V_1 = \abs{\A}$. We also have $E = \frac{1}{2} (V_1 + 3 V_3)$ and $\corank(\A) = E - V + 1 \,\, (= \rank(G))$. Eliminating $V$, $V_1$, and $V_3$ gives $E = \DFS(\A)+1$.

We next claim that for every free splitting $\Gamma \act T$ rel~$\A$ there exists a generic free splitting $\Gamma \act S$ rel~$\A$ and a natural collapse map $S \xrightarrow{[\sigma]} T$. From this claim we obtain the following consequences. First, item~\pref{ItemNatEBound} holds because $E(T) \le E(S) = \DFS(\A)+1$. Next, the implication \pref{ItemNatECount}$\implies$\pref{ItemTMaximal} holds, because if \pref{ItemTMaximal} does not hold then there exists $U$ and a collapse $U \xrightarrow{[\sigma]} T$ such that $[U] \ne [T]$, and so $\sigma$ is nontrivial, implying by~\pref{ItemNatEBound} that $E(T) < E(U) \le \DFS(\A)+1$. Next, \pref{ItemTMaximal}$\implies$\pref{ItemTGeneric}, because if $T$ is not generic then the collapse map $S \mapsto T$ is nontrivial and so $[T]$ is not maximal. Finally, item~\pref{ItemVMax} follows because the collapse map $S \xrightarrow{[\sigma]} T$ takes the natural vertices of $S$ onto the natural vertices of $T$ and so $V(S) \ge V(T)$, with equality if and only if $\sigma=\emptyset$ if and only if $[S]=[T]$ if and only if $T$ is generic, and 
$$V(S) = 1 - \rank(S/\Gamma) + E(S) = 1 - \corank(\A) + \DFS(\A) + 1
$$

To prove the claim we do a sequence of expansions of $T$ one at a time to build up the properties of a generic free splitting rel~$\A$. First, by applying the expansion from Lemma~\ref{LemmaRealCollapse}~\pref{ItemRCFFtoFS} we may assume that $\Gamma \act T$ satisfies $\Fell T=\A$. 

Next, by expanding $T$ we may assume that if $v \in T$ is a vertex with nontrivial stabilizer, and so $[\Stab(v)] \in \A$, then the number $k_v$ of $\Stab(v)$-orbits in the set $D_v\Gamma$ satisfies $k_v=1$. Otherwise, if $k_v \ge 2$, choose orbit representatives $d_1,\ldots,d_k \in D_v\Gamma$, do a simultaneous partial fold of these directions by identifying proper initial segments into a single segment $e$, having one vertex with the same stabilizer as $v$ and opposite vertex of valence $k+1$ and with trivial stabilizer. Extending these identifications equivariantly, the resulting free splitting is an expansion of $T$ because by collapsing the orbit of $e$ we recover~$T$. 

Finally, we may assume that if $v$ is a vertex with trivial stabilizer and valence~$\ge 3$ then $v$ has valence~$3$, for otherwise we may group $D_v\Gamma$ into two sets of cardinality $\ge 2$ and expand $T$ by pulling these two sets apart, inserting a new edge, and extending this expansion equivariantly over the orbit of~$v$. This expansion decreases the lexicographically ordered sequence $(V_3(T), V_4(T),\ldots)$.
\end{proof}

\begin{definition} Let $\Gamma$ be a group.
\begin{enumerate}
\item A natural collapse map $S \xrightarrow{[\sigma]} T$ of free splittings of $\Gamma$ is \emph{elementary} if $\sigma$ consists of a single orbit of natural edges. 
\item A \emph{one edge free splitting} is a free splitting $\Gamma \act T$ with exactly one natural edge orbit.
\end{enumerate}
\end{definition}

\smallskip


\begin{proposition}
\label{PropMaximizingSimplices}
For each $D$-simplex $[T_0] \prec [T_1] \prec \cdots \prec [T_D]$ in $\FS(\Gamma;\A)$ with corresponding sequence of natural collapse maps 
$$T_D \xrightarrow{[\sigma_D]} T_{D-1} \xrightarrow{[\sigma_{D-1}]} \cdots \xrightarrow{[\sigma_2]} T_1 \xrightarrow{[\sigma_1]} T_0
$$
the following are equivalent:
\begin{enumerate}
\item\label{ItemDIsDFS}
$D = \DFS(\A)$.
\item\label{ItemSmallSteps}
Each of the following holds: \, (a) $T_D$ is generic rel~$\A$; \, (b) each collapse map $T_d \mapsto T_{d-1}$ is elementary, for $d=1,\ldots,D$; \, (c) $T_0$ is a one-edge free splitting.
\item\label{ItemMaximalSimplex}
The simplex $[T_0] \prec [T_1] \prec \cdots \prec [T_D]$ is maximal, meaning it is not a face of any other simplex.
\end{enumerate}
As a consequence, the dimension of $\FS(\Gamma;\A)$ equals $\DFS(\A)$, and every simplex is a face of a simplex of maximal dimension $\DFS(\A)$.
\end{proposition}

\begin{proof} As in the proof of Proposition \pref{PropGenericFS}, we assume that all edge and vertices are natural, and we continue to use the notation $E(T),V(T)$ as in that proof.

The scheme of the proof is \pref{ItemDIsDFS}$\iff$\pref{ItemSmallSteps}$\iff$\pref{ItemMaximalSimplex}.

Assuming item \pref{ItemSmallSteps} we shall prove \pref{ItemDIsDFS}. By applying Proposition~\ref{PropGenericFS} one concludes $E(T_D)=\DFS(\A)+1$, and then one notices that from \pref{ItemSmallSteps} it follows that the edge orbits of $T_D$ are collapsed one-at-a-time until only one remains, implying that the number $D$ of collapse maps equals $\DFS(\A)$. 

Assuming \pref{ItemDIsDFS} we shall prove \pref{ItemSmallSteps}. For any natural collapse map $S \xrightarrow{[\sigma]} T$, letting $E(\sigma)$ be the number of natural edge orbits of $S$ contained in the $\Gamma$-equivariant natural subforest $\sigma$, we have $E(T) + E(\sigma)=E(S)$; recall also that $E(\sigma) =0 \iff \sigma = \emptyset \iff [S]=[T]$. Using that each of $E(\sigma_D),\ldots,E(\sigma_1),E(T_0)$ is $\ge 1$ we have
\begin{align*}
D + 1 &\le E(\sigma_D) + \cdots + E(\sigma_1) + E(T_0) \\
          &= E(T_D) \\
          &\le \DFS(\A) + 1 \qquad\text{(by Proposition~\ref{PropGenericFS}~\pref{ItemNatEBound})} \\
          &= D+1 \qquad\text{(by assumption of \pref{ItemDIsDFS})}
\end{align*}
and so all inequalities are equations. Applying Proposition~\ref{PropGenericFS}~\pref{ItemGenericEquivalencies}, it follows $T_D$ is generic. It also follows that $E(\sigma_D)=\cdots=E(\sigma_1)=E(T_0)=1$, which proves~\pref{ItemSmallSteps}.

Assuming \pref{ItemSmallSteps} holds, we prove \pref{ItemMaximalSimplex} as follows. Since $T(D)$ is generic, there does not exist any proper collapse map of the form $S \mapsto T_D$ for that would imply $E(S) > \DFS(\A)+1$, contradicting Proposition~\ref{PropGenericFS}~\pref{ItemNatEBound}. Since $T_d \mapsto T_{d-1}$ is elementary, there exist any factorization of $T_d \mapsto T_{d-1}$ into proper collapse maps of the form $T_d \mapsto S \mapsto T_{d-1}$ because that would imply $E(T_d) \ge E(T_{d-1}) + 2$, contradicting that $E(T_d) = E(T_{d-1})+1$. Nor does there exist any proper collapse map of the form $T_0 \mapsto S$, for that would imply $E(S) \le E(T_0)-1 = 1 - 1 = 0$. It follows that the simplex $[T_0] \prec \cdots \prec [T_D]$ is maximal.

Assuming \pref{ItemSmallSteps} fails, we prove that~\pref{ItemMaximalSimplex} fails as follows. One of (a), (b), or (c) must fail. If $T_D$ is not generic then by Proposition~\ref{PropGenericFS}~\pref{ItemTMaximal} there exists a free splitting $S$ and a proper natural collapse map $S \mapsto T_D$. If $T_d \xrightarrow{[\sigma]} T_{d-1}$ is not elementary then, first collapsing a single edge orbit of $\sigma$, there a sequence of proper natural collapse maps $T_d \mapsto S \mapsto T_{d-1}$. If $T_0$ has more than one edge orbit then, collapsing just one edge orbit, there exists a proper natural collapse map $T_0 \mapsto S$. In each case we obtain a simplex of one dimension higher containing the simplex $[T_0] \prec \cdots \prec [T_D]$.
\end{proof}

\subsection{The relative outer automorphism group $\Out(\Gamma;\A)$.}
\label{SectionRelOut}
Now that the sets of free factor systems and free splittings rel~$\A$ have been defined together with various relations and operations on them, we pause here to carefully define the relative outer automorphism group $\Out(\Gamma;\A)$ and its actions on those sets. We also define the action of the group $\Out(\Gamma;\A)$ on the relative free splitting complex $\FS(\Gamma;\A)$, although the definition of its action on the complex of free factor systems rel~$\A$ will await the definition of that complex to be given in Section~\ref{SectionFFRelComplex}.

The group $\Out(\Gamma)$ has a canonical left action on the set of free factor systems $\A$, namely: given $\phi \in \Out(\Gamma)$, choosing a representative $\Phi \in \Aut(\Gamma)$, and choosing a realization $\Gamma = A_1 * \cdots * A_I * B$ of $\A$, one defines
$$\phi(\A) = \{[\Phi(A_1)],\ldots,[\Phi(A_I)]\}
$$
This action is well-defined independent of choices, the left action equations $\phi (\psi(\A)) = (\phi\psi)(\A)$ and $\Id(\A)=\A$ hold, and the action preserves the extension partial order $\sqsubset$ and the meet operation~$\meet$. 

\smallskip\textbf{Relative outer automorphism groups.} Given a free factor system $\A$ of $\Gamma$, the subgroup of $\Out(\Gamma)$ that fixes $\A$ is denoted $\Out(\Gamma;\A)$ and is called the \emph{outer automorphism group of $\Gamma$ rel~$\A$}. This is the group studied by Guirardel and Levitt in \cite{GuirardelLevitt:outer} who derive information about the virtual cohomological dimension of $\Out(\Gamma;\A)$ using information about the virtual cohomological dimensions of the groups $A_i$, $\Aut(A_i)$, and $\Out(A_i)$, $i=1,\ldots,I$.

\smallskip\textbf{Action on relative free splitting complexes.} The group $\Out(\Gamma)$ has a canonical right action on the set of equivalence classes of free splittings of $\Gamma$ as follows. Consider the equivalence class $[T]$ of a free splitting $\Gamma \act T$ with associated homomorphism $\alpha \from \Gamma \to \Aut(T)$; incorporating $\alpha$ into the notation we write $\Gamma \act_\alpha T$. Consider also $\phi \in \Out(\Gamma)$ represented by $\Phi \in \Aut(\Gamma)$. Precomposing $\alpha$ by $\Phi$ we obtain a homomorphism $\alpha \composed \Phi \from \Gamma \to \Aut(T)$ which defines a free splitting $\Gamma \act_{\alpha\composed\Phi} T$, the equivalence class of which is defined to be $[T] \cdot \phi$. This free splitting is well-defined, the right action equations $[T] \cdot (\phi \psi) = ([T] \cdot \phi) \cdot \psi$ and $[T] \cdot \Id = [T]$ hold, and the action preserves the collapse partial order $\succ$. We obtain thereby an induced right action of $\Out(\Gamma)$ on linear chains of free splittings as follows: 
$$\biggl( [T_0]  \, \succ  \, [T_1] \succ  \, \cdots  \, \succ [T_K] \biggr) \cdot \phi = [T_0] \cdot \phi  \, \succ \,  [T_1] \cdot \phi  \, \succ  \, \cdots  \, \succ \,  [T_K] \cdot \phi
$$
Finally, for any free factor system $\A$ of $\Gamma$, we obtain by restriction a right action of $\Out(\Gamma;\A)$ on linear chains of free splittings rel~$\A$. These chains define simplices of the relative free splitting complex $\FS(\Gamma;\A)$, and so we immediately obtain the right action of $\Out(\Gamma;\A)$ on $\FS(\Gamma;\A)$ by simplicial isomorphisms.

\smallskip\textbf{Action on chains of relative free factor systems.} The action of $\Out(\Gamma)$ on free factor systems preserving $\sqsubset$ induces an action on linear chains of free factor systems: 
$$\phi \bigl( \A_0 \sqsubset \A_1 \sqsubset \cdots \sqsubset \A_K \bigr) = \phi(\A_0) \sqsubset \phi(\A_1) \sqsubset \cdots \sqsubset \phi(\A_K)
$$
For any given free factor system $\A$ we obtain by restriction a left action of $\Out(\Gamma;\A)$ on linear chains of free splittings rel~$\A$. Once the formal definitions are given in Section~\ref{SectionFFRelComplex}, we will immediately obtain the left action of $\Out(\Gamma;\A)$ on the complex of relative free factor systems $\CFFS(\Gamma;\A)$ by simplicial isomorphisms.

\bigskip

We record here one fact for later use, which is a simple consequence of the definitions:

\begin{lemma}
\label{LemmaFofTEquivariance}
The function $[T] \mapsto \Fell T$ satisfies the inverted equivariance condition with respect to the actions of $\Out(\Gamma)$: given an equivalence class of free splittings $[T]$ and $\phi \in \Out(\Gamma)$ we have the following equation of free factor systems: 

\smallskip
\hfill $\Fell\bigl([T] \cdot \phi\bigr) = \phi^\inv\bigl(\Fell T\bigr)$ \hfill \qed
\end{lemma}

\section{Fold paths and free splitting units}
\label{SectionFoldPathUnits}
In this section we fix a group $\Gamma$ and a free factor system $\A$ in $\Gamma$, and we study fold paths in the relative free splitting complex~$\FS(\Gamma;\A)$. Section~\ref{SectionFoldSequences} contains the basic definitions, generalizing fold paths following \FSOne\ but also following \cite{BestvinaFeighn:subfactor} to the extent of dropping the ``gate 3 condition'' of \FSOne. In Section~\ref{SectionFSLow} we use fold paths to give an explicit description of $\FS(\Gamma;\A)$ in the simplest cases where the free factor system~$\A$ is very close to maximal in~$\Gamma$. In Section~\ref{SectionCombing} we generalize the concepts of combing of fold paths following \FSOne. In Section \ref{SectionSubgraphComplexity} we consider a measurement of the complexity of a $\Gamma$-invariant subforest of a free splitting $\Gamma \act T$, and we study how this complexity can change along a fold path. In Section \ref{SectionFSU} we use change of complexity to define free splitting units along fold paths; in later sections these units are shown to give efficient upper and lower bounds to distance along fold paths. We note that while free splitting units as defined here are \emph{a fortiori} comparable to free splitting units as defined in \FSOne, the definition here is somewhat simpler and easier to work with.

\subsection{Fold sequences} 
\label{SectionFoldSequences}

Stallings was the first to construct fold factorizations to analyze equivariant maps between group actions on trees \cite{Stallings:folds}. Bestvina and Feighn made a careful study of fold factorizations in \cite{BestvinaFeighn:bounding}. The choices made in those constructions make clear that the resulting fold path is not unique. 

It is common in the literature to fix the failure of uniqueness by using various schemes to enforce unique choices of fold paths. The earliest example we know of this enforcement procedure is Skora's preprint \cite{Skora:deformations} which (implicitly) uses uniquely chosen fold paths to give a new proof of the Culler--Vogtman theorem \cite{CullerVogtmann:moduli} on the contractibility of the outer space of~$F_n$, and see \cite{GuirardelLevitt:DefSpaces} for an explicit version of Skora's proof.

In this work and its sequel \cite{\STLOneTag} we shall take the opposite tack, exploiting the freedom of choice inherent in nonuniqueness of fold paths for various applications. One example of this exploit is the basic distance bound for $\FS(F_n)$ that is found in \cite[Lemma 5.2 (3)]{\FSOneTag}, and that we generalize here for $\FS(\Gamma;\A)$ in Lemma~\ref{LemmaComplexityBounds}~\pref{ItemDiameterBound}. Another such exploit occurs in Step~1 of the proof of the \emph{Two Over All Theorem} in \cite[Section 5.4]{\STLOneTag}: see also \cite[Definition 5.2]{\STLOneTag} which formalizes these kinds of exploits with the concept of \emph{fold prioritization}; and see \cite[Theorem 2.17]{\STLOneTag} for an expanded version of Stallings Fold Theorem which emphasizes the freedom of choice that allows for fold prioritization. 

\paragraph{Foldable maps and fold maps.} Consider two free splittings $\Gamma \act S,T$ and a map $f \from S \to T$ which is injective on each edgelet. For each $p \in S$ there is an induced ``derivative'' $df_p \from D_p S \to D_{f(p)} T$, which maps the initial direction of each oriented edgelet $E \subset S$ with initial vertex $p$ to the initial direction of the path $f \restrict E$. The point pre-images of the map $df_p$ are called the \emph{gates} of $f$ at $p$. Two directions in $D_p S$ are said to form a \emph{foldable turn} (with respect to $f$) if those directions are contained in the same gate at $p$; otherwise those directions form an \emph{unfoldable turn}. Because $T$ is a free splitting, the two directions of a foldable turn cannot be in the same $\Gamma$-orbit of the action on the set of directions of~$S$.

We say that the map $f \from S \to T$ is \emph{foldable} if it is injective on each edgelet and has at least~$2$ gates at each vertex. A foldable map $f \from S \to T$ is a \emph{fold} if there exist oriented natural edges $E,E' \subset S$ with common initial vertex $E \intersect E' = V$, and initial segments $e \subset E$, $e' \subset E'$, and an orientation preserving homeomorphism $h \from e \to e'$, such that the equivalence relation on $S$ defined by $x \cong y \iff f(x)=f(y)$ is generated by the relation $g \cdot x \sim g \cdot h(x)$ for all $x \in e$ and $g \in \Gamma$. Also, we say more specifically that $f$ \emph{folds the turn $\{E,E'\}$}, and still more specifically that $f$ \emph{folds the segments $e,e'$}, emphasizing here the fact that $f$ does not identify any point of $E-e$ with any point of $E'-e'$. Noting that $df_v$ identifies the initial directions of $E$ and $E'$, those directions are in different $\Gamma$-orbits. 

We briefly review the Bestvina-Feighn classification of folds given in \cite[Section~2]{BestvinaFeighn:bounding}, incorporating certain simplifications that occur in our current setting of free splittings. Our statement of the classification makes use of the following fact:
\begin{description}
\item[Claim:] Continuing with the notation of a fold $f \from S \to T$ as above, and letting $W,W'$ be the terminal endpoints of $e,e'$ respectively, the path $e \, \union \, e'$ intersects each of its nontrivial translates in a subset of the set $\{W,V,W'\}$.
\end{description}
The proof is found below. 

Applying this claim in the meanwhile, folds are classified depending on the partition of the three point set $\{W,V,W'\}$ into orbit types under the action of $\Gamma$, i.e.\ its partition into point preimages under the orbit map $\pi \from S \to S/\Gamma$ to the quotient graph of groups: 
\begin{description}
\item[$f$ has type IA --- a vee fold] if $\pi$ is one-to-one on $\{W,V,W'\}$. 
\item[$f$ has type IB --- a loop fold] if $\pi(W)=\pi(V)\ne\pi(W')$  or $\pi(W) \ne \pi(V)=\pi(W')$.
\item[$f$~has type IIIA --- a bigon fold] if $\pi(W)=\pi(W')$ (which may or may not equal $\pi(V)$). 
\end{description}
This covers all classes from \cite[Section~2]{BestvinaFeighn:bounding} that can actually occur in our current setting: the remaining classes, known as types II, IIIB, and IIIC, all involve nontrivial edge stabilizers, which do not occur in a free splitting. Nicknames refer to the appearance of projected image $\pi(e \, \union \, e')$ in $S / \Gamma$, as follows. For type IA, the ``vee'' appearance of $e \union e'$ is retained under embedding to $S/\Gamma$. For type IB, if, say, $\pi(W)=\pi(V)$ then $\pi(e)$ is a loop in $S/\Gamma$, and $\pi(e')$ is an edge sticking off of that loop. For type IIIA, $\pi(e)$ and $\pi(e')$ are two paths in $S/\Gamma$ forming a bigon, i.e. a closed loop $\pi(e) \overline{\pi(e')}$: either a simple closed loop in $S/\Gamma$, when $\pi(V) \ne \pi(W)=\pi(W')$; or a loop going exactly once each across two edges of a rank~$2$ rose subgraph of $S/\Gamma$, when $\pi(V)=\pi(W)=\pi(W')$. Note also that for type IIIA, $W$ need not be a natural vertex: in the graph of groups $S/\Gamma$, the vertex $\pi(W)=\pi(W')$ could have valence~$2$ and be labelled by the trivial group.

In all cases of the classification of folds, the extension $\Fell S \sqsubset \Fell T$ can be described explicitly. For types IA or IB: if at least one of $\Stab(W)$, $\Stab(W')$ is trivial then $\Fell S=\Fell T$ (and this is the only case of equality); otherwise $[\Stab(W)]$, $[\Stab(W')]$ are two components of $\Fell S$, and $\Fell T$ is obtained from $\Fell S$ by replacing those two with the strictly larger component $[\<\Stab(W) \union \Stab(W')\>]$. For a fold of type IIIA, let $g \in \Gamma$ be such that $g(W)=W'$. If $\Stab(W)$ is nontrivial and hence $[\Stab(W)] \in \Fell S$, then $\Fell T$ is obtained from $\Fell S$ by replacing the component $[\Stab(W)]=[g^\inv \Stab(W') g]=[\Stab(W')]$ with the strictly larger component $[\<\Stab(W)\> * \<g\>\>]$. If on the other hand $\Stab(W)$ is trivial then $\Fell T$ is obtained from $\Fell S$ by adding one new component, namely $[\<g\>]$.

\begin{proof}[Proof of the Claim] Note first that the Claim is equivalent to the statement that the open arc $(e \union e') - \{W,W'\}$ is disjoint from each of its translates by nontrivial elements of~$\Gamma$. 

Arguing by contradiction, the only way this could fail is if there existed some $\gamma \ne \Id \in \Gamma$ such that $\gamma \cdot E' = E$ or $E^\inv$ and such that $\gamma \cdot e'$ has nontrivial overlap with~$e$. In the first case where $\gamma \cdot E'=E$ we obtain an equation of oriented arcs $\gamma \cdot f(E') = f(\gamma \cdot E') = f(E)$. It follows that in $T$ the oriented arcs paths $\gamma \cdot f(E')$ and $f(E)$ have the same initial direction, contradicting that $\Gamma$ acts freely on the set of directions of~$T$ (see the opening paragraph of Section~\ref{SectionFSSucc}). 

In the second case where $\gamma \cdot E' = E^\inv$, note that $\gamma$ acts loxodromically on $S$ with $E$ as a fundamental domain for the action of $\gamma$ on its axis in~$S$. We shall prove that $\gamma^2$ fixes a direction of $T$, again contradicting that $\Gamma$ acts freely on the set of directions of~$T$. The proof requires a careful examination of the axis of $\gamma$ in~$S$. We write that axis as a concatenation of oriented natural edges $\cdots E_{-2} \, E_{-1} \, E_0 \, E_1 \, E_2 \cdots$ with $E=E_0$, such that the labelling is $\<\gamma\>$-equivariant in the sense that $\gamma^i \cdot E_0 =E_i$. The natural vertices along the axis are written $\<\gamma\>$-equivariant as $V_i$, so that the initial and terminal vertices of $E_i$ are $V_i$ and $V_{i+1}$ respectively. By composition we get a homeomorphism $e \xrightarrow{h} e' \xrightarrow{\gamma} \gamma \cdot e'$ from the initial oriented segment $e \subset E_0$ to the orientation reversed terminal segment $\gamma \cdot e' \subset E_0$, and since these segments have overlapping interiors, that overlap contains a point $M_0$ fixed by the homeomorphism. We may therefore subdivided $E$ into four oriented subpaths
$$\xymatrix{
V_0 \ar[r]^-{\kappa_0} & W'_0 \ar[r]^-{\lambda_0} & M_0 \ar[r]^-{\mu_0} & W_0 \ar[r]^-{\nu_0} & V_1
}$$
such that $e = \kappa_0 \, \lambda_0 \, \mu_0$ and $\gamma \cdot e' = \bar\nu_0 \, \bar\mu_0 \, \bar \lambda_0$, with $W=W_0$ and $\gamma \cdot W' = W'_0$. We extend these subdivisions and labellings to each edge $E_i$ in a $\<\gamma\>$-equivariant manner, as shown in the following diagram: 

$$\xymatrix@C-28pt{
V_{-1}  \ar[dr]_{\kappa_{-1}}&&&&&&&& V_1 \ar[dr]_{\kappa_1}  \\
& W'_{-1} \ar[dr]_{\lambda_{-1}}&&&&&& W_1  \ar[ur]^{\nu_0}&&   W'_{-2} \ar[dr]_{\lambda_1}\\
&& M_{-1} \ar[dr]_{\mu_{-1}}&&&&  M_0 \ar[ur]^{\mu_0}&&&& M_1 \ar[dr]_{\mu_1} \\
&&& W_{-1}  \ar[dr]_{\nu_{-1}}&&  W'_0  \ar[ur]^{\lambda_0}&&&&&&  W_1  \ar[dr]_{\nu_1} \\
&&&& V_0 \ar[ur]^{\kappa_0} &&&&&&&&  V_2
}$$
We note that $e' = h(e) = \bar\nu_{-1} \, \bar\mu_{-1} \bar \lambda_{-1}$.
By assumption the map $f$ folds $e'$ and $e$, identifying them into a single oriented segment of $T$. The two oriented subsegments $\bar\lambda_{-1}$ and $\mu_0$ are therefore identified by $f$ to a single oriented segment of $T$. Also, $f$ folds the two oriented segments $\gamma \cdot e'$ and $\gamma \cdot e$ of $S$ into a single oriented segment of $T$; the two oriented segments $\mu_0$ and $\bar\lambda_1$ of $S$ are therefore identified by $f$ into a single oriented segment of $T$. In $T$ we have verified the following equations of oriented segments:
$$\gamma^\inv \cdot f(\bar\lambda_0) = f(\gamma^\inv \cdot \bar\lambda_0) = f(\bar\lambda_{-1}) = f(\mu_0) = f(\bar\lambda_1) = f(\gamma \cdot \bar\lambda_0) = \gamma \cdot f(\bar\lambda_0) 
$$
This shows that $\gamma^2$ fixes the initial direction of the path $f(\bar\lambda_0)$ in $T$, the final contradiction that completes the proof of the claim. We remark that this fold map $f$ does exist outside of the realm of free splittings, being the composition of a bigon fold $f_1 \from S \to U$ of the segments $\kappa_0\,\lambda_0$ and $\bar\nu_{-1}\bar\mu_{-1}$ which creates a vertex $f_1(M_0) \in U$ stabilized by $\<\gamma\>$, followed by another fold of type IIA --- in the classification scheme of \cite[Section~2]{BestvinaFeighn:bounding} --- that pulls $\gamma^2$ out of the vertex~$f_1(M)$.
\end{proof}

\paragraph{Fold sequences.} A \emph{foldable sequence} is an indexed sequence of maps of free splittings of the form
$$T_I \xrightarrow{f_{I+1}} T_{I+1} \xrightarrow{f_2} \cdots \xrightarrow{f_K} T_K
$$
such that each map $f^i_j  = f_j \composed \cdots f_{i+1} \from T_i \to T_j$ is foldable, $I \le i \le j \le J$. In discussing foldable sequences we often restrict our attention to a subsequence $T_i \mapsto\cdots\mapsto T_j$ parameterized by an integer subinterval $[i,j] = \{k \in \Z \suchthat i \le k \le j\} \subset [I,J]$. A \emph{fold sequence} is a foldable sequence denoted as above, in which each of the maps $f_i \from T_{i-1} \to T_i$ is a fold map ($I < i \le K$). Lemma~\ref{ThmFoldPathExists} to follow is an instance of Stallings' fold method, and implies that every foldable sequence of $K$ foldable maps (as denoted above) can be \emph{interpolated} by a fold sequence, meaning that for each $k=1,\ldots,K$ the foldable map $T_{k-1} \mapsto T_k$ may be factored as a fold sequence, and these $K$ fold sequences may then be concatenated to obtain a fold sequence from $T_I$ to $T_K$.

\smallskip\emph{Remark on the ``gate 3 condition''.} In \FSOne, in the setting of $\Gamma=F_n$, the definition of a foldable map $f \from S \to T$ had an additional requirement, the following ``gate 3 condition'': for any vertex $p \in S$ of valence~$\ge 3$ the map $f$ has at least three gates at~$p$. Here we follow Bestvina and Feighn \cite{BestvinaFeighn:subfactor} to the extent of weakening the definition of \FSOne\ by dropping the ``gate 3 condition''. In what follows we will occasionally explain how this change effects proofs. For the most part these are desirable changes, but there are occasional exceptions; see after the statement of Lemma~\ref{LemmaCoarseRetract} for a significant exception. Two desirable effects of dropping the gate 3 condition are as follows. First, it allows for a broader collection of fold sequences in the free splitting complex; this was an important motivation for dropping that condition in \cite{BestvinaFeighn:subfactor}. Second, the interpolation of the previous paragraph does not generally work when foldable maps are required to satisfy the gate~3 condition.

\bigskip

The following commonly used relativization tool is an immediate consequence of Lemma~\ref{LemmaRealCollapse}~\pref{ItemRCFFtoFS}:

\begin{lemma}\label{LemmaGoodStabilizers}
If $S \in \FS(\Gamma;\A)$ and $T \in \FS(\Gamma)$, and if there exists a map $f \from S \to T$, then $T \in \FS(\Gamma;\A)$. \qed
\end{lemma}

\paragraph{Construction of foldable maps.} The next lemma describes tools for constructing foldable maps and for inductive construction of foldable sequences. Item~\pref{ItemFoldableExists} says foldable maps exist after perturbing the domain slightly; and \pref{ItemFoldableFactorization} says that in any factorization of foldable maps, the factors are also foldable.


\begin{lemma}
\label{LemmaFoldableExists} 
\begin{enumerate}
\item\label{ItemFoldableExists}
(cf.\ Lemma 2.4 of \FSOne) \, 
For any $S,T \in \FS(\Gamma;\A)$ there exist $S',S'' \in \FS(\Gamma;\A)$ and a diagram of maps $S \xleftarrow{g} S' \xrightarrow{f'} S'' \xrightarrow{f''} T$ such that $g$ and $f'$ are collapse maps, $f''$ is foldable, and $f'' \circ f'$ is tight. If $\Fell S \sqsubset \Fell T$ then one can take $S=S'$ and $g = \text{Id}$.
\item\label{ItemFoldableFactorization}
(cf.\ Item (3) on page 1600 on \FSOne) \, For any $S,T,U \in \FS(\Gamma;\A)$ and any maps $S \xrightarrow{g} U \xrightarrow{f} T$ taking vertices to vertices, if the composition $f \circ g \from S \to T$ is foldable then the maps $f$ and $g$ are foldable.
\end{enumerate}
\end{lemma}

\textbf{Remark.} The proofs are considerably simpler than the indicated analogue in \FSOne, due to the removal of the gate~3 condition.

\begin{proof} To prove~\pref{ItemFoldableExists}, choose a free splitting $\Gamma \act S'$ such that $S \expands S'$ and $\Fell S' \sqsubset \Fell T$: if $\Fell S \sqsubset \Fell T$ choose $S'=S$; otherwise, applying Lemma~\ref{LemmaRealCollapse}~\pref{ItemRCFFtoFS}, choose $S'$ so that $S \expands S'$ and $\Fell S' = \A \sqsubset \Fell T$. In either case we have $S' \in \FS(\Gamma;\A)$. 

There exists a map $S' \mapsto T$ which on each edge of $S$ is either constant or injective: for each $v \in S'$ choose $f(v) \in T$ in a $\Gamma$-equivariant manner so that $\Stab(v) \subgroup \Stab(f(v))$, and extend linearly over each edge; this is possible because $\Fell S'=\A \sqsubset \Fell T$. For each such map, let $S'$ be subdivided so that each edgelet maps either to a vertex or an edge of $T$. Amongst all such maps $S' \mapsto T$, choose $f \from S' \to T$ to minimize the number of orbits of edgelets of $S'$ on which $f$ is nonconstant. Factor $f$ as $S' \xrightarrow{f'} S'' \xrightarrow{f''} T$ where $f'$ collapses to a point each component of the union of edgelets on which $f$ is constant. The map $f''$ is injective on each edgelet. Applying Lemma~\ref{LemmaGoodStabilizers} we have $S'' \in \FS(\Gamma;\A)$.

To prove that $f''$ is foldable it remains to show that at each vertex $v \in S''$ the map $f''$ has at least two gates. Suppose to the contrary that $f''$ has only one gate at~$v$, let $e_1,\ldots,e_I$ be the edgelets incident vertex $v$, and let $w_1,\ldots,w_I$ be their opposite endpoints. Let $S'' \mapsto S'''$ be the quotient map obtained by collapsing to a point each of $e_1,\ldots,e_I$ and all edgelets in their orbits, so we get an induced action $\Gamma \act S'''$. Noting that $w_1,\ldots,w_I$ all map to the same point in $T$, there is an alternate description of $S'''$ as follows: remove from $S''$ the point $v$ and the interiors of $e_1,\ldots,e_I$, identify $w_1,\ldots,w_I$ to a single point, and extend equivariantly. From this description it follows that the map $f'' \from S'' \mapsto T$ induces a map $S''' \mapsto T$, and by construction the composition $S' \xrightarrow{f'} S'' \mapsto S''' \mapsto T$ is nonconstant on a smaller number of edgelet orbits than $f$ is nonconstant on. This contradicts minimality of the choice of~$f$, completing the proof of~\pref{ItemFoldableExists}. 

To prove~\pref{ItemFoldableFactorization}, noting that $f \circ g$ is injective on each edgelet $e$ of $S$, it follows that $g$ is also injective on $e$. Also, for each edgelet $e' \subset U$, because $g$ takes vertices to vertices and is injective on each edgelet of $S$ it follows that there exists an edgelet $e \subset S$ and a subsegment $\hat e' \subset e$ such that $g$ maps $\hat e'$ injectively to $e'$; since $f \circ g$ is also injective on $\hat e'$ it follows that $f$ is injective on $e'$. We may now subdivide so that each of the maps $f,g$ is simplicial. It remains to consider any vertex $v \in S$ with image vertex $w=g(v) \in U$, and to prove that $D_v g$ and $D_w f$ are both nonconstant, but this follows immediately from the fact that $D_v(f \circ g) = D_w(f) \circ D_{v}(g)$ is nonconstant. 
\end{proof}

\paragraph{Stallings fold theorem.} This theorem is stated for the case $\Gamma=F_n$ in Lemma~2.7 of \FSOne: any foldable map of free splittings $S \mapsto T$ of $F_n$ factors into a fold sequence of free splittings. Pretty much the exact same proof works for general $\Gamma$, albeit with our current different definition of ``foldability''; and if furthermore one assumes that $S \in \FS(\Gamma;\A)$ then by applying Lemma~\ref{LemmaGoodStabilizers} inductively starting with $S$, it follows that each term in the fold sequence is in $\FS(\Gamma;\A)$. Nonetheless, because of the central importance of this result, we outline the proof here.

\begin{theorem}[Stallings Fold Theorem (cf.\ Lemma 2.7 of \FSOne)] 
\label{ThmFoldPathExists}
For any $S,T \in \FS(\Gamma;\A)$, any foldable map $S \xrightarrow{f} T$ factors as a fold sequence in $\FS(\Gamma;\A)$.\qed
\end{theorem}

\begin{proof}[Proof outline] Starting with the map $S=T_0 \xrightarrow{f=g_1} T_1=T$ as the basis case, consider by induction a ``length $J$ partial fold factorization'' of $f$ having the form
$$f \from S=T_0 \xrightarrow{f_1} T_1 \xrightarrow{f_2} \cdots \xrightarrow{f_J} T_J \xrightarrow{g_J} T
$$
meaning a foldable sequence for which each of $f_1,\ldots,f_J$ is a fold. Denote 
$$g_j = g_J \circ f_J \circ \cdots\circ f_{j+1} \from T_j \to T
$$
We include in the induction hypothesis that each $f_j \from T_{j-1} \to T_j$ is a \emph{maximal} fold ($j=1,\ldots,J$), meaning that if $e,e' \subset T_{j-1}$ are initial segments of edges that witness $f_j$ being a fold, then $e,e'$ are the maximal initial segments of those edges with respect to the property that $g_j(e)=g_j(e')$ and that $(e \union e') - \{w,w'\}$ is disjoint from its translates by nontrivial elements of $\Gamma$. If $g_J$ is injective then it is a homeomorphism and the induction is complete. If $g_J$ is not injective then (like any non-injective map between trees) $g_J$ is not locally injective. It follows that there exists a vertex $v \in T_J$, and there exist distinct oriented natural edges with initial vertex $v$ and with initial directions represented by initial segments $e,e' \subset T_J$ with terminal endpoints $w \in e$, $w' \in e'$, such that $g_J(e)=g_J(e')$, and such that $(e \union e') - \{w,w'\}$ is disjoint from its translates by nontrivial elements of $\Gamma$.  Choose any such segments $e,e'$ which are maximal with respect to the listed properties, and let $f_{J+1} \from T_J \to T_{J+1}$ be the fold map that identifies $e$ and~$e'$. The map $g_J$ factors as $T_J \xrightarrow{f_{J+1}} T_{J+1} \xrightarrow{g_{J+1}} T$. By Lemma~\ref{LemmaFoldableExists}~\pref{ItemFoldableFactorization} we obtain a length $J+1$ partial fold factorization of $f$, completing the induction. 

This process must stop with an actual fold factorization of $f$, because if we start by subdividing $S$ and $T$ so that $f$ is a simplicial map, one can prove by induction on $J$ that $f_1,\ldots,f_J$ and $g_J$ are simplicial maps; it follows that the number of edgelet orbits of $T_J$ decreases strictly as $J$ increases. In the inductive step, knowing that $g_J$ is simplicial, the maximality requirement on $e,e'$ implies that each of $e,e'$ is a subcomplex of $T_J$, i.e.\ each is a union of edgelets of $T_J$. 
\end{proof}

\paragraph{The distance of a fold.} Lemma~\ref{LemmaFoldLengthTwo} to follow is a version of Lemma 2.5 of \FSOne\ which, in the narrower setting where the ``gate~3'' condition is imposed, described the values for $d(S,T)$ when $f \from S \mapsto T$ is a fold. Here we limit Lemma~\ref{LemmaFoldLengthTwo} to the proof of the inequality $d(S,T) \le 2$. After the proof we discuss the case analysis of~$f$ that yields exact values of $d(S,T)$, with a sketch of the proofs of those values. Without imposing the gate~3 condition, this case analysis is considerably more complicated, and we limit ourselves to a brief discussion of a single case.

\smallskip

To set up the statement of Lemma~\ref{LemmaFoldLengthTwo}, consider a fold $f \from S \to T$ between two free splittings $S,T$ of $\Gamma$ rel~$\A$, let $E,E' \subset S$ be a pair of oriented natural edges with initial endpoint $v$ having maximal initial segments $e \subset E$, $e' \subset E'$ that are folded by~$f$, and let $w,w'$ be the terminal endpoints of $e,e'$ respectively. By choosing proper initial segments $\eta \subset e$, $\eta' \subset e'$ that have the same $f$-images $f(\eta)=f(\eta') \subset T$, we obtain a foldable factorization
$$f \from S \xrightarrow{g} U \xrightarrow{h} T
$$
where $g$ is the ``partial'' fold defined by folding $\eta$ and $\eta'$, and so $g(\eta)=g(\eta')$. In $U$ we denote $\Gamma$-invariant subforests
$$\sigma = \Gamma \cdot g(\eta) = \Gamma \cdot g(\eta') \qquad\text{and}\qquad \tau =  \Gamma \cdot \bigl(g(e \setminus \eta) \union g(e' \setminus \eta')\bigr)
$$


\begin{lemma}
\label{LemmaFoldLengthTwo}
For any fold map $f \from S \to T$ as above there exist collapse maps $\bar g$, $\hat h$ as follows:
$$\xymatrix{
S \ar@/^1pc/[r]^{g} & U \ar@/^1pc/[r]^h \ar@/^1pc/[l]_{\bar g}^{\<\sigma\>} \ar@/_1pc/[r]^{\hat h}_{\langle\tau\rangle} & T
}$$
The distance inequality $d(S,T) \le 2$ follows, and if $d(S,T)=2$ then $d(S,U)=d(U,T)=1$.
\end{lemma}

\begin{proof} Once the existence of $\bar g$ and $\hat h$ has been established, the sentence ``The distance inequality\ldots'' follows because for any collapse map between free splittings, the corresponding vertices in $\FS(\Gamma;\A)$ have distance~$\le 1$.

To set up the construction of $\bar g$ and $\hat h$, we add some notations. Let $b,b'$ denote the respective terminal endpoints of $\eta,\eta'$. Denote $\zeta = g(\eta)=g(\eta') \subset U$, an oriented subpath of a natural edge of $U$, with initial endpoint $g(v)$ and terminal endpoint denoted $c=g(b)=g(b')$. Also, $g(e \setminus \eta)$ and $g(e' \setminus \eta')$ are oriented initial subpaths of distinct oriented natural edges of~$U$, each with initial endpoint $c$; their respective terminal endpoints are $g(w)$, $g(w')$. The free splitting $S$ decomposes into invariant subforests $S = V \union W$ where $V = \Gamma \cdot (e \union e')$ and $W = S \setminus V$; clearly $V \intersect W = \Gamma \cdot \{v,w,w'\}$, and so $V,W$ are non-overlapping. Note that for $x \in V$ and $y \in W$, if $x \ne y$ then $f(x) \ne f(y)$, and so we obtain a decomposition of $U$ into non-overlapping invariant subforests $U = f(V) \union f(W)$. 

Define $\bar g \from U \to S$ as follows. The restriction $f \from W \to f(W)$ is an equivariant homeomorphism; let $\bar g \restrict f(W) \to W$  be the inverse homeomorphism. Next comes the collapse: let $\bar g(\zeta)=v$, and extend equivariantly over $\Gamma \cdot \zeta$. Next, $g(e \setminus \eta)$ is an oriented subpath of a natural edge of $U$ with initial and terminal endpoints $c$ and $g(w)$, those points already mapping by $\bar g$ to the initial and terminal endpoints $v,w$ of $e$, so we can extend to a homeomorphism $\bar g \from g(e \setminus \eta) \to e$; then we extend equivariantly over $\Gamma \cdot g(e \setminus \eta)$. Similarly we extend to a homeomorphism $\bar g \from g(e' \setminus \eta') \to e'$ and then extend equivariantly. It follows that $\bar g$ is the desired collapse map
$$\bar g \from U \xrightarrow{\langle \Gamma \cdot \zeta \rangle} S
$$

Define $\hat h \from U \to T$ as follows. Because the restriction $g \restrict W \from W \to g(W)$ is a homeomorphism, we can define the restriction $\hat h \restrict g(W)$ to be $f \circ (g \restrict W)^\inv$. Next comes the collapse: let 
$$\hat h \bigl(g(e \setminus \eta) \union g(e' \setminus \eta') \bigr) = f(w)=f(w')
$$
and extend equivariantly. Next, $\hat h$ is already defined on the initial and terminal endpoints $g(v)$ and $c$ of $\zeta$, taking them to the initial and terminal endpoints $f(v)$ and $f(w)=f(w')$ of $f(e)=f(e')$, so we can extend to a homeomorphism $\hat h \from \zeta \to f(e)=f(e')$, and then we extend equivariantly over $\Gamma \cdot \zeta$. It follows that $\hat h$ is the desired collapse map
$$\hat h \from U \xrightarrow{\Gamma \cdot \bigl(g(e \setminus \eta) \union g(e' \setminus \eta') \bigr)} T
$$


%
%
%
%
\end{proof}

Given a fold $f \from S \to T$ factored with the notation preceding Lemma~\ref{LemmaFoldLengthTwo}, we next discuss exact computation of \hbox{$d(S,T) \in \{0,1,2\}$}, based on a case analysis of~$f$. The cases depend on the case analysis from \cite[Section~2]{BestvinaFeighn:bounding} discussed earlier, \emph{and} on subcases regarding how many of the two inclusions $e \subset E$ and $e' \subset E'$ are proper:
\begin{description}
\item[Partial fold:] Both of the inclusions $e \subset E$ and $e' \subset E'$ are proper.
\item[Full fold:] At least one of the inclusions $e \subset E$ and $e' \subset E'$ is improper. Up to swapping notations we assume that $E'=e'$, and then there are two subcases:
\begin{description}
\item[Proper full fold:] The inclusion $e \subset E$ is proper.
\item[Improper full fold:] The inclusion $e \subset E$ is also improper, $E=e$.
\end{description}
\end{description}

\medskip
\centerline{
\begin{tabular}{l | c || c | c | c | c | c } 			
		\multicolumn{2}{c||}{}
		 	& \multicolumn{2}{c}{vee fold (IA)}
				& \multicolumn{2}{|c|}{loop fold (IB)} 
						&\multicolumn{1}{|c}{bigon fold (IIIA)}
\\ 
\cline{3-6}
		\multicolumn{2}{c||}{\textbf{Values of $d(S,T)$}}		
			& partial 	
				& full	
					& \multicolumn{2}{|c|}{full}	
\\ 
\cline{5-6}
		\multicolumn{2}{c||}{}
		 	& 	&      & proper   	& improper 			
\\ 
\cline{2-7}
		&(Cases)
		 	& (a)	&  (b)    & (c)   	& (d) & (e)
\\
\hline\hline
$\text{valence}(v) = 3$ & (1)		
		& 0 & 0+1	& 0		& 0+1	& 0+1 
\\ 
\hline
$\text{valence}(v) \ge 4$ & (2) 	
		& 1+0 & 2	&  2		&2    & 2 
\end{tabular}
}
\medskip
\noindent
In each entry where $d(S,T)=1$ more information is given: ``\,$0+1$'' means that $d(S,U)=0$ and $d(U,T)=1$ ($\bar g$ is a trivial collapse); and ``$1+0$'' means $d(S,U)=1$ and $d(U,T)=0$ ($\hat h$ is a trivial collapse). To verify that the indicated values for $d(S,T)$ are \emph{upper} bounds for $d(S,T)$ one notes which of $\bar g$, $\hat h$ is a trivial collapse: in case (1), the map $\bar g$ is a trivial collapse in all subcases, whereas $\hat h$ is a trivial collapse only in subcases (a) and (c); and in case (2), $\bar g$ is never a trivial collapse, and $\hat h$ is a trivial collapse only in subcase (a).

Verifying that the indicated distance values in the table are also \emph{lower} bounds is more laborious, and we do not have any applications of it, so we provide just a sketch for some cases. 

For the cases where $d(S,T)=1$ we need only rule out equivalence of $S$ and $T$: in cases (1b) and (1d), $S$ has one more $\Gamma$-orbit of natural edges than $T$; in case (1e) the bigon represents an element of $\Gamma$ that is loxodromic in $S$ but elliptic in $T$; and in case (2a), $T$ has one more orbit of natural edges than $S$ has. 

For the cases where $d(S,T)=2$, to verify the lower bound $d(S,T) \ge 2$ one can use the \emph{translation spectrum} of a free splitting $U$ of $\Gamma$, namely the function that assigns to each $\gamma \in \Gamma$ its translation length function $\tau_U(\gamma)$: if $\gamma$ is elliptic in $U$ then $\tau_U(\gamma)=0$; whereas if~$\gamma$ is loxodromic in $U$ then $\tau_U(\gamma)$ is the number of $\gamma$-orbits of natural edges along the axis of~$\gamma$ in $U$. The translation spectrum is a complete invariant of free splittings, meaning that two free splittings $S$ and $T$ are equivalent if and only if $\tau_S=\tau_T$ \cite{CullerMorgan:Rtrees}. Furthermore, if $S$ collapses to $T$ then $\tau_S(\gamma) \ge \tau_T(\gamma)$ for all $\gamma \in \Gamma$. One can therefore prove the inequality $d(S,T) \ge 2$ by finding $\gamma,\delta \in \Gamma$ such that $\tau_S(\gamma) < \tau_T(\gamma)$ and $\tau_S(\delta) > \tau_T(\delta)$. 

Consider for example Case (2b), a full vee fold $f \from S \to T$ with a fold vertex $v$ of valence~$\ge 4$; we use the notation $e \subset E$ and $e' \subset E'$ as described preceding the lemma. We consider only the subcase that $f$ is a proper fold with $e \ne E$ and $e'=E'$, and so $\Fell S = \Fell T$, hence $\Gamma$ has the same set of loxodromic elements for $S$ and for $T$. Consider a loxodromic element $\gamma$ with axes $L_S(\gamma) \subset S$ and $L_T(\gamma) \subset T$. The axis $L_T(\gamma)$ is obtained from $f(L_S(\gamma))$ by operations of cancellation and subdivision, one such operation for each $\gamma$-orbit of a translate of $e$ or its inverse in $L_S(\gamma)$: associated to each $\gamma$-orbit of a translate $\bar e' e$ or its inverse there is a cancellation; and associated to each $\gamma$-orbit of $e'' e$ or its inverse such that $e''$ is not a translate of $\bar e'$, there is a subdivision. The value of $\tau_T(\gamma)$ is obtained from $\tau_S(\gamma)$ by subtracting one for each $\gamma$-orbit of cancellation and adding one for each $\gamma$-orbit of subdivision. So one must find $\gamma$ having more $\gamma$-orbits of cancellations than subdivisions in $L_S(\gamma)$ in order to guarantee that $\tau(\gamma;S) < \tau(\delta;T)$; and one must find $\delta$ having more $\gamma$-orbits of subdivisions than cancellations in $L_S(\gamma)$ in order to guarantee that $\tau(\delta;S) > \tau(\delta;T)$. The method for finding such $\gamma,\delta$ has its own subsubcases depending on the number of connected components of $(S/\Gamma) \setminus ((e \union e')/\Gamma)$. The easiest subsubcase is when there is just one connected component: one can then find $\gamma$ having exactly one cancellation and no subdivisions; and one can find $\delta$ having exactly one subdivision and no cancellations. With more connected components of $(S/\Gamma) \setminus ((e \union e')/\Gamma)$, the constructions of $\gamma$ and $\delta$ become more complicated, but here we stop, leaving details to the interested reader.

\paragraph{Fold paths.} A sequence of vertices $(T_i)_{i \in I}$ in $\FS(\Gamma;\A)$, parameterized by some subinterval $I \subset \Z$, is called a \emph{fold path} if for each $i-1,i \in I$ there exists a map $f_i \from T_{i-1} \to T_i$ such that the sequence of maps $\cdots \xrightarrow{f_{i-1}} T_{i-1} \xrightarrow{f_i} T_i \xrightarrow{f_{i+1}} \cdots$ is a fold sequence.

\begin{corollary}\label{CorollaryConnected}
$\FS(\Gamma;\A)$ is connected, and fold paths form an almost transitive sequence of paths in $\FS(\Gamma;\A)$. More precisely, for any $S,T \in \FS(\Gamma;\A)$ there is a fold path starting at distance~$\le 2$ from $S$, making jumps of distance~$\le 2$, and ending at $T$. Furthermore, any fold path can be interpolated to obtain a fold path with jumps of distance~$\le 1$.
\end{corollary}

\begin{proof} The ``more precisely'' sentence follows by combining Lemma \ref{LemmaFoldableExists}, Theorem~\ref{ThmFoldPathExists} and Lemma~\ref{LemmaFoldLengthTwo}~(cf.\ remark following Theorem~3.1 of \FSOne). For the ``furthermore'' sentence, consider any fold path $T_0 \xrightarrow{f_1} T_1 \xrightarrow{f_2} \cdots \xrightarrow{f_J} T_J$, and so $d(T_{j-1},T_j) \le 2$ for all $j$. Given any $j=1,\ldots,J$ such that $d(T_{j-1},T_j) = 2$, by applying Lemma~\ref{LemmaFoldLengthTwo} we obtain a foldable factorization $f_J \from T_{j-1} \xrightarrow{f'_J} T'_j \xrightarrow{f''_J} T_j$ such that $d(T_{j-1},T'_j) = d(T'_j,T_j)=1$. By applying Lemma~\ref{LemmaFoldableExists}~\pref{ItemFoldableFactorization} for each such $j$ one at a time, we 
can interpolate $f_J$, replacing it by $f'_j$ followed by $f''_J$, to obtain the desired fold path with jumps of distance~$\le 1$. 
\end{proof}

\subsection{$\FS(\Gamma;\A)$ in low complexity cases}
\label{SectionFSLow}
\marginparLee{I do not like leaving this conjecture buried here \ldots}
Using the results of Section~\ref{SectionFoldSequences} we now give a complete description of free splitting complexes $\FS(\Gamma;\A)$ in two low complexity cases where $\FS(\Gamma;\A)$ is a very specific finite diameter tree. In all remaining cases we conjecture that $\FS(\Gamma;\A)$ is of infinite diameter, indeed that the action of $\Out(\Gamma;\A)$ on $\FS(\Gamma;\A)$ has loxodromic elements. 

The first low complexity case is when $\DFF(\A)=0$, which occurs if and only if $\A =\{[A_1],[A_2]\}$ and $\Gamma = A_1 * A_2$. The second is when $\DFF(\A)=1$ and $\abs{\A} \le 1$, which occurs if and only if $\A=\{[A]\}$ and $\Gamma = A * Z$ where $Z$ is infinite cyclic. We consider these cases separately in Propositions~\ref{PropOneCoedgeTwoComps} and~\ref{PropOneCoedgeOneComp} to follow. 

\begin{proposition}\label{PropOneCoedgeTwoComps}
Suppose that $\DFF(\A)=0$, equivalently $\A = \{[A_1],[A_2]\}$ has a realization of the form $\Gamma = A_1 * A_2$. In this case $\FS(\Gamma;\A)$ is a single point, corresponding to the Bass-Serre tree of the free factorization $\Gamma = A_1 * A_2$.
\end{proposition}

\textbf{Remark.} In the case that $A_1,A_2$ are free of finite rank, this proposition is contained in \BookOne\ Corollary 3.2.2. The proof here is an extension of that proof.

\begin{proof} We first note the fact that for any nonfull free factor system $\A'$ of $\Gamma$, if $\A \sqsubset \A'$ then $\A=\A'$. It follows that for any free splitting $\Gamma \act T$ rel~$\A$, we have $\A=\Fell T$. We next note the fact that since $S$ has one edge orbit, if $S \collapses S''$ then $S$ and $S''$ are equivalent.

Given a vertex $T \in \FS(\Gamma;\A)$ we must prove that $S,T$ are equivariantly homeomorphic.  Applying Lemma~\ref{LemmaFoldableExists} and the facts noted above, it follows that there exists a foldable map $f \from S \to T$. Applying Lemma~\ref{ThmFoldPathExists}, there exists a fold sequence from $S$ to $T$. However, at each vertex $v \in S$ all of the directions at $v$ are in the same orbit of the subgroup $\Stab(v)$, because the quotient graph of groups $S / \Gamma$ has two vertices each of valence~$1$. A fold map cannot fold two directions in the same orbit. Thus the fold sequence from $S$ to $T$ has length zero and $S$, $T$ are equivalent.
\end{proof}

For describing the next case, we need a few definitions.

Consider a free product $\Gamma = A * Z$ where $Z = \<z\>$ is infinite cyclic, and consider the free factor system $\A = \{[A]\}$. Define a monomorphism $A \inject \Out(\Gamma;\A)$ denoted $a \mapsto \phi_a$, where $\phi_a$ is represented by $\Phi_a \in \Aut(\Gamma)$ which is characterized by $\Phi_a \restrict A = \Id$, $\Phi_a(z) = za$. Noting that the subgroups $\<z\>$ and $\<za\>$ are conjugate in $\Gamma$ if and only if $a$ is trivial, it follows that the homomorphism $a \mapsto \phi_a$ is injective.

In any $1$-complex~$X$, a \emph{star point} is a $0$-cell $v$ such that for each component $A$ of $X - v$, the closure of $A$ in $X$ equals $A \union \{v\}$ and is an arc called a \emph{beam} of $X$ (we do not require a beam to consist of a single edge). If a star point exists then $X$ is a \emph{star graph}. 

\begin{proposition}\label{PropOneCoedgeOneComp}
Suppose that $\DFF(\A)=1$ and $\abs{\A} = 1$, equivalently $\A=\{[A]\}$ has a realization of the form $\Gamma = A*Z$ where $Z=\<z\>$ is infinite cyclic. In this case $\FS(\Gamma;\A)$ is a star graph with star point $T$ such that each beam has the form $T \expands S \collapses R$ with quotient graphs of groups as follows (assuming natural cell structures):
\begin{description}
\item[Loop type:] $T / \Gamma$ has one vertex labelled $A$ and one edge forming a loop with both endpoints at the vertex. 
\item[Sewing needle type:] $S/\Gamma$ has two vertices, one labelled $A$ and the other of valence~$3$ labelled with the trivial group, with one edge connecting the $A$ vertex to the valence~$3$ vertex, and one edge forming a loop with both ends at the valence~$3$ vertex.
\item[Edge type:] There exists a realization $\Gamma = A*Z$ of $\A$ such that $R/\Gamma$ has two vertices, one labelled $A$ and the other labelled by the infinite cyclic group $Z$, and one edge connecting the two vertices.
\end{description}
Furthermore, under the monomorphism $A \inject \Out(\Gamma;\A)$ given by $a \mapsto \phi_a$ described above, the induced action $A \act \FS(\Gamma;\A)$, is free and transitive on the set of beams, allowing beams to be enumerated as follows: 
\begin{itemize}
\item Every free factorization of the form $\Gamma = A * Z'$ satisfies $Z' = \<za\>$ for a unique $a \in A$.
\item There are bijections: $\{$beams of $\FS(\Gamma;\A)\} \leftrightarrow \{$edge-type free splittings rel~$\A\} \leftrightarrow \{$free factorizations $\Gamma = A * \<za\>$, $a \in A\} \leftrightarrow A$.
\end{itemize}
Since $A$ is nontrivial, there are at least two beams and the diameter of $\FS(\Gamma;\A)$ equals~$4$.
\end{proposition}

\textbf{Remark.} As was the case for Proposition~\ref{PropOneCoedgeTwoComps}, the proof of Proposition~\ref{PropOneCoedgeOneComp} is an elaboration upon the proof of Corollary~3.2.2 of \BookOne\ which is concerned with the case that $\Gamma$ is free of some finite rank $n$ and $A$ is free of rank~$n-1$.

\begin{proof} The proof uses Bass-Serre theory \cite{ScottWall} and the Bestvina--Feighn classification of folds \cite{BestvinaFeighn:bounding} that was reviewed earlier.

For any free splitting $\Gamma \act U$ representing a $0$-simplex of $\FS(\Gamma;\A)$, the free factor system $\Fell U$ satisfies either $\Fell U=\A$, or $\Fell U = \A \union \{[Z]\}$ for some free factorization $\Gamma = A * Z$ with $Z$ infinite cyclic. It follows that $U$ has a unique vertex $v(U)$ such that $\Stab(v(U)) = A$.

First we prove existence of a free splitting rel~$\A$ of loop type. From the hypotheses on $\A$ it follows that there exists a free factorization $\Gamma = A * Z$ with $Z$ infinite cyclic, the Bass-Serre tree of which is an edge type free splitting $\Gamma \act R$. Expanding $R$ by blowing up the $Z$ vertex of $R/\Gamma$ into a loop one gets a free splitting $\Gamma \act S$ of sewing needle type. Collapsing the non-loop edge of $S/\Gamma$ one gets a free splitting of loop type.

Fix now a loop type free splitting $\Gamma \act T$ rel~$\A$. Consider any free splitting $\Gamma \act U$ rel~$\A$. Since $\Fell T \sqsubset \Fell U$ and $T$ has one edge orbit it follows, as in the proof of Proposition~\ref{PropOneCoedgeTwoComps}, that there is a foldable map $f \from T \to U$. Note that $f(v(T))=v(U)$. The derivative $d_{v(T)} \from D_{v(T)} T \to D_{v(U)} U$ is either one-to-one or two-to-one. 

In the first case where $d_{v(T)}$ is one-to-one, the map $f$ is a homeomorphism and $T \equiv U$, just as in the proof of Proposition~\ref{PropOneCoedgeTwoComps}. 

In the second case where $d_{v(T)}$ is two-to-one, consider a fold sequence that factors the map $f$, given by $T = T_0 \xrightarrow{f_1} T_1 \to \cdots \xrightarrow{f_K} T_K = U$ with $K \ge 1$. Using the Bestvina-Feighn classification of fold types described earlier, the folds in this sequence are as follows. If $K=1$ then $f \from T \to U$ is either of type IA and $U$ is of sewing needle type, or $f$ is of type IIIA and $U$ is of edge type. If $K \ge 2$ then each of $f_1,\ldots,f_{K-1}$ is of type IA, and each of $T_1,\ldots,T_{K-1}$ is of sewing needle type; the final fold $f_K$ is either of type IA and $T_K=U$ is also of sewing needle type, or $f_K$ is of type IIIA and $T_K=U$ is of edge type. 

Note in particular that if $U$ is of loop type then $d_{v(T)}$ is not two-to-one, and so any foldable map $T \mapsto U$ is a homeomorphism and $T \equiv U$, so there is a unique loop-type $0$-cell in $\FS(F_n;\A)$.

We have proved that each $0$-cell in $\FS(F_n;\A)$ is represented by a free splitting of one of the three types described. Each $0$-cell of sewing needle type collapses to exactly two other $0$-cells, namely the unique one of loop type and one other of edge type. Each $0$-cell of edge type expands to exactly one other $0$-cell, that being of sewing needle type. This proves that the unique loop type $0$-cell is a star point and each beam is as described. 

To prove the ``Furthermore'' clause, by applying Proposition~\ref{PropOneCoedgeTwoComps} to any free factor system of the form $\{[A],[Z']\}$ where $Z'$ is a cofactor of $\A$, it follows that there is an $\Out(\Gamma;\A)$-equivariant bijection between the set of edge-type free splittings rel~$\A$ and the set of conjugacy classes of cofactors of realizations of $\A$. Each realization of $\A$ is conjugate in $\Gamma$ to one of the form $\Gamma = A * Z'$. It therefore suffices to show that each realization of the latter form is conjugate to a unique one of the form $A * \<za\>$, $a \in A$. Uniqueness follows from the observation that $\<za\>$ is conjugate to $\<zb\>$ if and only if $a=b$. To prove existence, pick a generator $Z' = \<z'\>$. The two free factorizations $\Gamma = A * \<z\> = A * \<z'\>$ determine two loop type free splittings rel~$\A$, namely the Bass-Serre trees of the two HNN extensions of $A$ over the trivial group, one with stable letter $z$ and the other with stable letter~$z'$. But we proved above that any two loop type free splittings rel~$\A$ are equivalent, and it follows that $z' = b z^{\pm 1} c$ for some $b,c \in A$. After possibly replacing $z'$ with its inverse we have $z' = bzc$, which is conjugate to $zcb$, and taking $a=cb$ we are done.
\end{proof}

\subsection{Combing}
\label{SectionCombing}
Consider a foldable map $f \from S \to T$ of free splittings of $\Gamma$ rel~$\A$. Given a nondegenerate subgraph $\sigma_T \subset T$, its \emph{pullback} is the nondegenerate subgraph $\sigma_S \subset S$ obtained from $f^\inv(\sigma_T)$ by removing degenerate components. In what follows, we will often assume without explicit mention that invariant subgraphs are nondegenerate, particularly in contexts where the pullback operation is used. 

Following \FSOne\ Section 4.1 (but using the current definition of foldable sequences), a \emph{combing rectangle} in $\FS(\Gamma)$ is defined to be a commutative diagram of free splittings of $\Gamma$ of the form
$$\xymatrix{
S_I \ar[r]^{f_{I+1}} \ar[d]_{[\sigma_I]}^{\pi_I} 
 & \cdots \ar[r]^{f_{i-1}} 
 & S_{i-1} \ar[d]_{[\sigma_{i-1}]}^{\pi_{i-1}} \ar[r]^{f_i} 
 & S_i \ar[d]_{[\sigma_i]}^{\pi_i} \ar[r]^{f_{i+1}}
 & \cdots \ar[r]^{f_J}  
 & S_J \ar[d]_{[\sigma_J]}^{\pi_J}  \\
T_I \ar[r]^{g_{I+1}}                                                
 & \cdots \ar[r]^{g_{i-1}} & T_{i-1}                                   \ar[r]^{g_i} 
 & T_i                                 \ar[r]^{g_{i+1}}
 & \cdots \ar[r]^{g_J} 
 & T_J
}
$$
where the top and bottom rows are foldable sequences, each vertical arrow $\pi_i \from S_i \to T_i$ is a collapse map with indicated collapse forest $\sigma_i$, and each $\sigma_i$ is the \emph{pullback} of $\sigma_J$ under the map $f^i_J$. If $\A$ is a free factor system of $\Gamma$ and $S_i,T_i \in \FS(\Gamma;\A)$ for all $i$ then we also say this is a combing rectangle \emph{in $\FS(\Gamma;\A)$}.

Denoting a combing rectangle in shorthand as $(S_i;T_i)_{I \le i \le J}$, two given combing rectangles $(S_i;T_i)_{I \le i \le J}$ and $(S'_i;T'_i)_{I' \le i \le K'}$ are said to be equivalent if $K'-K=I'-I=D$ and if there are equivariant homeomorphisms $S^{\vphantom{\prime}}_{i} \leftrightarrow S'_{i+D}$ and $T^{\vphantom{\prime}}_{i} \leftrightarrow T'_{i+D}$ making all resulting squares commute.


\begin{lemma}[Relative combing by collapse, cf.\ \FSOne\ Proposition~4.3] 
\label{LemmaCombingByCollapse}
For any combing rectangle, if its top row is in $\FS(\Gamma;\A)$ then so is its bottom row. For any foldable sequence $S_I \mapsto\cdots\mapsto S_J$ and any collapse $\pi_J \from S_J \to T_J$ in $\FS(\Gamma;\A)$, there exists a combing rectangle with the given top row and right edge, and that combing rectangle is unique up to equivalence.
\end{lemma}

\begin{proof} The first sentence follows from Lemma~\ref{LemmaGoodStabilizers}. The existence statement in second sentence is proved in the case $\Gamma=F_n$, $\A=\emptyset$ in \FSOne\ Proposition~4.3, that proof works without change to prove existence in our present setting, \emph{and} the proof also gives uniqueness. In outline: define $\sigma_i \subset S_i$ uniquely as required by the definition; use $\sigma_i$ to uniquely define the collapse map $S_i \xrightarrow{[\sigma_i]} T_i$; check that there is a well-defined induced map $g_i \from T_{i-1} \to T_i$ which uniquely defines the bottom row; and then check that the bottom row is a foldable sequence. 
\end{proof}

\begin{lemma}[Relative combing by expansion, cf.\ \FSOne\ Proposition~4.4]
\label{LemmaCombingByExpansion}
For any foldable sequence $T_I \mapsto \cdots \mapsto T_J$ and any collapse map $\pi_J \from S_J \to T_J$ in $\FS(\Gamma;\A)$ there exists a combing rectangle in $\FS(\Gamma;\A)$ with the given bottom row and right edge, and that combing rectangle is unique up to equivalence.
\end{lemma}

\begin{proof} The existence proof in the case $\Gamma = F_n$, $\A=\emptyset$ is found in \FSOne, Proposition 4.4, ``Step~1'' and ``Preparation for Step 2'' (the further work in Step 2 of that proof is entirely concerned with establishing the gate~3 condition for the $S$ row, and so is not relevant to us here). Following that proof, consider the fiber product of the two free splittings $\Gamma \act T_i$, $\Gamma \act S_J$ with respect to the two $\Gamma$-equivariant maps $T_i \mapsto T_J$, $S_J \mapsto T_J$. This fiber product is the subset of the Cartesian product $T_i \times S_J$ consisting of ordered pairs $(x,y)$ such that the image of $x$ in $T_J$ equals the image of $y$ in $T_J$. It is a simplicial tree on which $\Gamma$ acts with trivial edge stabilizers, and we define $S_i$ to be the minimal subtree for that action. The two projection maps of the Cartesian product induce maps $\pi_i \from S_i \to T_i$ and $h^i_J \from S_i \to S_J$. Exactly as in ``Step 1'', the map $\pi_i$ is a collapse map which collapses a subforest $\sigma_i \subset S_i$, and $\sigma_i$ is the set of nondegenerate components of $(h^i_J)^\inv(\sigma_J)$. And exactly as in ``Preparation for Step 2'', the map $h^i_J$ is injective on edgelets and has $\ge 2$ gates at each vertex, and so $h^i_J$ is foldable according to our current definition. We thus have a combing diagram in $\FS(\Gamma)$, and we need to check that $S_i \in \FS(\Gamma;\A)$. For each subgroup $A \subgroup \Gamma$ such that $[A] \in \A$, since $A$ fixes unique points of $T_k$ and of $S_J$ it follows that $A$ fixes a unique point of the fiber product tree; since $A$ is nontrivial, that fixed point is in the minimal subtree $S_i$, and so $S_i \in \FS(F_n;\A)$.

Uniqueness follows by noticing that for any combing rectangle, the maps $\pi_i \from S_i \to T_i$ and $f^i_J \from S_i \to S_J$ embed $S_i$ in the fiber product tree of the two maps $T_i \mapsto T_J$ and $S_J \mapsto T_J$. Since the action of $\Gamma$ on $S_i$ is minimal it follows that $S_i$ is identified with the minimal subtree of the fiber product tree, and under this identification the maps $\pi_i,f^i_J$ are identified with the restrictions of the projection maps of the Cartesian product. The desired uniqueness property is an immediate consequence.
\end{proof}


\subsection{Complexity of nondegenerate subgraphs of free splittings}
\label{SectionSubgraphComplexity}

This section is concerned with an important technical underpinning of the proof of hyperbolicity. The key idea is that as one moves along a fold path, one studies the concept of a ``pullback sequence'' along that path, meaning a sequence of nondegenerate subgraphs, one in each free splitting along that fold path, each of which is the pullback of the next one with respect to the given fold map. We focus on how the topology of the subgraph varies along a pullback sequence, and we use numerical measurements of ``complexity'' to measure this change of topology. These subgraphs are just forests, of course: the only aspects of their topology that concern us are their component sets and the action of $\Gamma$ on those sets; and the only aspects of change of topology that we consider will be the $\Gamma$-equivariant maps on component sets induced by foldable maps.

The way the results of this section will be applied in what follows is to use upper and lower bounds on the change of complexity along fold paths to obtain information about upper and lower bounds on distance in $\FS(\Gamma;\A)$ along folds paths; see the discussion just below regarding the definition of complexity.

\subsubsection{Definition of complexity.} Consider a free factor system~$\A$ of $\Gamma$, a free splitting $\Gamma \act T$ rel~$\A$, and a nondegenerate subgraph $\beta \subset T$. We shall define a positive integer valued \emph{complexity} denoted $C(\beta)$ which is a sum of several terms. This complexity $C(\beta)$ will be dominated by a single term $C_1(\beta)$ called the \emph{component complexity} of $\beta$, defined to be the number of $\Gamma$-orbits of components of $\beta$. We will see in Lemma~\ref{LemmaComplexitySummandBounds} that the difference \hbox{$C(\beta)-C_1(\beta)$} is a non-negative integer bounded above by a constant depending only on~$\abs{\A}$ and~$\corank(\A)$. 

The definition of complexity is designed so that various upper and lower bounds on $C(\beta)$ can be used to obtain topological and metric conclusions. The most important of these conclusions are as follows:
\begin{itemize}
\item From upper bounds on complexity we obtain upper bounds on diameters along fold paths: see Lemma~\ref{LemmaComplexityBounds}~\pref{ItemDiameterBound} and Lemma~\ref{LemmaPreimageComplexityBound}, and applications of those lemmas in later sections. Underlying these diameter bounds is the key technical result Lemma~\ref{SublemmaNew}. The terms forming the difference $C(\beta)-C_1(\beta)$ are designed specifically to make Lemma~\ref{SublemmaNew} work.
\item From lower bounds on $C(\beta)$ we obtain lower bounds on $C_1(\beta)$, from which we deduce that some component of $\beta$ is an arc in the interior of a natural edge of~$T$: see Lemma~\ref{LemmaCompArc} and Fact~\ref{PropFSUProps}~\pref{ItemFSUBoundsUpsComp}. Ultimately this leads to lower bounds on diameter along fold paths, as expressed in Theorem~\ref{TheoremRelFSUParams}.
\end{itemize}
One may formally view the proof of hyperbolicity of $\FS(\Gamma;\A)$ as a game in which upper and lower bounds on complexity are played against each other, to obtain various upper and lower bounds on distance as needed for proving hyperbolicity.

The complexity $C(\beta)$ is defined by adding four non-negative integer summands:
$$C(\beta) = C_1(\beta) + C_2(\beta) + C_3(\beta) + C_4(\beta)
$$
These summands are each tailored to cases in the proof of Lemma~\ref{SublemmaNew}. For defining them, recall the free splitting $T/\beta$ obtained from $T$ by collapsing to a point each component of $\beta$. The free factor system $\Fell(T/\beta)$ decomposes into two subsets $\Fell(T/\beta) = \F[\beta] \, \disjunion \, \F[T-\beta]$ as follows: given $[B]\in\Fell(T/\beta)$, put $[B]$ in $\F[\beta]$ if $B$ stabilizes some component of $\beta$, and put $[B]$ in $\F[T-\beta]$ if $B$ stabilizes some point of $T-\beta$. 
\begin{itemize}
\item Define $C_1(\beta) = \abs{\beta / \Gamma}$, the number of components of the orbit space $\beta/\Gamma$, equal to the number of $\Gamma$-orbits of components of $\beta$.
\item Define $C_2(\beta) = \DFF(\Fell(T/\beta))$. 
\item Define $C_3(\beta)$ to be the number of components $[A] \in \A$ satisfying the following: the component of $\Fell(T/\beta)$ containing $[A]$ is in the set $\F[T-\beta]$; equivalently $A$ stabilizes some vertex of $T-\beta$.
\end{itemize}
For defining $C_4(\beta)$, first apply Lemma~\ref{LemmaExtension} using $\A$ and $\A'=\Fell(T/\beta)$, with the following conclusions: the components of $\Fell(T/\beta)$ can that do not contain any component of $\A$ can be listed as $[B_1],\ldots,[B_N]$ where each of $B_1,\ldots,B_N$ is free of finite rank and their free product $B_1,\ldots,B_N$ is a free factor of a cofactor of a realization of $\A$ (in the notation of Lemma~\ref{LemmaExtension} these components are $[A'_{J+1}],\ldots,[A'_K]$). Up to re-indexing there exists $M \in \{0,\ldots,N\}$ so that $[B_1],\ldots,[B_M] \in \F[T-\beta]$ and $[B_{M+1}],\ldots,[B_N] \in \F[\beta]$. Thus $[B_1],\ldots,[B_M]$ are precisely the components of $\Fell(T/\beta)$ that do not contain a component of $\A$ and whose representative subgroups $B_1,\ldots,B_M$ each fix some point of $T-\beta$. Since each $B_m$ is free of finite rank, it follows that the set $\Fell(T/\beta) - \{[B_1],\ldots,[B_M]\}$ is still a free factor system, and it is still true that $\A \sqsubset  \Fell(T/\beta) - \{[B_1],\ldots,[B_M]\}$. Define
$$C_4(\beta) = \corank\biggl(\Fell(T/\beta) - \bigl\{[B_1],\ldots,[B_M]\bigr\}\biggr) = \corank
\bigl(\Fell(T/\beta)\bigr) 
+ \sum_{m=1}^M \rank(B_m)
$$
In the following lemma the inequalities in conclusion~\pref{ItemCSBIneq} are immediate. Conclusion~\pref{ItemCSBEq} will be applied in the sequel \cite{HandelMosher:RelComplexHypII}.

\begin{lemma}\label{LemmaComplexitySummandBounds}
For any free splitting $T$ of $\Gamma$ rel~$\A$ and $\Gamma$-invariant proper subgraph $\beta$,
\begin{enumerate}
\item\label{ItemCSBIneq}
The three summands $C_2(\beta), C_3(\beta), C_4(\beta)$ have the following bounds:
\begin{align*}
0 \le C_2(\beta) &\le  2 \, \corank(\A) + \abs{\A} - 1 \quad\text{(by Lemma~\ref{LemmaFFSNorm}~\pref{ItemFFSxIneq})} \\
0 \le C_3(\beta) &\le \abs{\A} \\
0 \le C_4(\beta) &\le \corank(\A)  \quad\text{(by Corollary~\ref{PropCorankIneq})} 
\end{align*} 
and hence $C_2(\beta) + C_3(\beta) + C_4(\beta) \le 3 \corank(\A) + 2 \abs{\A} - 1$.
\item\label{ItemCSBEq}
Consider free splittings $\Gamma \act_\alpha T$ and $\Gamma \act_{\alpha'} T'$ of $\Gamma$ rel~$\A$ and $\phi \in \Out(\Gamma;\A)$ such that $[T] \cdot \phi = [T']$, as witnessed by a simplicial isomorphism $h \from T \to T'$ which is equivariant with respect to the actions $\Gamma \act_{\alpha \circ \Phi} T$ and $\Gamma \act_{\alpha'} T'$ (where $\Phi \in \Aut(\Gamma;\A)$ represents $\phi$; see Section~\ref{SectionRelOut}). For any $\Gamma$-invariant proper subgraphs $\beta \subset T$ and $\beta' \subset T'$ such that $h(\beta)=\beta'$ we have $C_i(\beta)=C_i(\beta')$ for $i=1,\ldots,4$, and hence $C(\beta)=C(\beta')$.
\end{enumerate}
\end{lemma}

\begin{proof} We need only prove~\pref{ItemCSBEq}. The map $h$ induces a bijection from $\Gamma$-orbits of components of $\beta$ to $\Gamma$-orbits of components of $\beta'$, hence $C_1(\beta)=C_1(\beta')$. 

Note that $[T/\beta] \cdot \phi = [T'/\beta']$, hence $\phi^\inv \cdot \Fell(T/\beta) = \Fell T'/\beta'$ (by Lemma~\ref{LemmaFofTEquivariance}). Since $\Out(\Gamma;\A)$ acts on free factor systems rel~$\A$ preserving the partial order $\sqsubset$, it follows that $C_2(\beta) = \DFF(\Fell(T/\beta)) = \DFF(\Fell T'/\beta') = C_2(\beta')$. 

Noting that $\Phi$ permutes those free factors $A \subgroup \Gamma$ such that $[A] \in \A$, and that $A$ stabilizes a vertex of $T \setminus \beta$ if and only if $\Phi^\inv(A)$ stabilizes a vertex of $T' \setminus \beta'$, it follows that $C_3(\beta)=C_3(\beta')$.

Finally, let $[B_1],\ldots,[B_M]$ be as in the definition of $C_4(\beta)$, namely those components of $\Fell(T/\beta)$ represented by cofactors of $\A$ that are stabilizers of vertices of $T-\beta$. Letting $B'_i = \Phi^\inv(B_i)$, and hence $[B'_i] = \phi^\inv \cdot [B_i]$, it follows that $[B'_1],\ldots,[B'_M]$ are those components of $\Fell T'/\beta'$ represented by cofactors of $\A$ that are stabilizers of vertices of $T'-\beta'$. Since $B_i$, $B'_i$ are isomorphic, they have the same rank, and hence $C_4(\beta)=C_4(\beta')$.
\end{proof}

\subsubsection{Consequence of a lower bound on complexity.} Lemma~\ref{LemmaCompArc} to follow gives a very simple topological consequence for a specific lower bound on the subgraph complexity. Further consequences of that lower bound are derived later in Proposition~\ref{PropFSUProps}~\pref{ItemFSUBoundsUpsComp}, and those consequences will play an important role in the central arguments of Section~\ref{SectionFFRelAHyp}, particularly in the statement and proof of Proposition~\ref{PropMMTranslation} where that constant is denoted~$b_1 = 5 \corank(\A) + 4 \abs{\A} - 3$. 

\begin{lemma}\label{LemmaCompArc}
For any free splitting $\Gamma \act T$ rel~$\A$, and for any $\Gamma$-invariant subgraph $\beta \subset T$, if \, $C(\beta) > 5 \corank(\A) + 4 \abs{\A} - 3$ then some component of $\beta$ is an arc contained in the interior of a natural edge of $T$.
\end{lemma}

\begin{proof} Combining the hypothesis with Lemma~\ref{LemmaComplexitySummandBounds} it follows that 
$$C_1(\beta) > 2 \corank(\A) + 2\abs{\A} - 2
$$
The right hand side is the maximal number of natural vertex orbits amongst all free splittings of $\Gamma$ rel~$\A$, according to Proposition~\ref{PropGenericFS}~\pref{ItemVMax}. So $\beta$ has more than that number of component orbits, and hence one of those orbits must be disjoint from the natural vertices of~$T$. Each component in that orbit is therefore contained in the interior of some natural edge.
\end{proof}

\subsubsection{Complexity change under a foldable map.} We now turn to a study of subgraph complexity along a foldable sequence, starting with its behavior under a single foldable map.

\begin{lemma}[Monotonicity properties of complexity]
\label{SublemmaNew}
(c.f.\ \FSOne\ Sublemma 5.3) Let $S,T$ be free splittings of~$\Gamma$ rel~$\A$, let $f \from S \to T$ be a foldable map, let $\beta_T \subset T$ be a proper $\Gamma$-invariant subgraph, and let $\beta_S \subset T$ be the pullback of $\beta_T$ under $f$. Then we have $C_i(\beta_S) \ge C_i(\beta_T)$ for $i=1,2,3,4$, and hence $C(\beta_S) \ge C(\beta_T)$. Furthermore, if $C(\beta_S)=C(\beta_T)$ then $f$ induces a bijection between the set of components of $\beta_S$ and the set of components of $\beta_T$.
\end{lemma}

\begin{proof} The $\Gamma$-equivariant surjection $f \from \beta_S \to \beta_T$ induces a $\Gamma$-equivariant surjection of component sets $f_* \from \pi_0(\beta_S) \to \pi_0(\beta_T)$ which induces in turn a surjection of component orbit sets $f_{**} \from \pi_0(\beta_S) / \Gamma \to \pi_0(\beta_T) / \Gamma$. It follows that $C_1(\beta_S) \ge C_1(\beta_T)$. By Lemma~\ref{LemmaCombingByCollapse} the foldable map $f \from S \mapsto T$ induces a foldable map $f/\beta \from S/\beta_S \to T/\beta_S$ and hence $\Fell S / \beta_S \sqsubset \F(T / \beta_T)$. Applying Lemma~\ref{LemmaFFSNorm}~\pref{ItemFFSxIneq} it follows that $C_2(\beta_S) \ge C_2(\beta_T)$. 

To prove the inequality $C_3(\beta_S) \ge C_3(\beta_T)$, we use the fact that the composition of containment functions $\A \mapsto \Fell S / \beta_S \mapsto \F(T / \beta_T)$ is the containment function $\A \mapsto \F(T / \beta_T)$. It follows that for each component $[A]$ of $\A$, if the component of $\Fell S / \beta_S$ containing $[A]$ is in $\F[\beta_S]$ then the component of $\F(T / \beta_T)$ containing $[A]$ is in $\F[\beta_T]$. The inequality $C_3(\beta_S) \ge C_3(\beta_T)$ follows. Furthermore, for the equation $C_3(\beta_S) = C_3(\beta_T)$ to hold is equivalent to saying that for each $[A] \in \A$, $[A]$ is contained in $\F[S-\beta_S]$ if and only if $[A]$ is contained in $\F[T-\beta_T]$. 

We next prove the inequality $C_4(\beta_S) \ge C_4(\beta_T)$. Consider nested components $[B] \sqsubset [B']$ of $\F[\beta_S] \sqsubset \F[\beta_T]$, respectively. Note that if $B'$ stabilizes a point of $T - \beta_T$ then $B$ stabilizes a point of $S - \beta_S$, and so if $[B'] \in \F[T - \beta_T]$ then $[B] \in \F[S - \beta_S]$. Let $[B_{1}],\ldots,[B_{M}]$ be the components of $\F[S - \beta_S]$ not containing a component of $\A$, and let $[B'_{1}],\ldots,[B'_{M'}]$ be the components of $\F[T-\beta_T]$ not containing a component of $\A$. It follows that we have an extension of free factor systems to which we apply Lemma~\ref{LemmaFFSNorm}~\pref{ItemFFSxIneq}:
\begin{align}
\Fell(S/\beta_S) - \{ [B_{1}],\ldots,[B_{M}]\} &\sqsubset \Fell(T/\beta_T) - \{ [B'_{1}],\ldots,[B'_{M'}] \} \\
\corank\biggl(\Fell(S/\beta_S) - \bigl\{ [B_{1}],\ldots,[B_{M}]\bigr\}\biggr) &\ge \corank\biggl(\Fell(T/\beta_T) - \bigl\{ [B'_{1}],\ldots,[B'_{M'}] \bigr\}\biggr)\\
C_4(\beta_S) &\ge C_4(\beta_T)
\end{align}
with equality holding in (3.2) if and only if it holds in (3.3).

\medskip

Assuming that $C(\beta_S)=C(\beta_T)$, and so $C_i(\beta_S)=C_i(\beta_T)$ for $i=1,2,3,4$, it remains to prove that $f_*$ is a bijection. From surjectivity of $f \from \beta_S \to \beta_T$ it follows that $f_*$ is also surjective, and what is left is to show that $f_*$ is injective. Consider a component $b'$ of~$\beta_T$; we must prove that there is exactly one nondegenerate component of $f^\inv(b')$. Since $C_1(\beta_S)=C_1(\beta_T)$ it follows that $f_{**}$ is a bijection, and so all of the nondegenerate components of $f^\inv(b')$ are in the same $\Gamma$-orbit. If $f^\inv(b')$ has more than one nondegenerate component then any element of $\gamma$ taking one to the other is a nontrivial element of $\Stab(b')$; therefore if $\Stab(b')$ is trivial then $f^\inv(b')$ has only one nondegenerate component and we are done. 

We have reduced to the case that $\Stab(b')$ is nontrivial; by definition we have that $[\Stab(b')] \in \F[\beta_T]$. Since 
$$\DFF(\Fell(S/\beta_S)) + 1 = C_2(\beta_S)=C_2(\beta_T) = \DFF(\Fell(T/\beta_T)) + 1
$$
and since $\Fell(S/\beta_S) \sqsubset \Fell(T/\beta_T)$, by applying Lemma~\ref{LemmaFFSNorm}~\pref{ItemFFSxIneq} we have:
$$ (*) \qquad \Fell(S/\beta_S) = \Fell(T/\beta_T)
$$
Consider the subcase that some nondegenerate component $b$ of $f^\inv(b')$ has nontrivial stabilizer. Since $\Stab(b) \subgroup \Stab(b')$, it follows from $(*)$ that $\Stab(b) = \Stab(b')$. But since all nondegenerate components of $f^\inv(b')$ are in the same orbit, $b$ must be the only such component, because otherwise any $\gamma \in \Gamma$ taking $b$ to a different nondegenerate component is an element of $\Stab(b')$ but not of $\Stab(b)$. The proof is therefore complete in this subcase.

We have further reduced to the subcase that all nondegenerate components of $f^\inv(b')$ have trivial stabilizer, and in this subcase we shall derive a contradiction. It follows from $(*)$ that there exists $x \in S - \beta_S$ such that $\Stab(x) = \Stab(b') \equiv H \subgroup \Gamma$. By definition we have $[H]=[\Stab(x)] \in \F[S-\beta_S]$ whereas $[H]=[\Stab(b')] \in \F[\beta_T]$. We now break into two cases, depending on whether $[H]$ contains some element of $\A$.

Suppose first that $[H]$ contains some $[A] \in \A$, and so up to conjugacy we have $A \subgroup H$. Since $C_3(\beta_S)=C_3(\beta_T)$, it follows that $A$ stabilizes a point of $S-\beta_S$ if and only if $A$ stabilizes a point of $T-\beta_T$ if and only if $A$ does not stabilize any component of $\beta_T$. But $A$ stabilizes the point $x$ of $S-\beta_S$ and the component $b'$ of $\beta_T$, a contradiction. 

Suppose next that $[H]$ contains no element of $\A$. In the notation of (3.1), up to conjugacy we have $H = \Stab(x) = B_{m}$ for some $m=1,\ldots,M$ and so $[H]$ is not an element of the left hand side of (3.1), although $[H] = [\Stab(b')]$ is an element of the right hand side. By Lemma~\ref{LemmaExtension} applied to the extension in (3.1), it follows that the free factor system on the left hand side of (3.1) has a realization with a cofactor $B$ that freely factors into two or more nontrivial  terms, one term up to conjugacy being $H = B_{m}$ (one of the terms denoted $A'_{J+1},\ldots,A'_{K}$ in Lemma~\ref{LemmaExtension}), and another term being a cofactor $B'$ for a realization of the free factor system on the right hand side. From this we obtain $C_4(\beta_S) = \rank(B)  \ge \rank(B') + \rank(B_m) > \rank(B') = C_4(\beta_T)$, contradicting that~$C_4(\beta_S)=C_4(\beta_T)$.
\end{proof}

\subsubsection{Complexity change along a foldable sequence.} 

Given a foldable sequence $T_I \xrightarrow{f_{I+1}} \cdots \xrightarrow{f_J} T_J$, a \emph{pullback sequence} is a sequence of nonempty, nondegenerate subgraphs $\beta_i \subset T_i$ ($I \le i \le J$) such that for each fold $f_i \from T_{i-1} \to T_i$ the subgraph $\beta_{i-1}$ is the pullback of $\beta_i$ via $f_i$; it follows for all $I \le i \le j \le K$ that the subgraph $\beta_i$ is the pullback of $\beta_j$ via $f^i_j$. For example the sequence of nondegenerate subgraphs occurring in a combing rectangle (see Section~\ref{SectionCombing}) is a pullback sequence, assuming those subgraphs are nonempty. Note that the complementary sequence $\rho_i = \closure(T_i - \beta_i) \subset T_i$ is also a pullback sequence. Furthermore, the subgraphs $\beta_i,\rho_i$ decompose the tree $T_i$ in the sense that every edgelet of $T_i$ is in either $\beta_i$ or $\rho_i$; such a sequence of decompositions $T_i = \beta_i \union \rho_i$ is called a \emph{pullback blue--red decomposition}.

The next lemma uses upper bounds on complexity to derive diameter bounds along fold sequences, as expressed in item~\pref{ItemDiameterBound}, the proof of which exploits nonuniqueness of fold paths.


\begin{lemma}[cf.\ \FSOne\ Lemma 5.2] 
\label{LemmaComplexityBounds}
Given a pullback sequence of a foldable sequence of free splittings of $\Gamma$ rel~$\A$, as denoted above, the following holds:
\begin{enumerate}
\item\label{ItemComplexityNonincreasing}
The quantities $C_1(\beta_k)$, $C_2(\beta_k)$, $C_3(\beta_k)$, $C_4(\beta_k)$, and $C(\beta_k)$ are all nonincreasing as functions of~$k$ (for $I \le k \le J$).
\item\label{ItemComplexityEquality}
Equality $C(\beta_{k-1}) = C(\beta_k)$ implies that $f_k \from T_{k-1} \to T_k$ restricts to a bijection from components of $\beta_{k-1}$ to components of $\beta_k$.
\item\label{ItemComplexityConstantImplies}
On any subinterval $a \le k \le b$ along which $C(\beta_k)$ is constant we have:
\begin{enumerate}
\item\label{ItemDiameterBound} The diameter of $\{T_a,\ldots,T_b\}$ in $\FS'(\Gamma;\A)$ is at most~$4$.
\item\label{ItemComponentBound} $C_1(\beta_a) \le C_1(\beta_b) + (3 \corank(\A) + 2 \abs{\A} - 1)$
\end{enumerate}
\end{enumerate}
\end{lemma}

\begin{proof} Items~\pref{ItemComplexityNonincreasing} and~\pref{ItemComplexityEquality} follow from Lemma~\ref{SublemmaNew}. Item \pref{ItemComponentBound} is an immediate consequence of Lemma~\ref{LemmaComplexitySummandBounds} combined with the equation $C(\beta_a) = C(\beta_b)$ and the monotonicity of $C_2(\beta_i)$, $C_3(\beta_i)$ and $C_4(\beta_i)$.

To prove~\pref{ItemDiameterBound}, first 
apply item~\pref{ItemComplexityEquality} to conclude that each $f_k$ induces a bijection from the component set of $\beta_{k-1}$ to the component set of~$\beta_k$. We now repeat the proof of \FSOne\ Lemma~5.2~(3). Given $a \le i \le j \le b$, the required distance bound $d(T_i,T_j) \le 4$ is proved using a certain refactorization of the foldable map $T_i \mapsto T_j$. We shall refactor that map as a fold sequence $T_i = U_0 \mapsto \cdots \mapsto U_P \mapsto\cdots\mapsto U_Q = T_j$ (with $0 \le P \le Q$), and we shall construct free splittings $X,Y$ and a collapse expand sequence $T_i = U_0 \collapses X \expands U_P \collapses Y \expands U_Q = T_j$. To do this, we first construct a partial fold factorization 
$$T_i = U_0 \mapsto \cdots \mapsto U_P \mapsto T_j
$$ 
by prioritizing folds of blue foldable turns, until the induced foldable map $U_P \mapsto T_j$ is injective over $\beta_j$. This is possible because the map $f^i_j \from T_i \to T_j$ induces a bijection from the components of $\beta_i$ to the components of $\beta_j$: if this map is not actually injective over $\beta_j$ then a blue turn exists in $T_i$ that is foldable with respect to $f^i_j$, and we can choose to fold that turn in the first fold factor of~$f^i_j$; proceeding inductively, as long as the map is not yet injective over $\beta_j$, there will exist another blue foldable turn which one can choose to fold in the next fold factor. This partial fold factorization must stop at some $U_P$ so that the foldable map $U_P \mapsto T_j$ is injective over $\beta_j$. The second part of the fold sequence $U_P \mapsto\cdots\mapsto U_Q \approx T_j$ is then chosen arbitrarily: all folded turns along that sequence must be red turns, lying in the inverse image of $\rho_j$. There is a single free splitting $\Gamma \act X$ obtained by collapsing all blue edgelets of $U_0$ or of $U_P$, and a single free splitting $\Gamma \act Y$ obtained by collapsing all red edgelets of $U_P$ or of $U_Q$. Applying Lemma~\ref{LemmaGoodStabilizers} and using that $T_i \in \FS(\Gamma;\A)$, it follows, in order, that $X,U_P,Y \in \FS(\Gamma;\A)$, and hence \hbox{$d(T_i,T_j) \le 4$.} 
\end{proof}

The construction in the following lemma is essentially an argument of Bestvina and Feighn which we have translated into the language of complexity. In the context of $\FS(F_n)$ this construction has the simplifying effect of enfolding two upper bounds on distance from \FSOne---the ``almost invariant edge bound'' and the ``blue--red decomposition bound''---into a single distance bound. In the current context, the upper bound in Lemma~\ref{LemmaPreimageComplexityBound} will be used in the proof of Lemma~\ref{LemmaCoarseRetract} to get certain upper bounds to distance along fold paths. 

\begin{lemma}[\protect{\cite[Lemma~4.1]{BestvinaFeighn:subfactor}}]
\label{LemmaPreimageComplexityBound}
For any foldable map $f \from S \to T$ of free splittings of $\Gamma$ rel~$\A$, and for any point $x \in T$ contained in the interior of some edgelet of ~$T$, there is a nonempty nondegenerate subgraph $\beta_T \subset T$ with pullback subgraph $\beta_S \subset S$ such that 
$$C(\beta_S) \le \abs{f^\inv(x)} + (3 \corank(\A) + 2 \abs{\A} - 1)
$$
\end{lemma}

\begin{proof} Let $e \subset T$ be the edgelet whose interior contains $x$, so $f^\inv(e)$ is a union of $\abs{f^\inv(x)}$ edgelets, and by subdividing further we may assume that this is a disjoint union. Letting $\beta_T = \Gamma \cdot e$, it follows that $C_1(\beta_S) = \abs{f^\inv(x)}$, and the conclusion follows immediately from Lemma~\ref{LemmaComplexitySummandBounds}.
\end{proof}

\subsection{Free splitting units}
\label{SectionFSU}
In \FSOne\ we defined free splitting units along fold paths of $\FS(F_n)$, and we applied them in two ways: in the guts of the ``big diagram'' argument for the proof of hyperbolicity of $\FS(F_n)$; and to give uniformly quasigeodesic parameterizations of fold paths in $\FS(F_n)$ (done here in Theorem~\ref{TheoremRelFSUParams} and Corollary~\ref{CorollaryCompFSU}). We shall do the same here in the relative setting of $\FS(\Gamma;\A)$.

The intuition behind free splitting units arises from close examination of arguments that yield distance bounds along fold paths: buried in such arguments one observes more information in the form of certain combinatorial bounds. One can thus think of free splitting units as the outcome of bringing those combinatorial bounds to the surface where they can be applied --- in this paper these bounds occur as various inequalities involving complexities~$C(\beta)$, as found in Section~\ref{SectionSubgraphComplexity}.

The manner in which free splitting units are defined in this paper is a little different than in \FSOne, with influences from \cite{BestvinaFeighn:subfactor}. Originally, in \FSOne\ Section~5.2, the definition of free splitting units along a fold path in $\FS(F_n)$ involved two distance bounds: a ``blue--red'' bound; and an ``almost invariant edge'' bound. As it turns out, the two bounds can be enfolded into a single, simpler bound that is expressed in Lemma~\ref{LemmaComplexityBounds}, a fact which we overlooked in \FSOne. We found a hint of this fact in \cite[Appendix A]{BestvinaFeighn:subfactor} which introduced simplifications of certain steps of the proof of hyperbolicity of $\FS(F_n)$; see for example Lemma~\ref{LemmaPreimageComplexityBound} and the preceding discussion. Taking this hint, our new version of free splitting units in the general setting of $\FS(\Gamma;\A)$, found in Definition~\ref{DefFSU}, is based on the newer, simpler version of the complexity function $C(\beta)$ that is found in Section~\ref{SectionSubgraphComplexity}. 

See also Definition~\ref{DefCFSU} for the still simpler ``component free splitting units'', based on the still simpler complexity function $C_1(\beta)$. We expect component free splitting units to be more easily applicable and thus to have more powerful applications. For example, in \cite[Section 4.2]{HandelMosher:RelComplexHypIII} we use component free splitting units to prove that amongst all $\phi \in \Out(\Gamma;\A)$ whose action on $\FS(\Gamma;\A)$ has positive translation length, there is a positive lower bound for the translation length depending only on $\corank(\A)$ and $\abs{\A}$. 


\begin{definition}[Free Splitting Units]
\label{DefFSU}
Consider a fold sequence in $\FS(\Gamma;\A)$, denoted
$$S_I \xrightarrow{f_{I+1}} S_{I+1} \xrightarrow{f_{I+2}} \cdots \xrightarrow{f_J} S_J
$$
\begin{enumerate}
\item\label{ItemCollapseExpandDiagram}
For $I \le i \le j \le J$, define a \emph{collapse--expand diagram (rel~$\A$) over $S_i \mapsto\ldots\mapsto S_j$} to be a commutative diagram of the form
$$\xymatrix{
T_i \ar[r] \ar[d]
                 & T_{i+1} \ar[r] \ar[d]
                                & \cdots \ar[r]           & T_{j-1} \ar[r] \ar[d]
                                                                           & T_j  \ar[d]
                                                                           \\
S'_i \ar[r]                         
                 & S'_{i+1} \ar[r]                            
                                & \cdots \ar[r]           & S'_{j-1} \ar[r]   
                                                                            & S'_j & \\
S_i \ar[r]^{f_{i+1}}\ar[u]
                 & S_{i+1} \ar[r]^{f_{i+2}} \ar[u]
                                 & \cdots \ar[r]^{f_{j-1}}  & S_{j-1} \ar[r]^{f_{j}} \ar[u]
                                                                            & S_j \ar[u]
}$$
where the middle and top rows are foldable sequences and each of the two rectangles shown is a combing rectangle rel~$\A$. The diagram is \emph{trivial} if all vertical arrows are simplicial isomorphisms. 
\item\label{ItemLessThanOneFSU}
We say that $S_i,S_j$ \emph{differ by $< 1$ free splitting unit} if there exists a collapse expand diagram over $S_i \mapsto \ldots \mapsto S_j$, denoted as above, such that  on the top row $T_i \to \cdots \to T_j$ there exists a pullback sequence $\beta_k \subset T_k$ of constant complexity~$C(\beta_k)$. 
\item\label{ItemNumberOfFSU}
More generally, the \emph{number of free splitting units between $S_i$ and $S_j$} is the maximum length $\Upsilon=\Upsilon_{ij}$ of a subsequence $i \le i(0) < \cdots < i(\Upsilon) \le j$ that satisfies the following:
\begin{itemize}
\item For each integer $u$ satisfying $1 \le u \le \Upsilon$, the number of free splitting units between $S_{i(u-1)}$ and $S_{i(u)}$ is not~$<1$. 
\end{itemize}
Any such subsequence $i(0)<\cdots<i(\Upsilon)$ of $[i,j]$ is called a \emph{greedy sequence} between $S_i$ and~$S_j$ (with respect to free splitting units). While we do not require $i=i(0)$ and $i(\Upsilon)=j$, Proposition~\ref{PropFSUProps}~\pref{ItemFSUGreedy} below will guarantee that such a greedy sequence exists. 
\item\label{ItemBackAndFrontGreedy}
For $0 \le i \le j \le K$, the \emph{back greedy subsequence} between $i,j$ is the decreasing sequence $j=L_0 > L_1 > \cdots > L_U \ge i$ defined inductively as follows: if $L_u$ is defined, and if there exists $k$ with $i \le k < u$ such that $L_k, L_u$ differ by $\ge 1$ free splitting unit rel~$\A$, then $L_{u+1}$ is the largest such value of~$k$. The \emph{front greedy subsequence} is the increasing sequence in $[0,K]$ defined similarly.
\item We extend free splitting units to a symmetric function by requiring $\Upsilon_{ij}=\Upsilon_{ji}$.
\end{enumerate}
\end{definition}
\noindent
\emph{Remark on terminology.} Note that in the context of item~\pref{ItemNumberOfFSU} of this definition, to say that the number of free splitting units between $S_i$ and $S_j$ is $\Upsilon=0$ is equivalent to saying that $S_i$ and $S_j$ differ by $<1$ free splitting unit: the bulleted statement in item~\pref{ItemNumberOfFSU} is vacuously true when~$\Upsilon=0$. When these equivalent statements hold, we will sometimes favor the terminology ``$<1$ free splitting unit'' as a reminder of its special meaning given in item~\pref{ItemLessThanOneFSU}, expressed in terms of the combinatorics of the fold subsequence $S_i \mapsto\cdots\mapsto S_j$.

\smallskip

The following summarizes basic properties of free splitting units. Of particular importance is item~\pref{ItemFSUBoundsUpsComp}, which derives from Lemma~\ref{LemmaCompArc}, and which plays a central role in the ``big diagram argument'', the proof of Proposition~\ref{PropMMTranslation}.


\begin{proposition}\label{PropFSUProps} 
Consider a fold sequence $S_I \xrightarrow{f_{I+1}} \cdots \xrightarrow{f_J} S_J$ in $\FS(\Gamma;\A)$, and for $I \le i \le j \le J$ let $\Upsilon_{ij}$ be the number of free splitting units rel~$\A$ between $S_i,S_j$. We have:
\begin{enumerate}
\item\label{ItemFSUStable} For any $I \le i \le j \le J$ and $i',j' \in [i,j]$, if $S_i,S_j$ differ by $<1$ free splitting unit then $S_{i'},S_{j'}$ differ by $<1$ free splitting unit. More generally, $\Upsilon_{i'j'} \le \Upsilon_{ij}$.
\item\label{ItemFSUGreedy} (c.f.\ \FSOne\ after Definition 5.10) For any $I \le i \le j \le J$, if $\Upsilon_{ij} \ge 1$ then there exists a greedy sequence between $S_i$ and $S_j$ with first term $i$ and with last term~$j$: $i=i(0) < \cdots < i(\Upsilon_{ij}) = j$. Also, the front and back greedy sequences between $S_i$ and $S_j$ are, indeed, greedy sequences, in particular they have length equal to $\Upsilon_{ij}$. 
\item\label{ItemFSUTriangIneq} The``short triangle inequality'' (c.f.\ \FSOne\ Lemma 5.12):
For any $i,j,k \in \{I,\ldots,J\}$ we have 
$$\Upsilon_{ik} \le \Upsilon_{ij} + \Upsilon_{jk} + 1
$$
\item\label{ItemLongTriangEq} 
For any sequence $k_0,k_1,\ldots,k_L \in \{I,\ldots,J\}$ the ``long triangle inequality'' holds::
$$\Upsilon_{k_0, \, k_L} \le \, \Upsilon_{k_0, \, k_1} + \cdots + \Upsilon_{k_{L-1}, \, k_L}  + L - 1
$$
In addition, if $k_0 < k_1 < \ldots < k_L$ then the ``reverse long triangle inequality'' holds:
$$\Upsilon_{k_0, \, k_1} + \cdots + \Upsilon_{k_{L-1}, \, k_L} \le  \Upsilon_{k_0, \, k_L} 
$$
%
\item\label{ItemFSUUpperBound} 
For any collapse expand diagram as in Definition~\ref{DefFSU}, and for any pullback sequence $\beta_k \subset T_k$ defined for $i \le k \le j$, we have:
\begin{enumerate}
\item\label{ItemFSUBoundsUpsC}
$\Upsilon_{ij} \le C(\beta_i) - C(\beta_j)$
\item\label{ItemFSUBoundsUpsComp}
If $\Upsilon_{ij} \ge 5 \corank(\A) + 4 \abs{\A} - 3$ then some component of $\beta_i$ is arc in the interior of a natural edge of $T_i$.
\end{enumerate}
\item\label{ItemFSUBoundsDiam} (c.f.\ \FSOne\ Lemma 5.11)
The diameter of $\{S_I,\ldots,S_J\}$ is $\le 10 \Upsilon_{IJ} + 8$.
\end{enumerate}
\end{proposition}

\begin{proof} For proving item~\pref{ItemFSUStable}, we transpose the notation as needed to guarantee that $i' \le j'$. The first sentence then follows from Definition~\ref{DefFSU}: any collapse expand diagram over $S_i \mapsto \ldots \mapsto S_j$ which witnesses that there is $<1$ free splitting unit between $S_i$ and $S_j$ restricts to such a collapse expand diagram over $S_{i'} \mapsto\ldots\mapsto S_{j'}$. The second sentence of item~\pref{ItemFSUStable} follows by observing that any maximal length greedy sequence between $S_{i'}$ and $S_{j'}$ is a greedy sequence between $S_i$ and~$S_j$. Items~\pref{ItemFSUGreedy}, \pref{ItemFSUTriangIneq} follow exactly as in the references given above; their proofs are elementary.

The first inequality of \pref{ItemLongTriangEq} --- the long triangle inequality --- follows by induction from \pref{ItemFSUTriangIneq}. To prove the reverse inequality of \pref{ItemLongTriangEq}, for $l=1,\ldots,L$ apply \pref{ItemFSUGreedy} to obtain a subsequence of $[k_{l-1},k_l]$ which is a greedy sequence between $S_{k_{l-1}}$ and $S_{k_l}$, which starts with $k_{l-1}$, which ends with $k_l$, and which has length $\Upsilon_{k(l-1),k(l)}$. The union of these subsequences is a sequence of length equal to the sum $\Upsilon_{k_0, k_1} + \cdots + \Upsilon_{k_{L-1},k_L}$, because the subsequence of $[k_{l-1},k_l]$ ends with $k_l$ and the subsequence of $[k_l,k_{l+1}]$ begins with $k_l$. Between any two terms of this sequence the number of free splitting units is $\ge 1$, so a greedy sequence between $S_{k_0}$ and $S_{k_L}$ has length no less than the sum.


Item~\pref{ItemFSUBoundsUpsComp} follows from~\pref{ItemFSUBoundsUpsC} and Lemma~\ref{LemmaCompArc} together with $C(\beta_j) \ge 1$.  To prove~\pref{ItemFSUBoundsUpsC}, apply~\pref{ItemFSUGreedy} to obtain a greedy sequence 
$$i = k(0) < k(1) < \cdots < k(\Upsilon_{ij}) = j
$$
%
By Lemma~\ref{LemmaComplexityBounds}~\pref{ItemComplexityNonincreasing} we can uniquely decompose the interval $[i,j] = [k(0),k(\Upsilon_{ij})] $ as a concatenation of $M$ maximal subintervals on each of which $C(\beta_k)$ is constant:
$$[\underbrace{l(0)}_{=k(0)},l(1)],\,\, [l(1)\!+\!1,l(2)],\,\,\ldots,\,\,[l(M\!-\!2)\!+\!1,l(M\!-\!1)],\,\,[l(M\!-\!1)\!+\!1,\underbrace{l(M)}_{=k(\Upsilon_{ij})}]
$$
If~$\Upsilon_{ij} \ge C(\beta_i) - C(\beta_j)$ then, since $C(\beta_i) - C(\beta_j) \ge M-1$, it follows $\Upsilon_{ij} + 1 \ge M$, and so there exists $u \in [1,\Upsilon_{ij}]$ and $m \in [1,M]$ such that $k(u-1),k(u) \in [l(m-1)+1,l(m)]$. It follows further that $C(\beta_k)$ is constant for $k \in [k(u-1),k(u)]$, contradicting that there are $\ge 1$ free splitting units between $S_{k(u-1)}$ and $S_{k(u)}$.  

Item~\pref{ItemFSUBoundsDiam} is proven just as in the reference given, except that one applies Lemma~\ref{LemmaFoldLengthTwo} and Lemma~\ref{LemmaComplexityBounds}~\pref{ItemDiameterBound} in place of the analogous results of \FSOne: subdivide the interval $[i,j]$ into a concatenation of maximal subintervals on which $C(\beta_k)$ is constant; apply Definition~\ref{DefFSU}(\ref{ItemCollapseExpandDiagram},\ref{ItemLessThanOneFSU}) and Lemma~\ref{LemmaComplexityBounds}\pref{ItemDiameterBound} to obtain diameter~$\le 8$ over each subinterval; and apply Lemma~\ref{LemmaFoldLengthTwo} to obtain distance~$\le 2$ between incident endpoints of adjacent subintervals.
\end{proof}

\paragraph{Component free splitting units.} The complexity function $C(\beta)$ underlying the definition of free splitting units was specifically designed to have certain easily applicable properties, of which the most characteristic are perhaps Lemma~\ref{LemmaComplexityBounds}~\pref{ItemComplexityEquality} and~\pref{ItemComplexityConstantImplies}. Nonetheless other applications of free splitting units may be impeded by the complicated nature of the complexity function $C(\beta)=C_1(\beta)+C_2(\beta)+C_3(\beta)+C_4(\beta)$: to apply that function directly to a given pullback sequence along a foldable sequence, one must determine whether all four of the terms are constant. But since the $C_2$, $C_3$ and $C_4$ terms are uniformly bounded (see Lemma~\ref{LemmaComplexitySummandBounds}~\pref{ItemCSBIneq}), one can build a simpler version of free splitting unit based solely on the $C_1$ term: we will refer to these as \emph{component free splitting units},  and to $C_1(\beta)$ as the \emph{component complexity} of $\beta$, because $C_1(\beta)$ simply counts the number of components of the orbit space $\beta / \Gamma$. 

\begin{definition}[Component free splitting units]
\label{DefCFSU}
Given a fold path \hbox{$S_I \xrightarrow{f_{I+1}} \cdots \xrightarrow{f_J} S_J$} in $\FS(\Gamma;\A)$ and two indices $I \le i < j \le J$, to say that \emph{$S_i$, $S_j$ differ by $<1$ component free splitting unit} means that there is a collapse expand diagram over $S_i \mapsto\cdots\mapsto S_j$, denoted as in Definition~\ref{DefFSU}~\pref{ItemCollapseExpandDiagram}, such that along the top row $T_i \mapsto\cdots\mapsto T_j$ there is a pullback sequence $\beta_k \subset T_k$ of constant component complexity $C_1(\beta_k) = \abs{\beta_k / \Gamma}$. More generally, the number of \emph{component free splitting units rel~$\A$} between $S_i$ and $S_j$ is the maximum length $\Theta = \Theta_{ij}$ of a subsequence $i \le i(0) < \cdots < i(\Theta) \le j$ satisfying the property that for any integer $\theta$ such that $1 \le \theta \le \Theta$, the number of free splitting units between $S_{i(\theta-1)}$ and $S_{i(\theta)}$ is not~$<1$. As in Definition~\ref{DefFSU}~\pref{ItemNumberOfFSU}, the sequence $i(\theta)$ is called a \emph{greedy sequence} with respect to component free splitting units.
\end{definition}

We note that the analogues of conclusions~\pref{ItemFSUStable}--\pref{ItemLongTriangEq} and \pref{ItemFSUBoundsUpsC} of Proposition~\ref{PropFSUProps} hold as well for component free splitting units as they do for ordinary free splitting units, and with essentially the same proof outlines. Furthermore, as a consequence of Lemma~\ref{LemmaCFSU} to follow, analogues of conclusions~\pref{ItemFSUBoundsUpsComp} and~\pref{ItemFSUBoundsDiam} will also hold except with different constants.

\begin{lemma}
\label{LemmaCFSU}
For any fold path $S_I \mapsto \cdots \mapsto S_J$ in $\FS(\Gamma;\A)$ and any $I \le i < j \le J$, the number $\Upsilon_{ij}$ of free splitting units and the number $\Theta_{ij}$ of component free splitting units between $S_i$ and $S_j$ are related by the following linear inequalities:
$$\Theta_{ij} \,\,\le\,\, \Upsilon_{ij} \,\,\le\,\, (\mathcal C(\Gamma;\A) + 3) \cdot \Theta_{ij} \, + \, \mathcal C(\Gamma;\A)
$$
where $\mathcal C(\Gamma;\A) = 3 \corank(\A) + 2\abs{\A}-1$.
\end{lemma}

\begin{proof} As noted before the statement of the lemma, the analogue of Proposition~\ref{PropFSUProps}~\pref{ItemFSUGreedy} holds for component free splitting units, and so there is a front greedy subsequence $i = k(0) < k(1) < \cdots < k(\Theta_{ij}) \le j$ with respect to component free splitting units, meaning that for each $1 \le \theta \le \Theta_{ij}$ the following hold: $k(\theta)$ is the minimum value of \hbox{$k> k(\theta-1)$} such that between $S_{k(\theta-1)}$ and $S_{k(\theta)}$ there is not $<1$ component free splitting unit, if such a value exists; otherwise $k(\theta)=j$. It follows that along the fold subpath $S_{k(\Theta_{ij})} \mapsto\cdots\mapsto S_j$ there is $<1$ component free splitting unit.

For any \hbox{$1 \le \theta \le \Theta_{ij}$}, for any collapse expand diagram over $S_{k(\theta-1)} \mapsto S_{k(\theta)}$, and for any pullback sequence $\beta_k \subset T_k$ along the top row $T_{k(\theta-1)} \mapsto\cdots\mapsto T_{k(\theta)}$, the sequence $C_1(\beta_k)$ is not constant over the subinterval $k(\theta-1) \le k \le k(\theta)$. Applying Lemma~\ref{LemmaComplexityBounds}, the full complexity $C(\beta_k)$ is also not constant over the subinterval $i(\theta-1) \le k \le i(\theta)$. By maximality of $\Upsilon_{IJ}$ it follows that $\Theta_{IJ} \le \Upsilon_{IJ}$.

To prove the opposite inequality, consider each $1 \le \theta \le \Theta_{IJ}$. It follows by minimality of $i(\theta)$ that along the fold subpath $S_{i(\theta-1)} \mapsto\cdots\mapsto S_{i(\theta)-1}$ there is $<1$ component free splitting unit. We therefore have a decomposition of the fold subpath $S_j \mapsto\cdots\mapsto S_j$ into a concatenation of at most $2\Theta+1$ fold subpaths alternating between two types: at most $\Theta$ single folds, each of the form $S_{i(\theta)-1} \mapsto S_{i(\theta)}$; and at most $\Theta+1$ fold subpaths along each of which there is $<1$ component free splitting unit, namely each of the fold subpaths $S_{i(\theta-1)} \mapsto\cdots\mapsto S_{i(\theta)-1}$ as well as the fold subpath $S_{i(\Theta_{ij})} \mapsto\cdots\mapsto S_j$. Our desired upper bound on $\Upsilon_{ij}$ will come from the long triangle inequality for (ordinary) free splitting units (Proposition~\ref{PropFSUProps}~\pref{ItemLongTriangEq}) as applied to this decomposition, and so we must find upper bounds to the number of free splitting units along each of the terms of the decomposition.

Consider a fold subpath $S_l \mapsto\cdots\mapsto S_m$ along which there is $<1$ component free splitting unit, and let $\Upsilon_{lm}$ be the number of ordinary free splitting units. There is a collapse expand diagram over that subpath and a pullback sequence $\beta_i \subset S_i$ along which $C_1(\beta_i)$ constant, for $l \le i \le m$. Combining the numeric bounds in Lemma~\ref{LemmaComplexitySummandBounds} with the monotonicity property of Lemma~\ref{LemmaComplexityBounds}~\pref{ItemComplexityNonincreasing} it follows that the number of free splitting units from $S_l$ to $S_m$ satisfies the bound
\begin{align*}
\Upsilon_{lm} &\le C(\beta_l) - C(\beta_m) \\
&\le \bigl(C_2(\beta_l) + C_3(\beta_l) + C_4(\beta_l) \bigr) - \bigl(C_2(\beta_m) + C_3(\beta_m) + C_4(\beta_m)  \bigr) \\
  &\le \underbrace{3 \corank(\A) + 2\abs{\A}-1}_{= \, \mathcal C(\Gamma;\A)}
\end{align*}
Also, the number of free splitting units along any single fold map $S_{i-1} \mapsto S_{i}$ clearly satisfies $\Upsilon_{i-1,i} \le 1$.  The long triangle inequality therefore gives us the bound
\begin{align*}
\Upsilon_{ij} &\le (\Theta_{ij}+1) \cdot \mathcal C(\Gamma;\A) + \Theta_{ij} \cdot 1 + \bigl((2\Theta_{ij}+1) - 1\bigr) \\
& = \bigl(\mathcal C(\Gamma;\A) + 3\bigr) \cdot \Theta_{ij} + \mathcal C(\Gamma;\A)
\end{align*}
\end{proof}


\section{Hyperbolicity of relative free splitting complexes}
\label{SectionFFRelAHyp}

In this section we prove hyperbolicity of the relative free splitting complex $\FS(\Gamma;\A)$ for any group $\Gamma$ and any free factor system~$\A$. The proof uses the three Masur--Minsky axioms for hyperbolicity of a connected simplicial complex, which are reviewed in Section~\ref{SectionMMReview}, where one will also find specific details about how those axioms will be verified for $\FS(\Gamma;\A)$. Section~\ref{SubsectionProjection} contains the proof of the first of those axioms, the \emph{Coarse Retract Axiom}; this is where we pay the piper for dropping the ``gate 3 condition'' on fold paths. Section~\ref{SectionAxiomReduction} contains the statement of Proposition~\ref{PropMMTranslation}, which states that certain properties of fold maps and free splitting units together imply the two remaining Masur--Minsky axioms---the \emph{Coarse Lipschitz} and the \emph{Strong Contraction Axioms}. The proof of Theorem~\ref{TheoremRelFSGammaHyp} is thereby reduced to Proposition~\ref{PropMMTranslation}. Section~\ref{SectionFSUParameterization} also applies Proposition~\ref{PropMMTranslation} to the proof of Theorem~\ref{SectionFSUParameterization} which says that free splitting units give a quasigeodesic parameterization along a fold path. Section~\ref{SectionBigDiagrams} contains the proof of Proposition~\ref{PropMMTranslation}, what we call the ``Big Diagram'' argument, an argument concerning the large scale behavior of certain diagrams of combing rectangles in $\FS(F_n;\A)$.

The structure of this proof of Theorem~\ref{TheoremRelFSGammaHyp}, in particular the Big Diagram argument, follows very closely the structure of the proof of hyperbolicity of $\FS(F_n)$ given in \FSOne. But changing the definition of foldable maps by dropping the gate~3 condition has some major effects on this structure: the proof of the \emph{Coarse Retract Axiom} is quite a bit more complex and so has needed to be rewritten from the beginning; and subtle changes in the Big Diagram Argument make it necessary to re-present it from the beginning. In both cases, we take up these changes from the version of the proof given by Bestvina and Feighn in \cite{BestvinaFeighn:subfactor}.

\subsection{The Masur--Minsky axioms.} 
\label{SectionMMReview}

As noted in \cite[Section 3]{\FSOneTag}, the original Masur-Minsky axioms, when applied to a connected $1$-dimensional simplicial complex $X$ equipped with its simplicial metric, are easily shown to be equivalent to the following discretized version of those axioms. For integers $I \le K$ we denote the integer interval $[I,\ldots,K] = \{j \in \Z \suchthat I \le j \le K\}$, and we symmetrize this notation by letting $[K,\ldots,I]=[I,\ldots,K]$. The axioms require two structures to be given: a family of \emph{paths} $P$; and a family of projection maps $\pi_p$ one for each $p \in P$. Each  $p \in P$ is a function $p \from [I_p,\ldots,K_p] \to X^{(0)}$ with $I_p < K_p$, and this collection of ``paths'' satisfies the following \emph{almost transitivity} condition for some uniform constant~$A \ge 0$ independent of~$p$: if $I_p < \ell \le K_p$ then $d(p(\ell-1),p(\ell)) \le A$; and for each $x,y \in X^{(0)}$ there exists $p \in P$ such that $d(x,p(0)) \le A$ and $d(p(L_p),y) \le A$. 
Each projection map is a function $\pi_p \from X^{(0)} \to [I_p,\ldots,K_p]$.

The following three axioms must hold for each $p \in P$, with uniform constants $a,b,c>0$: 
\begin{description}
\item[Coarse Retract Axiom:] For each $\ell \in [I_p,\ldots,K_p]$ the diameter of the subpath $p[\ell,\ldots,\pi_p(p(\ell))]$ is at most~$c$.
\item[Coarse Lipschitz Axiom:] For all $x,y \in X^{(0)}$ such that $d(x,y) \le 1$, the diameter of the subpath $p[\pi_p(x),\ldots,\pi_p(y)]$ is at most $c$.
\item[Strong Contraction Axiom:] \qquad For all $x,y \in X^{(0)}$, if $d(x,p(\pi_p(x))) \ge a$ and if \break $d(x,y) \le b \, d(x,p(\pi_p(x)))$, then the diameter of the subpath $p[\pi_p(x),\ldots,\pi_p(y)]$ is at~most~$c$.
\end{description}
The theorem proved by Masur and Minsky \cite{MasurMinsky:complex1} says that if these axioms hold then $X$ is hyperbolic. Furthermore, the proof of that theorem given in \cite[Section 7]{MasurMinsky:complex1} gives some additional quantitative consequences, with constants $\delta > 0$, $\kappa \ge 1$, $\epsilon \ge 0$ depending only on $A$, $a$, $b$, $c$.
\begin{description}
\item[Uniform Hyperbolicity:] $X$ is $\delta$-hyperbolic.
\item[Uniform Reparameterized Quasigeodesics:] Each path $p \from [I_p,\ldots,K_p] \to X$ in $P$ is a \emph{$\kappa,\epsilon$ reparameterized quasigeodesic} meaning that there exists an interval $[x,y] \subset \mathbb R$, a decomposition $x=x_0 < \cdots < x_N=y$ with $N=K_p-I_p$, and a $\kappa,\epsilon$ continuous quasigeodesic path $\gamma \from [x,y] \to X$, such that $\gamma(x_n)=p(I_p+n)$ (for $0 \le n \le N$) and $\gamma \restrict [x_{n-1},x_n]$ is a path of length~$\le A$ in $X$ from $p(I_p+n-1)$ to $p(I_p+n)$.
\end{description}

Looking forward to how we verify the Masur--Minsky axioms and hence prove hyperbolicity of $\FS(\Gamma;\A)$, we shall take $P$ be the collection of all fold sequences of $\Gamma$ rel~$\A$, for which we have already established almost transitivity in Corollary~\ref{CorollaryConnected} with $A=2$. In Section~\ref{SubsectionProjection} we define the system of projection maps $\pi_p$ ($p \in P$), and we prove the Coarse Retract Axiom (see Lemma~\ref{LemmaCoarseRetract}). In Section~\ref{SectionAxiomReduction} we shall reduce the Coarse Lipschitz and Strong Contraction Axioms to a single statement, namely Proposition~\ref{PropMMTranslation}, regarding fold sequences and free splitting units rel~$\A$. In Section~\ref{SectionFSUParameterization} we apply Proposition~\ref{PropMMTranslation} to prove a combinatorial reformulation of the \emph{Uniform Quasigeodesics} statement above, saying that fold paths are uniformly quasigeodesic when reparameterized by free splitting units; see Theorem~\ref{TheoremRelFSUParams}. Finally in Section~\ref{SectionBigDiagrams} we prove Proposition~\ref{PropMMTranslation} using the ``big diagram argument''. In these proofs one will observe that the Masur--Minsky constants $A$, $a$, $b$, $c$ that are produced will depend only on $\corank(\Gamma;\A)$ and $\abs{\A}$, and hence $\delta$, $\kappa$, $\epsilon$ and $s$ will have the same dependence.

\subsection{Projection maps and the proof of the Coarse Retract Axiom.}
\label{SubsectionProjection}

Given a fold path in $\FS(\Gamma;\A)$ represented by a particular fold sequence, we now define the projection map to that fold path, as required in the formulation of the Masur--Minsky axioms. We then immediately turn to verification of the Coarse Retract Axiom, which takes up the bulk of this section.


\begin{definition}
\label{DefProjDiagram}
Given a fold sequence $S_I \mapsto\cdots\mapsto S_K$ and a free splitting $T$ each in $\FS(\Gamma;\A)$, a \emph{projection diagram rel~$\A$ from $T$ to $S_I \mapsto\cdots\mapsto S_K$} is defined to be a projection diagram as in \FSOne\ Section~4.1 in which all free splittings that occur are restricted to lie in $\FS(\Gamma;\A)$. This means a commutative diagram of maps amongst free splittings of $\Gamma$ rel~$\A$, having the form
$$\xymatrix{
T_I \ar[r] \ar[d] & \cdots \ar[r] & T_{J} \ar[r] \ar[d] &  T \\
S'_I \ar[r]          & \cdots \ar[r] & S'_{J} \\
S_I \ar[r] \ar[u] & \cdots \ar[r] & S_{J} \ar[r] \ar[u]  & \cdots \ar[r] & S_K \\
}$$
such that each row is a foldable sequence, and each of the two rectangles shown is a combing rectangle. The integer $J \in \{I,\ldots,K\}$ is called the \emph{depth} of the projection diagram. The \emph{projection of $T$ to $S_I \mapsto \cdots S_K$} is an integer $\pi(T) \in \{I,\ldots,K\}$ defined as follows: if there exists a projection diagram rel~$\A$ from $T$ to $S_I \mapsto \cdots S_K$ then $\pi(T)$ is the maximal depth of such diagrams; otherwise $\pi(T)=I$. Sometimes we abuse notation and refer to $S_{\pi(T)}$ itself as the projection.
\end{definition}

\begin{lemma}[The Coarse Retract Axiom]
\label{LemmaCoarseRetract}
For any fold sequence $S_I \mapsto\cdots\mapsto S_K$ in $\FS(\Gamma;\A)$ and any $I \le M \le K$, the number of free splitting units between $S_M$ and $S_{\pi(S_M)}$, and the diameter of the fold sequence between $S_M$ and $S_{\pi(S_M)}$, are both bounded above by constants depending only on $\corank(\A)$ and $\abs{\A}$.
\end{lemma}

By assuming the gate~3 condition, the proof of the Coarse Retract Axiom in \FSOne\ was significantly simpler than the argument to be presented here. In lieu of that assumption, we instead adapt some concepts and arguments of Bestvina and Feighn from \cite{BestvinaFeighn:subfactor}, namely the ``hanging trees'' of \cite[Proposition~A.9]{BestvinaFeighn:subfactor}; see ``Claim (\#)'' below.

\begin{proof} The proof starts as in \FSOne. Note that $\pi(S_M) \ge M$, because there exists a projection diagram rel~$\A$ from $S_M$ to $S_I \mapsto\cdots\mapsto S_K$ of depth $M$, namely the trivial diagram defined by taking $T_i=S'_i=S_i$ for $i=I,\ldots,M$. 

Choose a projection diagram rel~$\A$ of maximal depth $J=\pi(S_M) \ge M$ from $S_M$ to $S_I\mapsto\cdots\mapsto S_K$, as follows:

\centerline{\xymatrix{
T_I \ar[r] \ar[d] & \cdots \ar[r] & T_M \ar[r] \ar[d] & \cdots \ar[r] & T_{J} \ar[r] \ar[d] & S_M \\
S'_I \ar[r]          & \cdots \ar[r] & S'_M \ar[r]        & \cdots \ar[r] & S'_{J} \\
S_I \ar[r] \ar[u] & \cdots \ar[r] & S_M \ar[r] \ar[u] & \cdots \ar[r] & S_{J} \ar[r] \ar[u]  & \cdots \ar[r] & S_K \\
}}

\medskip
\noindent
Once we have bounded the number of free splitting units along the fold path on the bottom row between $S_M$ and $S_J=S_{\pi(S_M)}$, the diameter bound on the set $\{S_M,\ldots,S_{\pi(S_M)}\}$ follows from Proposition~\ref{PropFSUProps}~\pref{ItemFSUBoundsDiam}. 

The key observation is that in the foldable sequence $T_M \mapsto \cdots \mapsto T_J \mapsto S_M$, its first and last terms $T_M,S_M$ each collapse to the same free splitting, namely $S'_M$. This observation will be combined with the following:

\begin{description}
\item[Claim $(\#)$:] Consider a fold sequence $\displaystyle U_0 \xrightarrow{f_1} \cdots \xrightarrow{f_L} U_L$ in $\FS(\Gamma;\A)$. If there exists a free splitting $R \in \FS(\Gamma;\A)$ and collapse maps $U_0 \mapsto R$, $U_L \mapsto R$, then there exist integers $0=\ell_0 \le \ell_1 \le \ell_2 = L$, and for $i=1,2$ there exist $x_i \in U_{\ell_i}$, such that the inverse image $(f^{\ell_{i-1}}_{\ell_i})^\inv(x_i) \subset U_{\ell_{i-1}}$ has cardinality bounded by a constant $b_\# = b_\#(\Gamma;\A)$.
\end{description}

Before proving Claim~$(\#)$, we apply it to finish the proof of Lemma~\ref{LemmaCoarseRetract}, as follows. By replacing each individual arrow in the foldable sequence $T_M \mapsto\cdots\mapsto T_J \mapsto S_M$ by a fold sequence that factors it, we obtain a fold sequence which contains $T_M,\ldots,T_J,S_M$ as a subsequence. To that fold sequence we may then apply Claim~$(\#)$, combined with Lemma~\ref{LemmaPreimageComplexityBound} followed by Lemma~\ref{LemmaComplexityBounds}~\pref{ItemComplexityNonincreasing}, with the effect of subdividing the fold sequence between $T_M$ and $T_J$ into at most two subintervals along each of which there is pullback sequence of bounded complexity difference. Then applying Proposition~\ref{PropFSUProps}~\pref{ItemFSUUpperBound} we obtain a subdivision of the fold sequence between $S_M$ and $S_J$ into at most two subintervals along each of which the number of free splitting units rel~$\A$ is bounded. Applying Proposition~\ref{PropFSUProps}~\pref{ItemFSUTriangIneq} we obtain an upper bound to the number of free splitting units rel~$\A$ between $S_M$ and~$S_J$.

We turn to the proof of Claim~(\#). Denote $ V = U_0 \xrightarrow{f = f^0_L} U_L = W$. Choose oriented natural edges $e_V=[v_-,v_+] \subset V$,\, $e_W=[w_-,w_+] \subset W$ which map onto the same oriented natural edge $e_R = [r_-,r_+] \subset R$ under collapse maps $V,W \mapsto R$. Decompose $V \setminus e_V = V_- \union V_+$ and $W \setminus e_W = W_- \union W_+$ so that $v_\pm$ is the frontier of $V_\pm$, and~$w_\pm$ is the frontier of $W_\pm$, respectively. We have equations 
$$(*) \qquad f(V_+) = W_+ \union [w_+,f(v_+)], \qquad f(V_-) = W_- \union [w_-,f(v_-)]
$$
which are obtained by referring to \FSOne\ Lemma~5.5 and following the proof of the implication \hbox{(4)$\implies$(1)}, except that one may ignore the very last sentence which is the only place in that proof where the gate~3 condition was used. We briefly outline the proof of $(*)$ for $f(V_+)$. Decompose $R \setminus e_R = R_- \union R_+$ so that $r_\pm$ is in the frontier of~$R_\pm$. Let~$\Gamma_+$ be the set of elements of $\Gamma$ acting loxodromically on $R$ with axis contained in~$R_+$. First one shows the inclusion $W_+ \subset f(V_+)$ by proving for each $\gamma \in \Gamma_+$ that $\gamma$ acts loxodromically on $V$ and $W$ with axes contained in $V_+$ and $W_+$ respectively, and that the union of such axes over $\gamma \in \Gamma_+$ equals $V_+$, $W_+$ respectively, and finally using that for each $\gamma \in \Gamma_+$ the $f$ image of the axis of $\gamma$ in $V_+$ contains the axis of $\gamma$ in $W_+$. Next one shows, by bounded cancellation, that $f(V_+)$ is contained in a finite radius neighborhood $N_r(W_+)$ of $W_+$. Finally, using that neighborhood, one shows that if $(*)$ fails then $f(V_+)$ contains a valence~$1$ point distinct from $f(v_+)$, that point has the form $f(x)$ for some $x \in V_+ - v_+$, and $f$ has only one gate at $x$, contradicting foldability of~$f$.

Consider the oriented segment $f(e_V)=[f(v_-),f(v_+)] \subset W$. If $f(e_V)$ intersects $\interior(e_W)$ and preserves orientation, then $f$ is one-to-one over some point $x \in\interior(e_W)$, so Claim~$(\#)$ is proved with $\ell_1=L$, $x_1=x_2=x$, and $b_\#=1$. If $f(e_V)$ intersects $\interior(e_W)$ and reverses orientation, then $f$ has only one gate at $v_-$ and at $v_+$, a contradiction. 

\smallskip
\textbf{Remark.} Under the gate~3 hypothesis on the given fold sequence the proof of Claim~$(\#)$ ends here, because the segment $f(e_V)$ must intersect $\interior(e_W)$; see the last lines of the proof of \FSOne\ Lemma~5.5. Without the gate~3 hypothesis our work continues for rather a long while.

\smallskip

We may assume that $f(e_V)$ is a subset of $W_-$ or of $W_+$; up to reverse of orientation we have 
$$f(e_V) \subset W_+, \qquad f(V_+)=W_+, \qquad\text{and}\qquad e_W = [w_-,w_+] \subset \underbrace{[w_-,f(v_-)]}_{= \alpha}
$$
We orient $\alpha$ with initial endpoint $w_-$ and terminal endpoint $f(v_-)$, and we parameterize $\alpha$ by simplicial distance from $w_-$, inducing a linear order on $\alpha$ which lets us speak of maxima and minima in $\alpha$. Let $\Sigma = V_- \intersect f^\inv \alpha$ and $\xi = \Sigma \intersect f^\inv(w_-)$. Note that $\Sigma-\xi$ is connected: otherwise the closure of some component of $\Sigma-\xi$ would not contain $v_-$, its image in $\alpha$ would have a maximum value achieved at some $x \in \interior(\Sigma)$, and $x$ would have one gate, a contradiction. It follows that $\Sigma$ is connected, that each point of $\xi$ has valence~$1$ in $\Sigma$, and that $\xi \in \frontier(\Sigma)$. Furthermore $\frontier(\Sigma) = \xi \union \{v_-\}$, and the map $f \from \Sigma \to \alpha$ takes $\frontier(\Sigma)$ to $\bdy\alpha$, mapping $\xi$ to $w_-$ and $v_-$ to $f(v_-)$. 

Consider the \emph{initial edgelets} of $\Sigma$ meaning the edgelets incident to points of $\xi$, each of which maps to the initial edgelet of $e_W$. It follows that the initial edgelets of $\Sigma$ are all in different $\Gamma$-orbits, and so there are only finitely many of them, implying that $\xi$ is finite and so $\Sigma$ is a finite tree. Since each vertex of $\Sigma-\Fr(\Sigma)$ has at least two gates with respect to the map $f \restrict \Sigma$, and since $f(\Sigma - \Fr(\Sigma)) \subset \alpha$ it follows that each vertex of $\Sigma-\Fr(\Sigma)$ has exactly two gates and that $f(\Sigma-\Fr(\Sigma)) \subset \interior(\alpha)=(w_-,f(v_-))$. Assign an orientation to each edgelet of $\Sigma$ so as to point towards $v_-$. By induction on distance to $\xi$ it follows that $f$ maps each edgelet of $\Sigma$ to an edgelet of $\alpha$ in an orientation preserving manner. It follows that for each $x \in \xi$ the map $f$ takes $[x,v_-]$ one-to-one onto $\alpha$. Furthermore at each $y \in \interior(\Sigma)$ there is therefore a unique \emph{positive direction} with respect to~$f$, namely the direction pointing towards $v_-$, which is the unique direction at $y$ whose image under $f$ is the direction at $f(y) \in \alpha$ pointing towards $f(v_-)$. All other directions at $y$ form the \emph{negative gate}, each mapping to the direction at $f(y) \in \alpha$ pointing back towards $w_-$.  This gives $\Sigma$ the structure of a ``hanging tree'' in the terminology of \cite{BestvinaFeighn:subfactor}. It follows that every edgelet of $\Sigma$ that maps to the initial edgelet of $e_W$ is an initial edgelet of $\Sigma$.

We break into two cases depending on the behavior of the following subset of $\Gamma$: 
\begin{align*}
\wh Z &= \{\gamma \in \Gamma \suchthat \interior(\Sigma) \intersect \interior(\gamma \cdot \Sigma) \ne \emptyset\} \\ &= \{\gamma \in \Gamma \suchthat \Sigma \intersect \gamma \cdot \Sigma\,\,\text{contains at least one edgelet}\}
\end{align*}

\smallskip

\textbf{Case 1:} $\wh Z = \{\text{Id}\}$. Consider the graph of groups $V/\Gamma$. It follows in Case~1 that the orbit map $V \to V / \Gamma$ restricts to an injection on $\interior(\Sigma)$. The images of the initial edgelets of $\Sigma$ are therefore all contained in distinct oriented natural edges of $V/\Gamma$. Applying Proposition~\ref{PropGenericFS}~\pref{ItemNatEBound} it follows that the number of initial edgelets is bounded by the number $2 \DFS(\A) + 2 =  6 \corank(\A) + 4 \abs{\A} - 6$. Taking this number to be $b_\#$, Claim~$(\#)$ is proved with $\ell_1=L$ and with $x_1=x_2=$ an interior point of the initial edgelet of~$e_W$.

\smallskip

\textbf{Case 2:} $\wh Z \ne \{\text{Id}\}$.  The action of each $\gamma \in \wh Z$ on the tree $W$ restricts as 
$$\tau_\gamma \from (\gamma^\inv \cdot \alpha) \intersect \alpha \to \alpha \intersect (\gamma \cdot \alpha)
$$
Furthermore, this map $\tau_\gamma$ is an isometry with respect to the parameterization of $\alpha$ described earlier, so we may speak about whether $\gamma$ preserves or reverses orientation, and if $\gamma$ preserves orientation we may also speak about the translation length of $\gamma$, all by reference to what~$\tau_\gamma$ does to the parameterization of~$\alpha$.

We next show:
\begin{enumerate}
\item\label{ItemFrontierToEndpoints}
For each $\gamma \in \wh Z$ the arc $\alpha \intersect \gamma \cdot \alpha$ has endpoints in the set $\bdy \alpha \union \gamma \cdot \bdy \alpha$.
\end{enumerate}
This follows from the earlier description of how $f$ maps $\Sigma$ to $\alpha$ and $\frontier(\Sigma)$ to $\bdy\alpha$, together with the fact that $\frontier(\Sigma \intersect (\gamma\cdot \Sigma)) \subset \frontier (\Sigma) \union (\gamma\cdot \frontier(\Sigma))$. 

If $\gamma$ reverses orientation it follows from \pref{ItemFrontierToEndpoints} that $\gamma^2$ fixes the arc $\alpha \intersect \gamma \cdot \alpha$ and so, since $W$ is a free splitting, $\gamma^2$ is trivial, but that is a contradiction. Every element of $\wh Z$ therefore preserves orientation. 

We define $\gamma \in \wh Z$ to be \emph{positive} if $\gamma$ has positive translation length with respect to the parameterization of $\alpha$; for $\gamma$ to be \emph{negative} is similarly defined by requiring negative translation length. Thinking of the map $\Sigma \xrightarrow{f} \alpha$ as a ``height function'', an element of $\wh Z$ is positive if and only if it increases height in $\Sigma$, and negative if and only if it decreases height.

Letting $Z \subgroup \Gamma$ be the group generated by $\wh Z$, we shall show:
\begin{enumeratecontinue}
\item\label{ItemOverlapSign}
There exists a positive $\gamma \in \wh Z$ such that $Z = \<\gamma\>$ is infinite cyclic. 
\end{enumeratecontinue}
For any free splitting $\Gamma \act X$ in which the action of the cyclic group $Z$ is not elliptic, let $\Ax(X)$ denote the axis of that cyclic group. We show furthermore that:
\begin{enumeratecontinue}
\item\label{ItemOverlapTree}
The set $H = Z \cdot \Sigma \subset V$ is a two-ended tree on which $Z$ acts cocompactly, there is a $Z$-equivariant deformation retraction $H \mapsto \Ax(V)$, and 
$$f(H) = f(Z \cdot \Sigma) = Z \cdot f(\Sigma) = Z \cdot \alpha = \Ax(W)
$$
\item\label{ItemInfiniteHangingTree}
The map $f \from H \to \Ax(W)$ gives $H$ the structure of a ``bi-infinite hanging tree'' as follows: at each $x \in H$ the map $f \restrict H$ has a \emph{positive gate} consisting of the unique direction at $x$ whose $f$-image points towards the positive end of $\Ax(W)$, and if $x$ is not of valence~$1$ then all other directions at $x$ are in a single \emph{negative gate} whose $f$-image points towards the negative end of $\Ax(W)$.
\item\label{ItemOverlapTreeTranslates}
All translates of the tree $H$ by elements of $\Gamma - Z$ have disjoint interiors.
\end{enumeratecontinue}
For the proofs of \pref{ItemOverlapSign}--\pref{ItemOverlapTreeTranslates}, pick a positive $\gamma \in \wh Z$ whose translation distance on $\alpha$ is a minimum. Given a positive $\delta \in \wh Z$, note that $\gamma^\inv \delta$ is in $\wh Z$ and is non-negative. By induction there exists $i \ge 0$ such that $\gamma^{-i} \delta \in \wh Z$ and has translation number zero, implying that it fixes an arc of $\alpha$, and so $\delta = \gamma^i$. This proves~\pref{ItemOverlapSign}, and \pref{ItemOverlapTree} follow easily. Item~\pref{ItemInfiniteHangingTree} follows from the analogous properties of the map $f \from \Sigma \to [w_-,f(v_-)]$. For~\pref{ItemOverlapTreeTranslates}, suppose $\delta \in \Gamma$ has the property that the interiors of $H$ and $\delta \cdot H$ are not disjoint. Choose integers $i,j$ such that the interiors of $\gamma^i \cdot \Sigma$ and $\delta\gamma^j \cdot \Sigma$ are not disjoint, so the interiors of $\Sigma$ and $\gamma^{-i} \delta \gamma^j \Sigma$ are not disjoint. By \pref{ItemOverlapSign} we have $\gamma^{-i} \delta \gamma^j \in \wh Z$ and so $\delta\in Z$, proving~\pref{ItemOverlapTreeTranslates}.

From properties~\pref{ItemOverlapSign}--\pref{ItemOverlapTreeTranslates} it follows that the 1-complex $H / Z$ deformation retracts to the circle $\Ax(V) / Z$. Furthermore, the induced map $H / Z \mapsto V / \Gamma$ is an embedding on the complement of the valence~1 vertices. Define an \emph{initial edgelet} of $H$ to be an oriented edgelet whose initial vertex has valence~$1$ in $H$, and so we have a bijection between initial edgelets and valence~$1$ vertices. Define an initial edgelet  of $H/Z$ in a similar fashion. We have a bijection between $Z$-orbits of initial edgelets of $H$ and initial edgelets of $H/Z$. Under the map $H/Z \mapsto V/\Gamma$, the initial edgelets of $H/Z$ all map into distinct oriented natural edges of $V / \Gamma$. The number of $Z$-orbits of initial edgelets of $H$ is therefore bounded above by $2 \DFS(\A) + 2 =  6  \corank(\A) + 4 \abs{\A} - 6$ (see Proposition~\ref{PropGenericFS}~\pref{ItemNatEBound}), and so the number of $Z$-orbits of valence~$1$ vertices of $H$ has the same bound. A \emph{branch} of $H$ is an oriented arc with initial endpoint at a valence~1 vertex, terminal endpoint on $\Ax(V)$, and interior disjoint from $\Ax(V)$. Letting $\beta_v \subset H$ denote the branch with initial vertex~$v$, we have a $Z$-equivariant bijection $v \leftrightarrow \beta_v$ between valence~1 vertices and branches, and so: 
\begin{enumeratecontinue}
\item\label{ItemBranchOrbitBound}
The number of $Z$-orbits of branches of $H$ is bounded by 
$$2 \DFS(\A) + 2 = 6  \corank(\A) + 4 \abs{\A} - 6
$$
\end{enumeratecontinue}
We also have, as a consequence of property~\pref{ItemInfiniteHangingTree}, the following: 
\begin{enumeratecontinue}
\item\label{ItemBranchInject}
The map $f \from V \to W$ is injective on each branch $\beta_v$, mapping it homeomorphically to an arc of $\Ax(W)$. In particular, $\beta_v$ is legal with respect to $f$. 
\end{enumeratecontinue}

Consider now the whole fold sequence $V=U_0 \mapsto\cdots\mapsto U_L = W$. Choose $x_2 \in e_W$ to be in the interior of some edgelet. If $f^0_L$ is one-to-one over $x_2$ then we are done with $x_1=x_2$, $\ell_1 = L$, and $b_\#=1$. Otherwise, let $\ell_1 \in \{0,\ldots,L\}$ be the largest integer such that the map $f^{\ell_1}_L \from U_{\ell_1} \to U_L$ is not 1-to-1 over~$x_2$. Since $f^{\ell_1+1}_L$ is 1-to-1 over~$x_2$, and since the fold map $f_{\ell_1+1}$ is at worst 2-to-1 over the interior of each edgelet, it follows that $f^{\ell_1}_L$ is exactly 2-to-1 over $x_2$. Let $y \in U_{\ell_1+1}$ be the unique point of $(f^{\ell_1+1}_L)^\inv(x_2)$. Note that $y \in \Ax(U_{\ell_1+1})$, because 
\begin{align*}
x_2 \in \alpha & \subset \Ax(W) \qquad\qquad \qquad \text{(see item \pref{ItemOverlapTree})}\\
        &= \Ax(U_L) \subset f^{\ell_1+1}_L(\Ax(U_{\ell_1+1}))
\end{align*}
Under the fold map $f_{\ell_1+1} \from U_{\ell_1} \to U_{\ell_1+1}$ the point $y$ has exactly 2 pre-images, exactly one of which denoted $x_1$ is \emph{disjoint from} $\Ax(U_{\ell_1})$. Let $P = (f^0_{\ell_1})^\inv(x_1) \subset U_0$, which is \emph{disjoint from} $\Ax(U_0)$. It remains to show
\begin{description}
\item[Claim $(*)$] The cardinality of $P$ is $\le$ the number of $Z$-orbits of branches of~$H$. 
\end{description}
Applying this claim, it follows by \pref{ItemBranchOrbitBound} that the cardinality of $P$ is bounded above by $2 \DFS(\A) + 2 = 6  \corank(\A) + 4 \abs{\A} - 4$. Claim~$(\#)$ is then proved by taking $b_\# = \max\{2, 2 \DFS(\A) + 2\} = 2\DFS(\A)+2$. 

For proving Claim~$(*)$, by applying item~\pref{ItemOverlapTreeTranslates} we conclude that $P \subset \interior(H)$, so $P$ is the inverse image of $x_1$ under the restriction of $f^0_{\ell_1}$ to $H$. Consider $H_{\ell_1} = f^0_{\ell_1}(H) \subset U_{\ell_1}$. From items~\pref{ItemOverlapTree}, \pref{ItemInfiniteHangingTree}, which describe the infinite hanging tree structure on $H$ with respect to the map $f=f^0_L \from H \to \Ax(W)$, it follows that $H_{\ell_1}$ also has an infinite hanging tree structure with respect to the map $f^{\ell_1}_L \from H_{\ell_1} \to \Ax(W)$. For each branch $\beta_v \subset H$, by~\pref{ItemBranchInject} the map $f$ takes $\beta_v$ homeomorphically onto a subsegment of $\Ax(W)$, from which it follows that $f^0_{\ell_1}$ maps $\beta_v$ homeomorphically onto its image $f^0_{\ell_1}(\beta_v)$, a path which therefore takes no illegal turns in $H_{\ell_1}$ with respect to the map $f^{\ell_1}_L \from H_{\ell_1} \to \Ax(W)$. Combining this with the infinite hanging tree structure on $H_{\ell_1}$, it follows that if $\mu \subset \beta_v$ is a subpath, with homeomorphic image subpath  $f^0_{\ell_1}(\mu) \subset f^0_{\ell_1}(\beta_v)$, and if the endpoints of $f^0_{\ell_1}(\mu)$ are disjoint from $\Ax(U_{\ell_1})$, then all of $f^0_{\ell_1}(\mu)$ is disjoint from $\Ax(U_{\ell_1})$, because any path between points in distinct components of an infinite hanging tree minus its axis must contain an illegal turn.

If Claim~$(*)$ fails then there exist $b \ne b' \in P$, a branch $\beta_v \subset H$, and~$\gamma^i \in Z$, such that $b \in \beta_v$ and $b' \in \gamma^i \cdot \beta_v$. The path $\mu = [\gamma^{-i}(b'),b]$ is contained in $\beta_v$ and so it is mapped homeomorphically to the path $f^0_{\ell_1}(\mu) = f^0_{\ell_1}[\gamma^{-i}(b'),b]$. Also, the endpoints $\gamma^{-i}(b')$, $b$ of $\mu$ are mapped by $f^0_{\ell_1}$ to the endpoints $\gamma^{-i}(x_1)$, $x_1$ of $f^0_{\ell_1}(\mu)$, neither of which are in $\Ax(U_{\ell_1})$. It follows that $f^0_{\ell_1}(\mu)$ is disjoint from $\Ax(U_{\ell_1})$. In the tree $U_0$ consider the bi-infinite, $\gamma^i$-invariant sequence of paths
$$\cdots \underbrace{[\gamma^{-2i}(b'),\gamma^{-i}(b)]}_{\gamma^{-i}(\mu)}, \, \underbrace{[\gamma^{-i}(b'),\gamma^0(b)]}_{\mu}, \, \underbrace{[\gamma^0(b'),\gamma^i(b)]}_{\gamma^i(\mu)}, \, \underbrace{[\gamma^i(b'),\gamma^{2i}(b)]}_{\gamma^{2i}(\mu)}, \cdots
$$
Since $f^0_{\ell_1}(\gamma^{-mi}(b)) = \gamma^{-mi}(x_1) = f^0_{\ell_1}(\gamma^{-mi}(b'))$ for all~$m$, it follows that the image of the above sequence of paths under the map $f^0_{\ell_1}$ concatenates together to form a bi-infinite $\gamma^i$-invariant path in $U_{\ell_1}$ which is disjoint from $\Ax(U_{\ell_1})$, a contradiction.
\end{proof}

\subsection{Proof of Theorem~\ref{TheoremRelFSGammaHyp}: Reducing the Coarse Lipschitz and Strong Contraction Axioms to Proposition \ref{PropMMTranslation}.} 
\label{SectionAxiomReduction}
In \FSOne\ we proved hyperbolicity of $\FS(F_n)$ by using fold sequences and free splitting units to verify hyperbolicity axioms established by Masur and Minsky in \cite{MasurMinsky:complex1}. We follow the same method here to prove hyperbolicity of $\FS(\Gamma;\A)$. 

As alluded to earlier, Proposition~\ref{PropMMTranslation} may be regarded as a translation of the Coarse Lipschitz and Strong Contraction Axioms into a single statement regarding fold sequences and free splitting units. To state it we need one more definition. 

Given a fold sequence $S_I\mapsto\cdots\mapsto S_K$ and a free splitting $T$ in $\FS(\Gamma;\A)$, an \emph{augmented projection diagram over~$\A$ of depth $J$ from $T$ to $S_I\mapsto\cdots\mapsto S_K$} is a commutative diagram of free splittings and maps rel~$\A$ of the form shown in Figure~\ref{FigureAugProjDiagram} such that each horizontal row is a foldable sequence, the subsequence $T_J \mapsto \cdots \mapsto T_L$ is a fold sequence, and the two rectangles shown are combing rectangles. The diagram obtained from Figure~\ref{FigureAugProjDiagram} by replacing the sequence $T_J \mapsto\cdots\mapsto T_L$ with the composed foldable map $T_J \mapsto T_L$ is therefore an ordinary projection diagram as given in Definition~\ref{DefProjDiagram}. Conversely, any projection diagram as given in Definition~\ref{DefProjDiagram} can be converted into an augmented projection diagram by simply factoring the map $T_J \mapsto T_L$ as a fold sequence. 
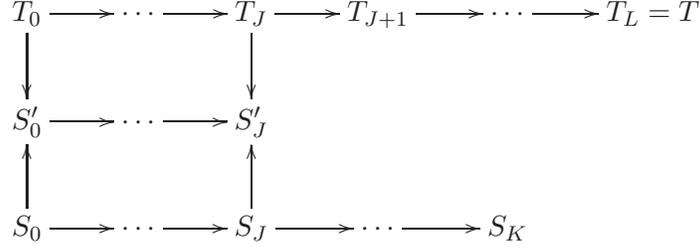
\begin{figure}
$$\xymatrix{
T_I \ar[r] \ar[d] & \cdots \ar[r] & T_{J} \ar[r] \ar[d] & T_{J+1} \ar[r] & \cdots \ar[r] & T_L = T\\
S'_I \ar[r]          & \cdots \ar[r] & S'_{J} \\
S_I \ar[r] \ar[u] & \cdots \ar[r] & S_{J} \ar[r] \ar[u]  & \cdots \ar[r] & S_K \\
}$$
\caption{An augmented projection diagram of depth $J$ from $T$ to $S_I\mapsto\cdots\mapsto S_K$.}
\label{FigureAugProjDiagram}
\end{figure}


\begin{proposition} 
\label{PropMMTranslation}[c.f.\ \FSOne\ Proposition 6.1] 
\qquad Let $b_1 = 5 \corank(\A) + 4 \abs{\A} - 3$, let \hbox{$S_I \mapsto\cdots\mapsto S_K$} be a fold sequence rel~$\A$, and let $\pi \from \FS(\Gamma;\A) \to \{I,\ldots,K\}$ be its associated projection map (as in Definition~\ref{DefProjDiagram}). Let $T$ be a free splitting rel~$\A$ with projection $J=\pi(T)$, and assume that the number of free splitting units rel~$A$ between $S_I$ and $S_J$ is at least $b_1$. Consider any augmented projection diagram rel~$\A$ of depth $J=\pi(T)$ (as denoted in Figure~\ref{FigureAugProjDiagram}), and let~$\Upsilon$ be the number of free splitting units rel~$\A$ between $T_J$ and~$T_L=T$. For any free splitting $R$ rel~$\A$, if $d(T,R) \le \max\{2 \lfloor \Upsilon/b_1 \rfloor,1\}$, then there exists $M \in [I,\pi(R)]$ such that the number of free splitting units between $S_M$ and $S_J$ is at most $b_1$.
\end{proposition}
\noindent
Noting that if $\pi(R) \ge J$ then the conclusion is witnessed by $M=J$, another way to word the conclusion is that the projection of $R$ to $S_I \to\cdots\to S_K$ is no further to the left of $S_J$ than $b_1$ free splitting units.


This proposition will be proved in the next section, using the Big Diagram argument. For now we use it to prove our main results on hyperbolicity of $\FS(\Gamma;\A)$ and on the uniform quasigeodesic parameterization of fold paths using free splitting units.

\begin{proof}[Proof of Theorem \ref{TheoremRelFSGammaHyp}]  
The argument follows closely the proof of hyperbolicity of $\FS(F_n)$ given in \FSOne\ Section~6.1, with Proposition~\ref{PropMMTranslation} standing in for \cite[Proposition 6.1]{\FSOneTag}.
 
We have already verified the \emph{Coarse Retract Axiom} in Lemma~\ref{LemmaCoarseRetract}. Fixing free splittings $T,R \in \FS(\Gamma;\A)$ we must verify the \emph{Coarse Lipschitz Axiom} and the \emph{Strong Contraction Axiom}, which we do with constant $c = 10 b_1 + 8$. After interchanging $T,R$ we may assume $\pi(R) \le \pi(T)=J$. Choose an augmented projection diagram from $T$ to $S_I \mapsto\cdots\mapsto S_K$ of depth $J$ as in Figure~\ref{FigureAugProjDiagram}, and let $\Upsilon$ be the number of free splitting units between $T_J$ and $T_L=T$. Applying Proposition~\ref{PropFSUProps}~\pref{ItemFSUBoundsDiam} we obtain $d(T_J,T) \le 10 \Upsilon + 8$ and so $d(S_J,T) \le 10 \Upsilon + 10$. 

We may assume that the number of free splitting units between $S_I$ and $S_J$ is at least~$b_1$, for otherwise by applying Proposition~\ref{PropFSUProps}~\pref{ItemFSUBoundsDiam} the set $\{S_I,\ldots,S_J\}$ and its subset $\{S_{\pi(R)},\ldots,S_J\}$ each have diameter $\le 10 b_1+8=c$ and the axioms follow.

For the \emph{Coarse Lipschitz Axiom}, if $d(T,R) \le 1$ then Proposition~\ref{PropMMTranslation} applies to produce $M \le J$ such $M \le \pi(R) \le J$ and such that between $S_M$ and $S_J$ there are at most $b_1$ free splitting units, so just as above the set $\{S_M,\ldots,S_J\}$ and its subset $\{S_{\pi(R)},\ldots,S_J\}$ each have diameter $\le 10 b_1 + 8=c$ and the axiom follows. 

For the \emph{Strong Contraction Axiom}, one considers two cases. For the first case where $\Upsilon \le 2b_1$, by Proposition~\ref{PropFSUProps}~\pref{ItemFSUBoundsDiam} we have $d(T,T_J) \le 20b_1 + 8$ and so $d(T,S_J) \le 20 b_1 + 10$, and taking $a=20b_1 + 10$ we may dispense with this case. Consider now the second case that $\Upsilon \ge 2 b_1 \ge 2$. It follows that 
$$\frac{\Upsilon}{b_1} \le 2 \left( \frac{\Upsilon}{b_1} - 1 \right) \quad\text{and so}\quad \frac{\Upsilon}{b_1} \le 2 \biggl\lfloor \frac{\Upsilon}{b_1} \biggr\rfloor
$$
Letting $b = 1/(15 b_1)$, we may then assume that  $d(T,R) \le b \, d(T,S_J)$, and it follows furthermore that 
$$d(T,R) \le b (10 \Upsilon + 10) \le b(10 \Upsilon + 5 \Upsilon) = \frac{\Upsilon}{b_1} \le 2 \biggl\lfloor \frac{\Upsilon}{b_1} \biggr\rfloor
$$
Proposition~\ref{PropMMTranslation} now applies with the conclusion that between there are at most $b_1$ free splitting units between $S_{\pi(R)}$ and $S_J$ and so, just as above, the set $\{S_{\pi(R)},\ldots,S_J\}$ has diameter $\le 10b_1 + 8=c$, and the axiom follows.

Having verified all of the Masur--Minsky axioms, hyperbolicity of $\FS(\Gamma;\A)$ therefore follows from~\cite{MasurMinsky:complex1}.
\end{proof}

\subsection{Theorem \ref{TheoremRelFSUParams}: Parameterizing fold paths using free splitting units}
\label{SectionFSUParameterization}

In the absolute case, fold paths in $\FS(F_n)$ become uniformly quasigeodesic when they are reparameterized based on free splitting units, as proved in \FSOne\ Proposition~6.2. That proof relativizes with very little change to the current setting using free splitting units rel~$\A$, as shown in Theorem~\ref{TheoremRelFSUParams} below. Also, we show in Corollary~\ref{CorollaryCompFSU} that there is a simpler reparameterization, albeit with worse constants, based on component free splitting units.

The free splitting unit parameterization is first described as a discrete parameterization of a fold sequence, and can then be extended to a continuous parameterization of the corresponding fold path. Consider a fold sequence $S_I \mapsto \cdots \mapsto S_J$ rel~$\A$. Let $\Upsilon \ge 0$ be the number of free splitting units rel~$\A$ from $S_I$ to $S_J$; recall that $\Upsilon=0$ if and only if  $S_I$ and $S_J$ differ by $<1$ free splitting unit as in Definition~\ref{DefFSU}~\pref{ItemLessThanOneFSU}. The domain of the discrete parameterization is the integer interval $0 \le u \le \max\{\Upsilon,1\}$; that parameterization is described in separate cases $\Upsilon=0$ and $\Upsilon \ge 1$. If $\Upsilon=0$, equivalently if there is $<1$ free splitting unit between $S_I$ and~$S_J$, then the domain is $\{0,1\}$, we set $m_0=I$ and $m_1=J$, and the parameterization is defined by $u \mapsto S_{m_u}$ for $u=0,1$. If $\Upsilon \ge 1$ then, associated to the fold sequence, there exists a greedy sequence of the form $I = m_0 < m_1 < \cdots < m_\Upsilon = J$ (Proposition~\ref{PropFSUProps}~\pref{ItemFSUGreedy}) (in particular, if $1 \le u \le \Upsilon$ then there is not $<1$ free splitting unit between $S_{m_{u-1}}$ and $S_{m_u}$). The discrete parameterization in this case is defined by $u \mapsto S_{m_u}$ for $u=0,\ldots,\Upsilon$. Regarding the continuous parameterization, consider the fold path in $\FS(\Gamma;\A)$ where each fold $S_{i-1} \mapsto S_i$ is replaced by an edge path in the $1$-skeleton of $\FS(\Gamma;\A)$ of length at most~$2$ (c.f.~Lemma~\ref{LemmaFoldLengthTwo}), hence each fold subsequence $S_i \mapsto \cdots S_j$ is replaced by an edge path of length at most $2\abs{j-i}$. For any consecutive integers $0 \le u-1 < u \le \max\{\Upsilon,1\}$, the discrete parameterization is extended to a continuous parameterization taking the real interval $u-1 \le t \le u$ onto the edge path between $S_{m_{u-1}}$ and $S_{m_u}$ of length $\abs{m_u-m_{u-1}}$. Since the set of vertices $\{S_m \suchthat m_{u-1} \le m \le m_u\}$ has uniformly bounded diameter in $\FS(\Gamma;\A)$ (by Proposition~\ref{PropFSUProps}~\pref{ItemFSUBoundsDiam}), uniform quasi-isometry of the integer parameterization and of the real parameterization are equivalent properties. 


\begin{theorem}[c.f.\ \FSOne\ Proposition 6.2]
\label{TheoremRelFSUParams}
For any fold path $\cdots S_i \mapsto S_{i+1} \mapsto \cdots$ in $\FS(\Gamma;\A)$, 
\begin{enumerate}
\item\label{ItemFSUDistComparison}
Between any two free splittings $S_I,S_K$ on the path with $I \le K$, the number of free splitting units is uniformly quasicomparable to the distance $d(S_I,S_K)$ in $\FS(\Gamma;\A)$, with constants depending only on $\corank(\A)$ and $\abs{\A}$.
\item\label{ItemFSUDistQG}
A reparameterization of the given fold path using free splitting units is uniformly quasigeodesic in $\FS(\Gamma;\A)$. For more details in the discrete case: there exist constants $k=k(\Gamma;\A) \ge 1$ and $c = c(\Gamma;\A) \ge 0$ such that for any fold subpath $S_I \mapsto\cdots\mapsto S_K$ with $\Upsilon$ free splitting units between $S_I$ and $S_K$, the free splitting parameterization $u \mapsto S_{m_u}$, which defined for integers $0 \le u \le \max\{\Upsilon,1\}$, is a $(k,c)$ quasigeodesic parameterization of the subpath in $\FS(\Gamma;\A)$. 
\end{enumerate}
\end{theorem}

\begin{proof} Conclusion~\pref{ItemFSUDistQG} clearly follows from~\pref{ItemFSUDistComparison}. For proving~\pref{ItemFSUDistComparison}, let $\Upsilon$ be the number of free splitting units between $S_I$ and $S_K$, and let $D = d(S_I,S_K)$. We have $D \le 10 \Upsilon + 8$, by Proposition~\ref{PropFSUProps}~\pref{ItemFSUBoundsDiam}. For the opposite inequality, choose a geodesic edge path $\rho$ in $\FS(\Gamma;A)$ between $S_I$ and $S_K$, denoted as
$$S_I = U_0 \,\, \text{---} \,\, \ldots  \,\, \text{---} \,\, U_{i-1} \,\, \text{---} \,\, U_{i}\,\, \text{---} \,\, \ldots    \,\, \text{---} \,\, U_D = S_K
$$ 
Let $\pi \from \FS(F_n;\A) \to \{I,\ldots,K\}$ denote the projection map to the fold path $S_I \mapsto\cdots\mapsto S_K$ (Definition~\ref{DefProjDiagram}), and denote $m(i) = \pi(U_i)$, and so along the fold path we have the following sequence of $D+3$ free splittings:
$$S_I \,\, , \,\, S_{m(0)} \,\, , \,\, \ldots \,\, , \,\, S_{m(i-1)}, \,\, S_{m(i)} \,\, , \,\, \ldots \,\, , \,\, S_{m(D)} \,\, , \,\, S_K
$$
By the Coarse Retract Axiom, Lemma~\ref{LemmaCoarseRetract}, we get a bound $u = u(\Gamma;\A)$ to the numbers of free splitting units between $S_I$ and $S_{m(0)} = S_{\pi(S_I)}$ and between $S_{m(D)} = S_{\pi(S_K)}$ and $S_K$. 
\begin{description}
\item[Claim:] There is a bound $v=v(\Gamma;\A)$ to the number of free splitting units between \break  $S_{m(i-1)}$ and $S_{m(i)}$, for each $i=1,\ldots,D$.
\end{description}
Once this claim is proved, by applying the ``long triangle inequality'' for free splitting units, Proposition~\ref{PropFSUProps}~\pref{ItemLongTriangEq}, it follows that $\Upsilon \le (u + v D + u) + ((D+2)-1) = (v+1)D + 2u+1$. This completes the proof of~\pref{ItemFSUDistComparison}, subject to the claim.

The claim is proved by applying Proposition \ref{PropMMTranslation} by setting $\{T,R\} =  \{U_{i-1},U_i\}$. But keeping in mind that the index sequence $m(i)$ need not be increasing, two cases must be considered:
\begin{itemize}
\item If $m(i-1) \ge m(i)$, set $T=U_{i-1}$ and $R=U_i$; 
\item If $m(i) \ge m(i-1)$, set $T=U_i$ and $R=U_{i-1}$. 
\end{itemize}
We write the proof of the claim only in the second case; the first case is then obtained with only a few notational changes, for the most part just swapping the roles of $m(i-1)$ and $m(i)$. 

Choose an augmented projection diagram from $T=U_i$ to $S_I \mapsto\cdots\mapsto S_K$ of depth $J=m(i)=\pi(U_i)=\pi(T)$ as denoted in Figure~\ref{FigureAugProjDiagram}. To apply Proposition~\ref{PropMMTranslation} we must check its two hypotheses. The upper bound required in the first hypothesis follows from $d(T,R) = d(U_i,U_{i-1})=1$. And we may assume that the second hypothesis also holds, requiring a lower bound $b_1$ to the number of free splitting units between $S_I$ and~$S_J$: if on the contrary that number is $\le b_1$ then, using that $I \le m_{i-1} \le m_i=J$, it follows from Proposition~\ref{PropFSUProps}~\pref{ItemFSUStable} that the number of free splitting units between $S_{m(i-1)}$ and $S_{m(i)}$ is bounded above by $b_1$ and we are done. Having verified all of the hypotheses of Proposition~\ref{PropMMTranslation}, its conclusions therefore hold, producing an index $M \in [I,\pi(R)] = [I,m(i-1)]$ such that the number of free splitting units between $S_M$ and $S_J=S_{m(i)}$ is $\le b_1$. Since $M \le m(i-1) \le m(i)$, applying Proposition~\ref{PropFSUProps}~\pref{ItemFSUStable} again it follows that the number of free splitting units between $S_{m(i-1)}$ and $S_{m(i)}$ is $\le b_1$.
\end{proof}

By applying Theorem~\ref{TheoremRelFSUParams} in conjunction with Lemma~\ref{LemmaCFSU} we immediately obtain:

\begin{corollary}
\label{CorollaryCompFSU}
For any fold path in $\FS(\Gamma;\A)$, 
\begin{enumerate}
\item Between any two free splittings on the path, the number of component free splitting units is quasicomparable to distance in $\FS(\Gamma;\A)$.
\item A reparameterization of the fold path using component free splitting units is a uniform quasigeodesic in $\FS(\Gamma;\A)$.\qed
\end{enumerate}
All quasicomparability and quasigeodesic constants depend only on $\abs{\A}$ and $\corank(\A)$.
\end{corollary}

\noindent
As observed at the end of Section~\ref{SectionMMReview}, one may also conclude that the projection map from $\FS(\Gamma;\A)$ to a fold path $S_I \mapsto\cdots\mapsto S_K$ given by $T \mapsto S_{\pi(T)}$ is a quasi-closest point projection, but in general not an almost closest point projection.

\subsection{The proof of Proposition \ref{PropMMTranslation}: Big Diagrams.} 
\label{SectionBigDiagrams}
Throughout the proof we fix the constant $b_1 = 5 \corank(\A) + 4 \abs{\A} - 3$, the geometric significance of which was established in Lemma~\ref{LemmaCompArc}. 

The proof of Proposition~\ref{PropMMTranslation} is, in essence, a study of the large scale geometry of certain diagrams of fold sequences and combing rectangles, diagrams that may be regarded as living in the relative free splitting complex $\FS(\Gamma;\A)$. We call these ``big diagrams''. We begin the proof by using the hypotheses of Proposition~\ref{PropMMTranslation} to set up the appropriate big diagram, and then we proceed to a study of its large scale geometry.

\paragraph{Constructing the Big Diagram, Step 0.} The reader may refer to Figure~\ref{FigureBigDiagram0} to follow this construction.

Consider a fold sequence $S_I \mapsto\cdots\mapsto S_K$ in $\FS(\Gamma;\A)$ with associated projection map $\pi \from \FS(\Gamma;\A) \to \{0,\ldots,K\}$. Consider also a free splitting $T \in \FS(\Gamma;\A)$ with augmented projection diagram over~$\A$ of depth $J=\pi(T) \in [I,\ldots,K]$ as denoted in Figure~\ref{FigureAugProjDiagram} with $T=T_L$. Along the foldable sequence in the top horizontal line of that augmented projection diagram, add a superscript~$0$, and so that sequence becomes
$$T^0_I \mapsto \cdots \mapsto T^0_J \mapsto \cdots \mapsto T^0_L=T
$$

Consider another free splitting $R \in \FS(\Gamma;\A)$ and consider also any geodesic path from $T^0_L$ to $R$ in the 1-skeleton of $\FS(\Gamma;\A)$. Since the concatenation of two collapse maps is a single collapse map, a geodesic necessarily has the form of a zig-zag path alternating between collapses and expansions. It is convenient for us to slightly alter the geodesic path from $T^0_L$ to $R$ so that it begins with a collapse and ends with an expansion: in order to achieve this, prepend a trivial collapse and/or append a trivial expansion as needed. The result is a path of even length~$D$ of the form 
$$T = T^0_{L} \rightarrow T^1_{L} \leftarrow T^2_{L} \rightarrow \cdots \leftarrow T^D_{L} = R
$$
where $d(T,R) \le D \le d(T,R)+2$. 

If $D=0$ then $T=R$ and we are done. Henceforth we assume $D \ge 2$.

Construct a stack of $D$ combing rectangles atop the foldable sequence $T^0_I \to\cdots\to T^0_L$, by alternately applying relative combing by collapse, Lemma~\ref{LemmaCombingByCollapse}, and relative combing by expansion, Lemma~\ref{LemmaCombingByExpansion}, using the arrows in the path from $T^0_L$ to $T^D_L$, for a total of $D$-applications. The result is the Big Diagram Step 0 depicted in Figure~\ref{FigureBigDiagram0}, in which $T^d_\ell$ denotes the entry in the ``row~$d$'' and ``column~$\ell$'' of the stack of combing rectangles, and in which we have highlighted certain columns and rows. 

Here is the general idea of the proof of Proposition~\ref{PropMMTranslation}. Notice that each column of the Big Diagram Step 0 is a zig-zag path, alternating between collapses and expansions. The diagram has the shape of a piece of corrugated aluminum. The far right edge is, by construction, a geodesic (except possibly for the first and last of its edges). The idea of the proof is that as one sweeps leftward through the Big Diagram, one discovers shorter vertical paths than the ones given in the diagram, allowing one to construct new Big Diagrams with fewer corrugations between the top and bottom rows. Eventually enough corrugations are removed to produce a projection diagram from $R$ to $S_I \mapsto\cdots\mapsto S_K$ from which one can estimate $\pi(R)$.

For each even integer $d$ with $2 \le d \le D-2$ we have a pair of collapse maps of the form $T^{d-1} \xleftarrow{[\rho]} T^d \xrightarrow{[\beta]} T^{d+1}$. If the subgraph $\rho \union \beta \subset T^d$ were proper in $T^d$, then there would be a path $T^{d-2} \rightarrow T^{d-1} \xrightarrow{[\beta']} T^h \xleftarrow{[\rho']} T^{d+1} \leftarrow T^{d+2}$ where $T^h$ is obtained by collapsing $T^d \xrightarrow{[\rho\union\beta]} T^h$, where $\beta'$ is the image of $\beta$ under $T^d \mapsto T^{d-1}$ and $\rho'$ is the image of $\rho$ under $T^d \mapsto T^{d+1}$; these images are proper, which is what allows this subpath to exist. But by concatenating the two collapse maps from $T^{d-2}$ to $T^h$ into a single collapse map, and similarly for the two collapse maps from $T^{d+2}$ to $T^h$ one obtains a shorter path between $T^0_L$ and $R$, contradicting that the chosen path was geodesic. It follows $\rho \union \beta$ is not proper, that is $T^d = \rho \union \beta$.

Letting $\Upsilon$ be the number of free splitting units rel~$A$ between $T_J$ and $T_L$, and letting $\Omega = \lfloor \Upsilon / b_1 \rfloor$, consider the sequence $L=L_0 > L_1 > \cdots > L_\Omega \ge J$ which is obtained from the right greedy sequence by taking only every $b_1^{\text{th}}$ term. By induction it follows for each $1 \le \omega \le \Omega$ that $L_\omega$ is the greatest integer $\le L_{\omega-1}$ such that between $T_{L_\omega}$ and $T_{L_{\omega-1}}$ there are $\ge b_1$ free splitting units. Columns in big diagrams indexed by $L_0,L_1,\ldots,L_\Omega$ will be emphasized as those diagrams evolve.

\smallskip

We have seen that we have a union $T^2_{L_0} = \rho_{L_0} \union \beta_{L_0}$. Knowing this, we may reduce to the case that this union is a blue--red decomposition meaning that $\rho_{L_0} \intersect \beta_{L_0}$ contains no edgelet: if this is not already so then we may alter the diagram to make it so, using exactly the same normalization process described in \FSOne\ Section 6.2. In brief, one replaces row~$T^2$ by collapsing the intersection of red and blue along this row.

\newcommand\bigdiagramcolor[1]{\textcolor{JungleGreen}{#1}}

\begin{figure}
$$\xymatrix{
T^D_I \ar@{.}[d]\ar[r]&\cdots \ar[r]&T^D_{J} \ar@{.}[d] \ar[r] &  \cdots \ar[r]& 
T^D_{L_1} \ar[r]\ar@{.}[d] &  \cdots \ar[r] & T^D_{L_0}\ar@{.}[d] \ar@{=}[r] & R \\
T^4_I \ar[r]\ar[d]&\cdots \ar[r]&T^4_{J}  \ar[d] \ar[r] &  \cdots \ar[r]& 
T^4_{L_1} \ar[r]\ar[d] &  \cdots \ar[r] & T^4_{L_0}\ar[d] \\ 
T^3_I \ar[r]&\cdots \ar[r]&T^3_{J} \ar[r] &  \cdots \ar[r]& 
T^3_{L_1} \ar[r]         &  \cdots \ar[r] & T^3_{L_0} \\ 
T^2_I \ar[r]\ar[d]^{[\rho_I]}\ar[u]_{[\beta_I]}&\cdots \ar[r]&T^2_{J} \ar[u]_{[\beta_{J}]} \ar[d]^{[\rho_{J}]} \ar[r] &  \cdots \ar[r]& 
T^2_{L_1} \ar[r]\ar[d]^{[\rho_{L_1}]}\ar[u]_{[\beta_{L_1}]} &  \cdots \ar[r] & T^2_{L_0}\ar[d]^{[\rho_{L_0}]}\ar[u]_{[\beta_{L_0}]} \\ 
T^1_I \ar[r]&\cdots \ar[r]&T^1_{J} \ar[r] &  \cdots \ar[r]& 
T^1_{L_1} \ar[r]         &  \cdots \ar[r] & T^1_{L_0} \\ 
T^0_I \ar[r] \ar[d] \ar[u]& \cdots \ar[r] & T^0_{J} \ar[u] \ar[r] \ar[d] &  \cdots \ar[r] & 
T^0_{L_1} \ar[r]\ar[u] &  \cdots \ar[r] & T^0_{L_0} \ar[u]  \ar@{=}[r] & T
\\
\bigdiagramcolor{S'_I} \ar[r]          & \cdots \ar[r] & \bigdiagramcolor{S'_{J}} \\
\bigdiagramcolor{S_I} \ar[r] \ar[u] & \cdots \ar[r] & \bigdiagramcolor{S_{J}} \ar[r] \ar[u]  & \cdots \ar[r] & \bigdiagramcolor{S_K} \\
}
$$
\caption{The Big Diagram, Step 0. Certain columns $L=L_0$, $L_1$, \ldots are emphasized, using free splitting units along the fold path $T^0_J \to \cdots \to T^0_L$. 
As the Big Diagram evolves, and up until nearly the end of the evolution, the original projection diagram atop which the diagram is built, which involves the $S'$ and~$S$ rows, will not change. Those rows will be suppressed in the meantime, returning only in the Penultimate Diagram of Figure~\ref{FigurePentDiagram}.}
\label{FigureBigDiagram0}
\end{figure}
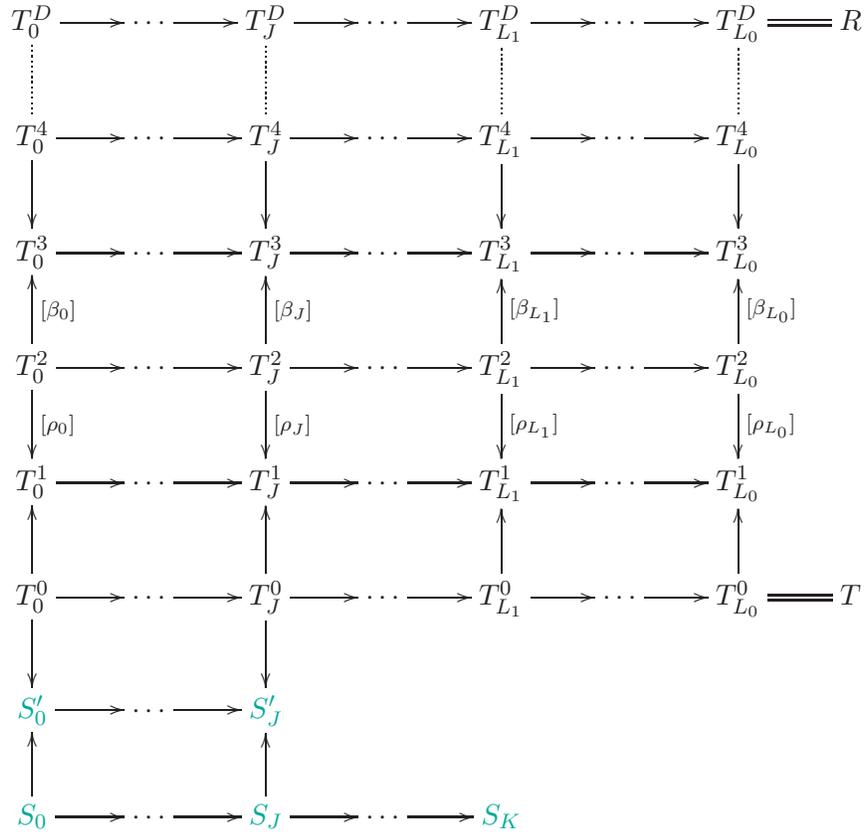

As in \FSOne, the heart of the argument is an induction, starting with the Big Diagram step 0 and producing Big Diagrams steps $1$, $2$, \ldots, $(D-2)/2$, each of which consists of a stack of combing diagrams grouped into successive pairs forming collapse--expand diagrams. At each step the number of combing rectangles decreases by~2, the final diagram at step $(D-2)/2$ being just a stack of 2 combing diagrams forming a single collapse--expand diagram. At all stages of the induction we highlight column~$T_J$, the number $J$ being the projection of $T$ onto $S_I \mapsto\cdots\mapsto S_K$. Throughout the induction we suppress the projection diagram atop which all big diagrams are formed. In particular the foldable sequence $T^0_I \mapsto\cdots\mapsto T^0_J$ is unaltered up until the case of the Big Diagram step~$(D-2)/2$, at which point we again highlight the projection diagram, obtaining in that case a stack of 4 combing rectangles. At that step we carry out one final alteration, producing a stack of 2 combing rectangles forming a projection diagram from $R$ to $S_I \mapsto\cdots\mapsto S_K$ the depth of which is no more than $b_1$ free splitting units to the left of~$S_J$. 

For the induction step, assuming that $D \ge 4$, we adopt variations introduced by Bestvina and Feighn in \cite{BestvinaFeighn:subfactor} for the method of successively producing the next Big Diagram. We describe in detail the first step of the induction, going from step~0 in Figure~\ref{FigureBigDiagram0} to step~1 in Figure~\ref{FigureBigDiagram1}; further steps of the induction are then described very briefly. The induction is complete at step $(D-2)/2$, after which there will be one final special alteration step, to be described in detail later. 

\subparagraph{The first induction step when $D \ge 4$.} Consider the collapse--expand diagram defined by the subrectangle $T^d_\ell$ for $(\ell,d) \in [L_1,\ldots,L_0] \times [0,1,2]$, along the top row of which we have an invariant blue--red decomposition $T^2_\ell = \beta_\ell \union \rho_\ell$. 

The key observation that gets the construction started is that $\beta_{L_1}$, the collapse forest for the map $T^2_{L_1} \xrightarrow{[\beta_{L_1}]} T^3_{L_1}$, has a component $[x,y]$ which is a subarc of the interior of some natural edge of the free splitting $T^2_{L_1}$. This follows from the fact that there are $\ge b_1$ free splitting units between $T^0_{L_1}$ and $T^0_{L_0}$, by applying Proposition~\ref{PropFSUProps}~\pref{ItemFSUBoundsUpsComp}. Let $b \subset \beta_{L_1}$ be the blue edgelet in $[x,y]$ with endpoint~$x$. Let $e$ be the red edgelet not in $[x,y]$ with endpoint~$x$. Factor the collapse map $T^2_{L_1} \xrightarrow{[\beta_{L_1}]} T^3_{L_1}$ as a product of two collapse maps as follows. The first factor collapses everything in $\beta_{L_1}$ except the orbit of $b$, collapsing the subgraph $\beta'_{L_1} = \beta_{L_1} \setminus F_n \cdot b$, and taking $b$ to an edgelet $b' \subset T^{3b}_{L_1}$. The second factor collapses the orbit of $b'$:
$$T^2_{L_1} \xrightarrow{[\beta'_{L_1} \, = \,\, \beta_{L_1} \setminus F_n \cdot b]} T^{3b}_{L_1} \xrightarrow{[F_n \cdot b']} T^3_{L_1}
$$
Note that $b'$ is contained in the interior of some natural edge $\eta'$ of $T^{3b}_{L_1}$. Also, letting $e' \subset T^{3b}_{L_1}$ be the image of $e$, note that arc $e' \union b'$ is also contained in $\eta'$. 

\medskip\textbf{Remark.} The particular way in which the edgelets $b$ and $e$ are used in the above paragraph is an innovation of Bestvina and Feighn in \cite{BestvinaFeighn:subfactor}, arising from dropping the gate~3 condition on fold paths, and having the effect of simplifying the Big Diagram argument.

\smallskip

The collapse map $T^{3b}_{L_1} \xrightarrow{[F_n \cdot b']} T^3_{L_1}$ is equivariantly homotopic to a homeomorphism $h \from T^{3b}_{L_1} \mapsto T^3_{L_1}$ as follows. The homotopy is stationary off of the orbit of $e' \union b'$. Restricted to the arc $e' \union b'$, the collapse is a quotient map taking $b'$ to a point, and that quotient map is homotopic, relative to the endpoints of the arc $e' \union b'$, to a homeomorphism; extend that restricted homotopy over the orbit of $e' \union b'$.

Using the above concatenation of two collapse maps, the combing rectangle $(\ell,d) \in [0,L_1] \times [2,3]$ factors it into a concatenation of two combing rectangles of the form shown in Figure~\ref{FigureFactoredCombing}, whose right side is the above factorization of the collapse map $T^2_{L_1} \mapsto T^3_{L_1}$ (here and later we silently apply the obvious generalizations to $\FS(\Gamma;\A)$ of the results of Section 4.3 of \FSOne\ which construct compositions and decompositions of combing rectangles).
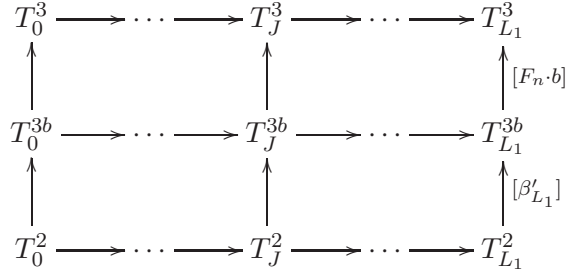
\begin{figure}
$$\xymatrix{ 
T^{3}_I \ar[r] &\cdots \ar[r] &T^{3}_{J} \ar[r] &\cdots \ar[r] & 
T^{3}_{L_1} \\
T^{3b}_I \ar[r] \ar[u] &\cdots \ar[r] &T^{3b}_{J} \ar[r] \ar[u] &\cdots \ar[r] & 
T^{3b}_{L_1} \ar[u]_{[F_n \cdot b]} \\ 
T^2_I \ar[r]\ar[u] &\cdots \ar[r]&T^2_{J} \ar[u] \ar[r] &  \cdots \ar[r]& 
T^2_{L_1} \ar[u]_{[\beta'_{L_1}]} 
}$$
\caption{Factoring the combing rectangle between rows $2,3$ and columns $0,\ldots,L_1$.}
\label{FigureFactoredCombing}
\end{figure}

Now we proceed from step 0 to step 0.1, depicted in Figure~\ref{FigureBigDiagram0.1}. Starting from the step 0 diagram depicted in Figure~\ref{FigureBigDiagram0}, discard the portion of the diagram that lies strictly below row $3$ and right of column~$L_1$, and the portion strictly above row $3$ and left of column~$L_1$. Next, replace the combing rectangle $(\ell,d) \in [0,L_1] \times [2,3]$ by inserting a certain portion of the two concatenated combing rectangles from Figure~\ref{FigureFactoredCombing}, namely, the lower of the two combing rectangles between row $2$ and row $3b$, plus the collapse map $T^{3b}_{L_1} \mapsto T^3_{L_1}$; do not insert any part of row $3$ to the left of $T^3_{L_1}$, nor any of the vertical arrows to the left of the collapse map $T^{3b}_{L_1} \mapsto T^3_{L_1}$. And  now replace the collapse map $T^{3b}_{L_1} \mapsto T^3_{L_1}$ by the equivariant homeomorphism $h \from T^{3b}_{L_1} \to T^3_{L_1}$, and using that homeomorphism identify the free splittings $T^{3b}_{L_1} \approx T^3_{L_1}$. This ostensibly completes the construction of the Big Diagram step 0.1 shown in Figure~\ref{FigureBigDiagram0.1}. 

Unfortunately, the map $h \from T^{3b}_{L_1} \approx T^3_{L_1}$ is not simplicial, because $h(x)$ is not a vertex of $T^3_{L_1}$. But $h$ does become simplicial, after subdividing $T^3_{L_1}$ at the orbit of $h(x)$. Unfortunately, after this subdivision the maps $T^4_{L_1} \mapsto T^3_{L_1} \mapsto T^3_{L_1+1}$ are no longer simplicial. To resolve this issue once and for all, we push the subdivision of $T^3_{L_1}$ up and to the right, throughout the upper right rectangle of Figure~\ref{FigureBigDiagram0.1} defined by $(\ell,d) \in [L_1,L_0] \times [3,D]$, restoring that all maps in this rectangle are simplicial. Do this restoration by the following procedure: first push the subdivision forward along the row $T^3_{L_1} \mapsto \cdots \mapsto T^3_{L_0}$ using the fold maps of that row; then pull the subdivision back to the row $T^4_{L_1} \mapsto\cdots\mapsto T^4_{L_0}$ under the collapse maps from row 4 to row 3; then push the subdivision forward to the row $T^5_{L_1} \mapsto\cdots\mapsto T^5_{L_0}$ under the collapse maps from row $4$ to row $5$; etc. Using the simplicial homeomorphism $h$ we may now identify $T^{3b}_{L_1}$ and $T^3_{L_1}$, truly completing the construction of the Big Diagram step 0.1.
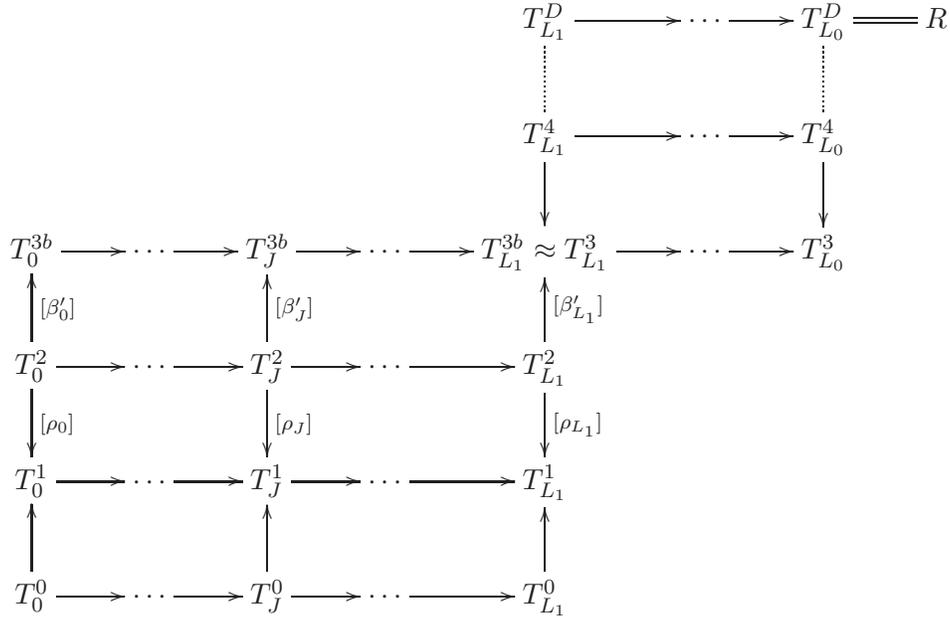
\begin{figure}
$$\xymatrix{
&&&&
T^D_{L_1} \ar[r]\ar@{.}[d] &  \cdots \ar[r] & T^D_{L_0}\ar@{.}[d] \ar@{=}[r] & R \\
&&&&
T^4_{L_1} \ar[r]\ar[d] &  \cdots \ar[r] & T^4_{L_0}\ar[d] \\ 
T^{3b}_I \ar[r]&\cdots \ar[r]&T^{3b}_{J} \ar[r] &  \cdots \ar[r]& 
T^{3b}_{L_1} \approx T^3_{L_1} \ar[r]         &  \cdots \ar[r] & T^3_{L_I} \\ 
T^2_I \ar[r]\ar[d]^{[\rho_I]}\ar[u]_{[\beta'_I]}&\cdots \ar[r]&T^2_{J} \ar[u]_{[\beta'_{J}]} \ar[d]^{[\rho_{J}]} \ar[r] &  \cdots \ar[r]& 
T^2_{L_1} \ar[d]^{[\rho_{L_1}]}\ar[u]_{[\beta'_{L_1}]}  \\ 
T^1_I \ar[r]&\cdots \ar[r]&T^1_{J} \ar[r] &  \cdots \ar[r]& 
T^1_{L_1} \\ 
T^0_I \ar[r] \ar[u]& \cdots \ar[r] & T^0_{J} \ar[u] \ar[r] &  \cdots \ar[r] & 
T^0_{L_1} \ar[u] 
}$$
\caption{The Big Diagram, step 0.1}
\label{FigureBigDiagram0.1}
\end{figure}

We must show that the following row in Figure~\ref{FigureBigDiagram0.1} is a foldable sequence:
$$T^{3b}_I \mapsto \cdots \mapsto T^{3b}_{J} \mapsto \cdots \mapsto T^{3b}_{L_\Omega} \mapsto\cdots\mapsto T^{3b}_{L_1} \approx T^3_{L_1} \mapsto\cdots\mapsto T^3_{L_0}
$$
where the homeomorphism $h$ is used to identify $T^{3b}_{L_1} \approx T^3_{L_1}$. By construction it is foldable from $T^{3b}_I$ to $T^{3b}_{L_1}$ and from $T^3_{L_1}$ to $T^3_{L_0}$, and so it suffices to show that if $0 \le \ell \le L_1$ then the map $T^{3b}_{\ell} \mapsto T^3_{L_0}$ has at least two gates at each vertex $y_\ell \in T^{3b}_{\ell}$. Let various images of $y_\ell$ under the maps in Figure~\ref{FigureFactoredCombing} be denoted $y_{L_1} \in T^{3b}_{L_1}$, $z_\ell \in T^3_\ell$, and $z_{L_1} \in T^3_{L_1}$, and so we have $z_{L_1} = h(y_{L_1})$. There are two cases depending on whether $y_{L_1} \in F_n \cdot b$. If $y_{L_1} \not\in F_n \cdot b$ then $y_\ell$ is not in the collapse graph of the map $T^{3b}_{\ell} \mapsto T^3_{\ell}$, and so under this collapse map the directions at $y_\ell$ and at $z_\ell$ correspond bijectively as do the directions at $y_{L_1}$ and at $z_{L_1}$. The gates at $y_\ell$ and at $z_\ell$ for the maps to $T^3_{L_1}$ therefore also correspond bijectively, and so the gates at $y_\ell$ and at $z_\ell$ for the maps to $T^3_{L_0}$ correspond bijectively, but at $z_\ell$ there are at least two such gates, and so at $y_\ell$ there are also at least two such gates. If $y_{L_1} \in F_n \cdot b$ then $y_{L_1}$ has valence~2 as does $z_{L_1}$, and at $y_\ell$ the map to $T^3_{L_1}$ has exactly two gates, one for each direction at $z_{L_1}$; those two directions map to two different directions in $T^3_{L_0}$, and so there are two gates at $y_\ell$ for the map $T^{3b}_\ell \to T^3_{L_0}$.

Next we proceed from step 0.1 to step 0.2, depicted in Figure~\ref{FigureBigDiagram0.2}. Starting from the step 0.1 diagram depicted in Figure~\ref{FigureBigDiagram0.1}, apply relative combing by collapse, Lemma~\ref{LemmaCombingByCollapse}, and relative combing by expansion, Lemma~\ref{LemmaCombingByExpansion}. These are applied alternately to insert $D-3$ combing rectangles into the upper left corner of step 0.1, between row $3b$ and row $D$ and between column~0 and column~$L_1$; we also delete everything strictly below row~$T^4$ and right of~$T_{L_1}$; the result is shown in Figure~\ref{FigureBigDiagram0.2}, with names $T^d_i$ re-used in the restored upper left corner. 

We note that in Figure~\ref{FigureBigDiagram0.2}, for each $4 \le d \le D$, the rectangle between rows $T^{d-1}$ and $T^d$ is a combing rectangle from column $I$ to $L_1$, and from column $L_1$ to $L_0$, and these piece together to form a single combing rectangle from column $I$ to $L_0$. In particular, for each $4 \le d \le D$ the $T^d$ row, from column $I$ to column $L_0$, is a foldable sequence. This all follows by applying the uniqueness clauses in the statements of relative combing by collapse, Lemma~\ref{LemmaCombingByCollapse}, and relative combing by expansion, Lemma~\ref{LemmaCombingByExpansion}.
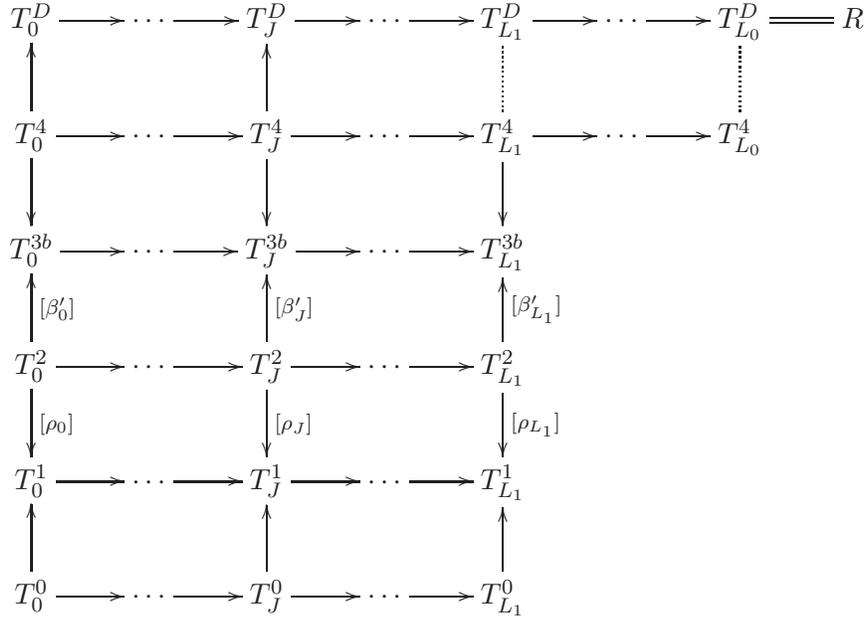
\begin{figure}
$$\xymatrix{
T^{D}_I \ar[r]&\cdots \ar[r]&T^{D}_{J} \ar[r] &  \cdots \ar[r]&T^D_{L_1} \ar[r]\ar@{.}[d] &  \cdots \ar[r] & T^D_{L_0}\ar@{.}[d] \ar@{=}[r] & R \\
T^4_I \ar[r]\ar[d]\ar[u]&\cdots \ar[r]&T^4_{J} \ar[u] \ar[d] \ar[r] &  \cdots \ar[r]&  
T^4_{L_1} \ar[r]\ar[d] &  \cdots \ar[r] & T^4_{L_0} \\ 
T^{3b}_I \ar[r]&\cdots \ar[r]&T^{3b}_{J} \ar[r] &  \cdots \ar[r]& 
T^{3b}_{L_1} \\ 
T^2_I \ar[r]\ar[d]^{[\rho_I]}\ar[u]_{[\beta'_I]}&\cdots \ar[r]&T^2_{J} \ar[u]_{[\beta'_{J}]} \ar[d]^{[\rho_{J}]} \ar[r] &  \cdots \ar[r]& 
T^2_{L_1} \ar[d]^{[\rho_{L_1}]}\ar[u]_{[\beta'_{L_1}]}  \\ 
T^1_I \ar[r]&\cdots \ar[r]&T^1_{J} \ar[r] &  \cdots \ar[r]& 
T^1_{L_1} \\ 
T^0_I \ar[r] \ar[u]& \cdots \ar[r] & T^0_{J} \ar[u] \ar[r] &  \cdots \ar[r] & 
T^0_{L_1} \ar[u] 
}$$
\caption{The Big Diagram, step 0.2}
\label{FigureBigDiagram0.2}
\end{figure}

Next we proceed to the Big Diagram step 0.3, depicted in Figure~\ref{FigureBigDiagram0.3}. Notice that in~$T^2_{L_1}$ we have an edgelet disjoint union
$$T^2_{L_1} = \underbrace{\rho_{L_1} \union \beta'_{L_1}}_{\kappa_{L_1}} \union (F_n \cdot b)
$$
Define a commutative ``baseball diagram'' of collapse maps:
$$
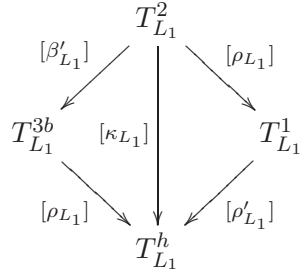
\xymatrix{
& T^{2}_{L_1} \ar[dl]_{[\beta'_{L_1}]} \ar[dd]_{[\kappa_{L_1}]} \ar[dr]^{[\rho_{L_1}]} \\
T^{3b}_{L_1} \ar[dr]_{[\rho_{L_1}]} & & T^1_{L_1} \ar[dl]^{[\rho'_{L_1}]} \\
& T^h_{L_1}
}$$
Using Combing by Collapse on each of the five arrows in this diagram we obtain similar baseball diagrams replacing $L_1$ by any $i \in [0,\ldots,L_1]$. The combing diagrams that correspond to the two arrows from 2nd base $T^2_{L_1}$ to 1st and 3rd bases $T^1_{L_1}$ and $T^{3b}_{L_1}$ are the same as the two combing rectangles depicted in Figure~\ref{FigureBigDiagram0.2} between rows $T^2$ and rows $T^1$ and $T^{3b}$. The Big Diagram step 0.3 is now constructed by replacing those two combing rectangles by the ones that correspond to the two arrows from 1st and 3rd bases to home base $T^h_{L_1}$.
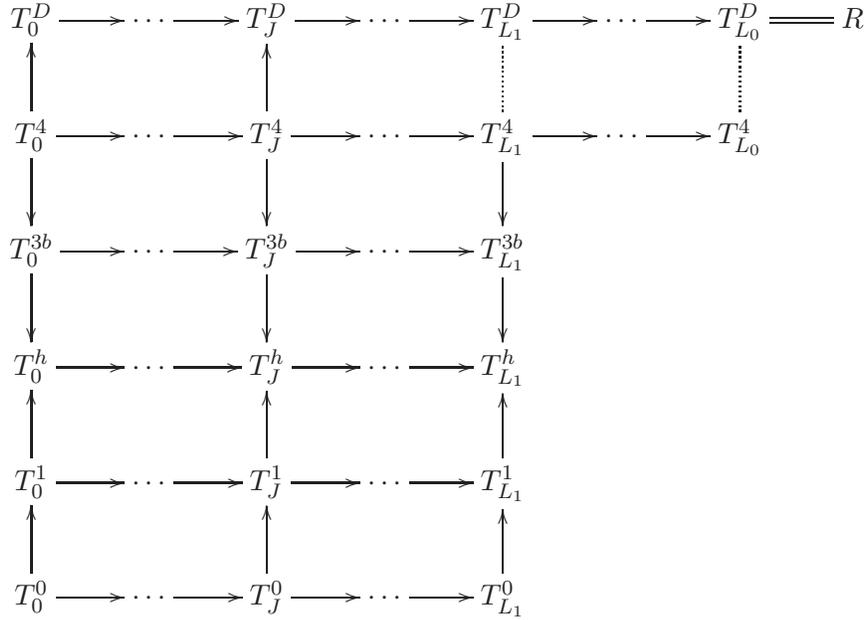
\begin{figure}
$$\xymatrix{
T^{D}_I \ar[r]&\cdots \ar[r]&T^{D}_{J} \ar[r] &  \cdots \ar[r]&T^D_{L_1} \ar[r]\ar@{.}[d] &  \cdots \ar[r] & T^D_{L_0}\ar@{.}[d] \ar@{=}[r] & R \\
T^4_I \ar[r]\ar[d]\ar[u]&\cdots \ar[r]&T^4_{J} \ar[u] \ar[d] \ar[r] &  \cdots \ar[r]&  
T^4_{L_1} \ar[r]\ar[d] &  \cdots \ar[r] & T^4_{L_0} \\ 
T^{3b}_I \ar[r]\ar[d] &\cdots \ar[r] &T^{3b}_{J} \ar[r]\ar[d]  &  \cdots \ar[r]& 
T^{3b}_{L_1}\ar[d] \\ 
T^h_I \ar[r] &\cdots \ar[r] & T^h_{J} \ar[r] &  \cdots \ar[r]& T^h_{L_1}  \\ 
T^1_I \ar[r]\ar[u] &\cdots \ar[r]&T^1_{J} \ar[r]\ar[u]  &  \cdots \ar[r]& 
T^1_{L_1}\ar[u]  \\ 
T^0_I \ar[r] \ar[u]& \cdots \ar[r] & T^0_{J} \ar[u] \ar[r] &  \cdots \ar[r] & 
T^0_{L_1} \ar[u] 
}$$
\caption{The Big Diagram, step 0.3}
\label{FigureBigDiagram0.3}
\end{figure}

Finally, the Big Diagram step~1, depicted in Figure~\ref{FigureBigDiagram1}, is obtained from step 0.3 by concatenating the two combing rectangles from row $T^0$ to $T^1$ and from row $T^1$ to $T^h$ into a single combing rectangle from row $T^0$ to row $T^h$, and by concatenating the two combing rectangles from row $T^4$ to row $T^{3b}$ and from row $T^{3b}$ to row $T^h$ into a single combing rectangle from row $T^4$ to row $T^h$.
\begin{figure}
$$\xymatrix{
T^{D}_I \ar[r]\ar@{.}[d] &\cdots \ar[r]&T^{D}_{J} \ar[r]\ar@{.}[d] &  \cdots \ar[r]&T^D_{L_1} \ar[r]\ar@{.}[d] &  \cdots \ar[r] & T^D_{L_0}\ar@{.}[d] \ar@{=}[r] & R \\
T^4_I \ar[r]\ar[d]&\cdots \ar[r]&T^4_{J} \ar[d] \ar[r] &  \cdots \ar[r]&  
T^4_{L_1} \ar[r]\ar[d] &  \cdots \ar[r] & T^4_{L_0} \\ 
T^h_I \ar[r] &\cdots \ar[r] & T^h_{J} \ar[r] &  \cdots \ar[r]& T^h_{L_1}  \\ 
T^0_I \ar[r] \ar[u]& \cdots \ar[r] & T^0_{J} \ar[u] \ar[r] &  \cdots \ar[r] & 
T^0_{L_1} \ar[u] 
}$$
\caption{The Big Diagram, step 1}
\label{FigureBigDiagram1}
\end{figure}
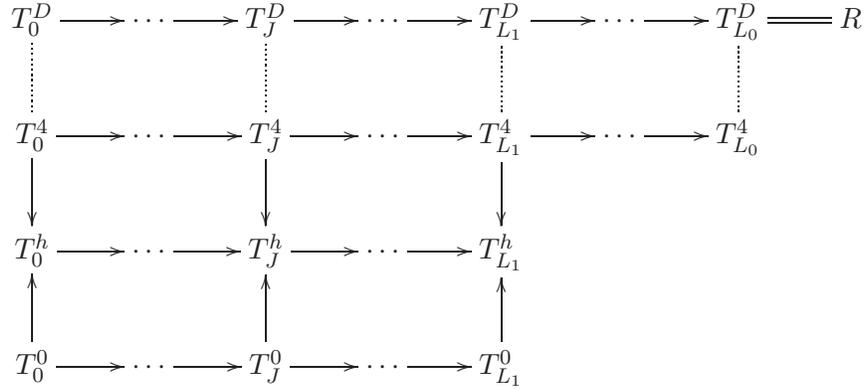
This completes the first step of the induction, constructing the Big Diagram step~$1$ from the Big Diagram step~0. 

\subparagraph{Further induction steps.} Continuing to assume that $D \ge 4$, each further induction step for $2 \le d \le (D-2)/2$ starts with the Big Diagram \hbox{step~$d-1$}, depicted as in Figure~\ref{FigureBigDiagram1} but with column subscript $L_1$ replaced by $L_{d-1}$ and row superscript $4$ replaced by $2d$. From there one constructs the Big Diagram step~$d$, using a straightforward notational variation of the construction from step~0 to step~1. The key observation which gets the construction started is that the collapse forest for the map $T^{2d}_{L_d} \mapsto T^{2d+1}_{L_d}$ has a component which is contained in the interior of a natural edge of $T^{2d}_{L_d}$. This follows by applying Proposition~\ref{PropFSUProps}~\pref{ItemFSUBoundsUpsComp} together with the fact that the number of free splitting units between $T^0_{L_d}$ and $T^0_{L_{d-1}}$ is greater than or equal to $b_1 = 5 \corank(\A) + 4 \abs{\A} - 3$.

\subparagraph{The final step.} When the induction is complete (which happens immediately if $D=2$), the Big Diagram step~$(D-2)/2$ consists of a single collapse--expand diagram. From this diagram discard everything strictly right of column $J$ and below the top row. Also, from the projection diagram for $T$ depicted in Figure~\ref{FigureAugProjDiagram} discard everything in the $T$ row strictly to the right of column $J$. Then glue these two diagrams together along the two copies of the sequence $T_I \mapsto T_J$, resulting in the Penultimate Diagram shown in Figure~\ref{FigurePentDiagram}. In Figures~\ref{FigurePentDiagram} and~\ref{FigureUltDiagram} we emphasize also column~$M$ where $M \in [I,\ldots,J]$ is chosen maximally so that there are $\ge b_1$ free splitting units between $S_M$ and $S_J$, and hence there are exactly $b_1$ free splitting units between $S_M$ and $S_J$; the existence of $M$ follows from the hypothesis of Proposition~\ref{PropMMTranslation} that there are $\ge b_1$ free splitting units between $S_I$ and $S_J$.
\begin{figure}
$$\xymatrix{
T^{D}_I \ar[r] \ar[d] &\cdots \ar[r] & T^{D}_M \ar[r] \ar[d] &\cdots \ar[r] & T^{D}_{J} \ar[r] \ar[d] &  \cdots \ar[r] & T^D_{L_0}\ar@{=}[r] & R \\
T^{h}_I \ar[r] &\cdots \ar[r] & T^{h}_M \ar[r] &\cdots \ar[r] & T^{h}_{J} \\ 
T_I \ar[r] \ar[d] \ar[u] & \cdots \ar[r] & T_M \ar[r] \ar[d] \ar[u] & \cdots \ar[r] & T_{J} \ar[d] \ar[u] \\
S'_I \ar[r]          & \cdots \ar[r] & S'_M \ar[r]          & \cdots \ar[r] & S'_{J} \\
S_I  \ar[r] \ar[u] & \cdots \ar[r] & S_M \ar[r] \ar[u] & \cdots \ar[r] & S_{J} \ar[r] \ar[u]  & \cdots \ar[r] & S_K \\
}$$
\caption{The Penultimate Diagram, obtained from the original projection diagram in Figure~\ref{FigureAugProjDiagram}, together with the Big Diagram step $(D-2)/2$, by discarding irrelevant portions strictly to the right of column $J$, and then gluing the two diagrams together.  Column $M$ is determined by requiring that $M$ is the largest integer $\le J$ such that between $S_M$ and $S_J$ there are $\ge b_1$ free splitting units.}
\label{FigurePentDiagram}
\end{figure}

The final construction is triggered by the observation that the collapse forest for the map from $T^{\vphantom h}_M$ to $T^h_M$ has a component that is contained in the interior of a natural edge of $T_M$, which follows by applying Proposition~\ref{PropFSUProps}~\pref{ItemFSUBoundsUpsComp} together with the assumption that between $S_M$ and $S_J$ there are $\ge b_1$ free splitting units. Based on this observation, we may now follow the same construction steps as above, the conclusion of which is a diagram of the form shown in Figure~\ref{FigureUltDiagram} (where the notations $T^D_i,T^h_i$ for $I \le i \le M$ have been reused).
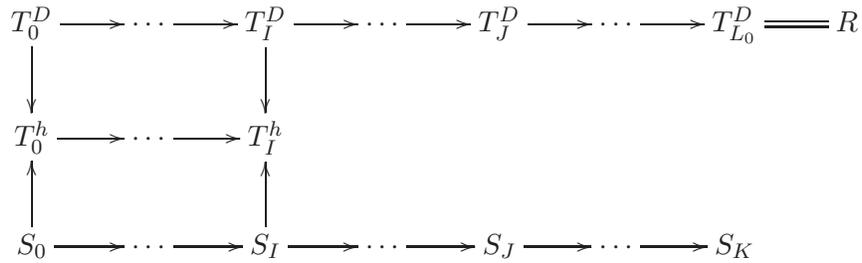
\begin{figure}
$$\xymatrix{
T^{D}_I \ar[r] \ar[d] &\cdots \ar[r] & T^{D}_M \ar[r] \ar[d] &\cdots \ar[r] & T^{D}_{J} \ar[r] &  \cdots \ar[r] & T^D_{L_0}\ar@{=}[r] & R \\
T^{h}_I \ar[r] &\cdots \ar[r] & T^{h}_M  \\ 
S_I  \ar[r] \ar[u] & \cdots \ar[r] & S_M \ar[r] \ar[u] & \cdots \ar[r] & S_{J} \ar[r]  & \cdots \ar[r] & S_K \\
}$$
\caption{The Ultimate Diagram. Notations $T^D_i,T^h_i$ from Figure~\ref{FigurePentDiagram} have been reused, for $0 \le i \le M$.}
\label{FigureUltDiagram}
\end{figure}
This is a projection diagram from $R$ to $S_I \mapsto\cdots\mapsto S_K$, having depth~$M$. The maximal depth of such a projection diagram, which by definition is $\pi(R)$, therefore satisfies $\pi(R) \ge M$, finishing the proof of Proposition~\ref{PropMMTranslation}.

\subsection{Addendum to the Masur--Minsky axioms: The quasi-closest point property}
\label{SectionQuasiClosest}

For application in Part~III of this work~\cite[Section 3.7]{\STLTwoTag} we need a~new quantitative property --- the \emph{quasi-closest point property} --- which is a general consequence of the three Masur--Minsky axioms: the Coarse Retract Axiom; the Coarse Lipschitz Axiom; and the Strong Contraction Axiom. In stating this property we assume that those axioms are satisfied with respect to the following ``givens'': a $1$-dimensional simplicial complex~$X$; a~family of paths~$P$, each $p \in P$ being a discrete path in $X$ of the form $p \from [I_p,\ldots,J_p] \to X^{(0)}$ that satisfies ``almost transitivity'' with respect to a constant~$A$; a family of projection functions of the form $\pi_p \from X^{(0)} \to [I_p,\ldots,J_p]$, one for each $p \in P$; and three more constants $a,b,c$. Here we also recall the symmetrized ``integer interval'' notation $[I,\ldots,J] = [J,\ldots,I] = \{j \in \Z \suchthat I \le j \le J\}$ defined for all $I \le J \in \Z$. Full details of the axioms are found in Section~\ref{SectionMMReview}; they will also be reviewed in context in the proof below.

\begin{proposition}[The quasi-closest point property] 
\label{PropQuasiClosest}
Under the above assumptions, there exist $K,C \ge 0$ depending only on $a$, $b$, $c$ and $A$, such that for each $p \in P$ and $x \in X$, if the index $M \in [I_p,\ldots,J_p]$ is chosen so that $D=d(x,p(M))$ minimizes the distances $d(x,p(\ell))$ over all $\ell \in [I_p,\ldots,J_p]$, then 
$$\text{diam} \bigl( p[\pi_p(x),\ldots,M] \bigr) \le K \log D + C
$$
\end{proposition}

\paragraph{Remarks: ``Quasi-Closest'' versus ``Almost Closest''.} In early drafts of Part~III~\cite{\STLTwoTag}, we had applied an unfortunate assertion, claiming that the projection functions of the Masur--Minsky axioms were \emph{almost} closest point projections, in the sense that the following stronger inequality holds with a constant $C$ independent of $p$ and $x$:
$$\text{diam} \bigl( p[\pi_p(x),\ldots,M] \bigr) \le C
$$
Eventually, while working on late drafts of Part~III, and finding neither a citation nor a proof of the above assertion, we realized that the assertion is \emph{not true} in general. Without affecting the truth of the Masur--Minsky axioms, one can alter the data for those axioms as follows: one can freely alter $b$ so that $b<1$; and then for any given $p \in P$ with $\pi_p[I_p,\ldots,J_p]$ of large diameter one can carefully alter values of $\pi_p$ so that as $x$ goes farther and farther from $\pi_p[I_p,\ldots,J_p]$ the values of $p(\pi_p(x))$ drift farther and farther from the closest point $p(M)$. Grappling with this reality is what led us to this new \emph{Quasi-Closest Point Property} --- which expresses some level of control over the drift --- and to its application in late draft changes of Part~III; see Section~3.7 of Part~III for further remarks on this story.

\newcommand\pmin{p(\ell_{\text{min}})}

\begin{proof}[Proof of the Quasi-Closest Point Property.] We may assume that $a$ is an integer by replacing it with $\lceil a \rceil$; this replacement does not affect the truth of the Masur--Minsky axioms, as one sees by examining the \emph{Strong Contraction Axiom} and using that $d(x,p(\pi_p(x)))$ is an integer. Also, for convenience we extend each projection map $p \from X^{(0)} \to [I_p,\ldots,J_p]$ over the $1$-complex $X$ so that for any $1$-simplex $[v,w]$ and any $x \in (v,w)$ we have $\pi_p(x) \in \{\pi_p(v),\pi_p(w)\}$; again the truth of the Masur--Minsky axioms is unaffected. 

Fixing $p \in P$ and $x \in X^{(0)}$, in the $1$-complex $X$ choose a geodesic path $x(t)$, parameterized by $0 \le t \le D$, from the vertex $x=x(0)$ to the vertex $x(D)=p(M)$. It follows that $p(M)$ --- which is the point on $p[I_p,\ldots,K_p] = \text{image}(p)$ that minimizes distance to $x(0)$ --- is also the point on $\text{image}(p)$ that minimizes distance to $x(t)$ for each $0 \le t \le D$. 

We break into cases, the ``generic'' case being that $a < D$ and $b < 1$.

\medskip\noindent
\emph{\textbf{Case 1:} $D \le a$.} Applying the Coarse Lipschitz Axiom it follows that
$$\diam\bigl(p[\pi_p(x(0)),\ldots,\pi_p(x(D))]\bigr) 
\le \sum_{k=1}^D \diam\bigl(p[\pi_p(x(k-1)),\ldots,p(\pi_p(x(k)]\bigr) \le D c \le a c
$$
and by applying the Coarse Retract Axiom it follows that
$$\diam\bigl(p[\pi_p(x(D)),\ldots,M]\bigr) = \diam\bigl(p[\pi_p(p(M)),\ldots,M]\bigr) \le c
$$
Putting these together we get 
$$\diam(p[\pi_p(x),\ldots,M]) \le (a+1)c
$$

\medskip\noindent
\emph{\textbf{Case 2:} $a < D$ and $b \ge 1$.} The Strong Contraction Axiom applies with $x=x(0)$ and \hbox{$y=x(D)=p(M)$,} because $d(x,p(\pi_p(x))) \ge d(x,p(M)) = d(x(0),x(D)) = D > a$ and \break $d(x,x(D))= d(x(0),x(D)) = D \le bD$; and so from the conclusion of that axiom we have 
$$\diam\bigl(p[\pi_p(x),\ldots,\pi_p(p(M)] \bigr) \le c
$$
Again from the Coarse Retract Axiom we get $\diam\bigl(p[\pi_p(p(M)),\ldots,M]\bigr) \le c$, and so 
$$\diam\bigl(p[\pi_p(x),\ldots,M] \bigr) \le 2c
$$

\medskip\noindent
\emph{\textbf{Case 3:} $a < D$ and $b < 1$.} Inductively subdivide the geodesic $x(t)$ at points 
$$x=x_0, x_1, \ldots, x_{L-1}, x_L = p(M)
$$
so that for $1 \le i \le L-1$ we have 
$$(*) \qquad d(x_{i-1},x_i) = b \, d(x_{i-1},x_L)) = b \, d(x_{i-1},p(M)) \qquad \hphantom{(*)}
$$
and so that $L-1 \ge 1$ is the first index such that $d(x_{L-1},p(M)) < a$, hence $d(x_{L-1},x_L) < a$. Note that the points $x_1,\ldots,x_{L-1}$ need not be in $X^{(0)}$, which is why we extended the projection map in the first paragraph of the proof.

Using that $p(M)$ minimizes distance from $x_{i-1}$ to $\text{image}(p)$, it follows that if~$i \le L-1$ then $d\bigl(x_{i-1},p(\pi_p(x_{i-1}))\bigr) \ge d(x_{i-1},p(M)) \ge a$. Combining this with equation~$(*)$, we can apply the Strong Contraction Axiom with the conclusion that
$$(**) \qquad \text{diam}\bigl(p[\pi_p(x_{i-1}),\pi_p(x_i)]\bigr) \le c \quad\text{for $1 \le i \le L-1$} \qquad\hphantom{(**)}
$$
By applying the Coarse Lipschitz Axiom it follows that
$$\diam\bigl(p[\pi_p(x_{L-1}),\ldots,\pi_p(x_L)]\bigr) \le ac
$$
and by applying the Coarse Retract Axiom we have
$$\diam\bigl(p[\pi_p(x_L),\ldots,M]\bigr) = \diam\bigl(p[\pi_p(p(M)),\ldots,M]\bigr) \le c
$$
Taking these together, we get
$$\diam\bigl(p[\pi_p(x_{L-1}),\ldots,M]  \bigr) \le (a+1)c
$$
Combining this with $(**)$, altogether we get
\begin{align*}
\diam\bigl(p[\pi_p(x),\ldots,M]\bigr) &\le \diam\bigl(p[\pi_p(x_0),\ldots,\pi_p(x_{L-1})]\bigr) + \diam\bigl(p[x_{L-1},\ldots,M]\bigr) \\
&\le \sum_{i=1}^{L-1} \diam\bigl(p[\pi_p(x_{i-1}),\pi_p(x_i)]\bigr) + (a+1)c \\
&\le (L-1)c + (a+1)c = Lc + ac
\end{align*}
It remains to bound $L$. Note that if $0 < i \le L-1$ then $d(x_i,x_L) = (1-b) \, d(x_{i-1},x_L)$ and so, by induction, 
$$d(x_i,x_L) = (1-b)^i D
$$
We may assume $L \ge 2$ in which case 
\begin{align*}
(1-b)^{L-2} D &= d(x_{L-2},x_L) \\
&= d(x_{L-2},p(M)) \\
& \ge a \\
(L-2) \log(1-b) + \log(D) &\ge \log(a) \\
L \log(1-b) &\ge -\log(D) + \log(a) + 2 \log(1-b) 
\end{align*}
Since $0<b<1$, we have $\log(1-b)<0$, and hence
$$L \le \left(-\frac{1}{\log(1-b)}\right) \log(D) + \frac{\log(a)}{\log(1-b)} + 2
$$
\end{proof}

\section{Hyperbolicity of the complex of relative free factor systems}
\label{SectionFFCHyp}

In this section, given a group $\Gamma$ and a free factor system $\A$ of $\Gamma$, we define $\CFFS(\Gamma;\A)$, the complex of free factor systems of $\Gamma$ rel~$\A$ (Section~\ref{SectionFFRelComplex}), we prove that $\CFFS(\Gamma;\A)$ is connected (Section~\ref{SectionFFConnected}), and we prove Theorem~\ref{TheoremRelFFGammaHyp} saying that $\CFFS(\Gamma;\A)$ is hyperbolic (Section~\ref{SectionFFHyperbolicProof}). 

Our proof of hyperbolicity applies the method of Kapovich and Rafi developed in \cite{KapovichRafi:HypImpliesHyp} and used by them to derive hyperbolicity of the free factor complex $\FFC(F_n)$ from hyperbolicity of $\FS(F_n)$. Their general method shows how to derive hyperbolicity of a connected simplicial complex~$Y$ from hyperbolicity of a given connected simplicial complex $X$, by exhibiting a surjective Lipschitz map $f \from X \mapsto Y$ satisfying a simple geometric condition. Intuitively this condition says that if a geodesic in $X$ has its endpoints mapped near each other in $Y$ by the map~$f$, then the entire $f$-image of that geodesic is bounded. We construct the required surjective Lipschitz map $\FS(\Gamma;\A) \to \CFFS(\Gamma;\A)$ in Section~\ref{SectionFFConnected}, and we prove that it satisfies the needed condition on geodesics in Section~\ref{SectionFFHyperbolicProof}.

\subsection{The complex of free factor systems relative to a free factor system.} 
\label{SectionFFRelComplex}

Fix a group $\Gamma$. Define the \emph{unreduced complex of free factor systems of $\Gamma$} to be the simplicial realization of the set of free factor systems with respect to the partial ordering~$\sqsubset$: there is a $0$-simplex for each free factor system $\B$, and more generally a $K$-simplex for each chain of proper extensions of the form $\B_0 \sqsubset \B_1 \sqsubset \cdots \sqsubset \B_K$. This unreduced complex has various undesirable features to be aware of. First, it has diameter $\le 2$: the vertex $\{[\Gamma]\}$ is connected to every other vertex by an edge; also, if $\Gamma$ has a Grushko free factor system~$\A$ --- for example if $\Gamma$ is finitely generated --- then the vertex $\A$ is connected to every other vertex by an edge. Also, if $\Gamma$ does \emph{not} have a Grushko free factor system, then the unreduced complex is infinite dimensional. For our present purposes, the only role of the unreduced complex is to serve as a home for each free factor complex of $\Gamma$ relative to some chosen free factor system. 

Given a free factor system $\A$ of $\Gamma$, the \emph{complex of free factor systems of $\Gamma$ rel~$\A$}, denoted $\CFFS(\Gamma;\A)$, is the flag subcomplex of the unreduced complex of free factor systems of $\Gamma$ that is spanned by those free factor systems $\B_0$ such that each of the two extensions $\A \sqsubset \B_0 \sqsubset \{[\Gamma]\}$ is proper; thus $\CFFS(\Gamma;\A)$ contains a simplex $\B_0 \sqsubset \cdots \sqsubset \B_K$ of dimension~$K$ for each chain of proper extensions 
$$\A \sqsubset \B_0 \sqsubset \cdots \sqsubset \B_K \sqsubset \{[\Gamma]\}
$$
Recall the formula for the free factor system depth of a free factor system~$\A$ (see Section~\ref{SectionDFF}),
$$\DFF(\A) = 2 \corank(\A) + \abs{\A} - 1
$$
Recall also that $\A$ is \emph{exceptional} if and only if \hbox{$\DFF(\A) \le 2$.} 

The following two propositions are both corollaries of Lemma~\ref{LemmaFFSNorm}, together with the analysis of the exceptional case which accompanies that lemma, plus a similar analysis of the case $\DFF(\A)=3$:

\begin{proposition}\label{PropTempDimFF}
For any group $\Gamma$ and any free factor system~$\A$ such that $\DFF(\A) \ge 2$, the complex $\CFFS(\Gamma;\A)$ has dimension $\DFF(\A)-2$, and any simplex is contained in a simplex of maximal dimension $\DFF(\A)-2$. \qed
\end{proposition}


\begin{proposition}
\label{PropExceptionalFFS}
For any group $\Gamma$ and any free factor system $\A$ of $\Gamma$, $\A$ is exceptional if and only if $\CFFS(\Gamma;\A)$ is empty or 0-dimensional, and $\DFF(\A)=3$ if and only if $\CFFS(\Gamma;\A)$ is $1$-dimensional. In more detail:
\begin{enumerate}
\item\label{ItemEdgeFFS}
 $\CFFS(\Gamma;\A) = \emptyset \iff \DFF(\A) \le 1 \iff $ either $\A = \{[\Gamma]\}$, or $\A = \{[A_1],[A_2]\}$ and $\Gamma = A_1 * A_2$.
\item $\CFFS(\Gamma;\A)$ is 0-dimensional $\iff \DFF(\A) = 2 \iff$ one of the following occurs: 
\begin{enumerate}
\item\label{ItemLoopFFS}
$\A = \{[A]\}$ and $\Gamma = A * B$ where $B$ is free of rank~$1$; or
\item $\A = \{[A_1],[A_2],[A_3]\}$ and $\Gamma = A_1 * A_2 * A_3$.
\end{enumerate}
\item\label{ItemFFOneD}
$\CFFS(\Gamma;\A)$ is 1-dimensional $\iff \DFF(\A) = 3 \iff$ one of the following occurs:
\begin{enumerate}
\item\label{ItemFFOneDAZero}
$\A=\emptyset$ and $\Gamma = B$ is free of rank~$2$; or
\item\label{ItemFFOneDATwo}
$\A = \{[A_1],[A_2]\}$ and $\Gamma = A_1 * A_2 * B$ where $B$ is free of rank~1; or
\item\label{ItemFFOneDAFour}
$\A = \{[A_1],[A_2],[A_3],[A_4]\}$ and $\Gamma = A_1 * A_2 * A_3 * A_4$. \qed
\end{enumerate}
\end{enumerate}
 \end{proposition}

\paragraph{Relation to the complex of free factors.} In ranks $n \ge 3$ there is a natural quasi-isometry between the \emph{complex of free factor systems} $\CFFS(F_n) \,\, (= \CFFS(F_n;\emptyset))$ and the \emph{free factor complex} denoted $\FFC(F_n)$, the latter of which was proved to be hyperbolic in \cite{BestvinaFeighn:FFCHyp}; see \cite{HatcherVogtmann:FreeFactors} and \cite{KapovichRafi:HypImpliesHyp} for other closely related complexes. 

Here we define the complex of relative free factor systems $\CFFS(\Gamma;\A)$ that generalizes $\FFC(F_n)$, and we construct a natural quasi-isometry between $\CFFS(\Gamma;\A)$ and $\FFC(\Gamma;\A)$ outside of some low complexity cases. These results require knowing that $\CFFS(\Gamma;\A)$ is connected, in order for its simplicial metric to make sense; for a proof of connectivity we refer to Proposition~\ref{PropConnectedLipschitz}~\pref{ItemConnected_pi} in the next section. We will not use $\FFC(F_n)$ for anything in these works.

First we generalize the definition of the free factor complex $\FFC(F_n)$ from \cite{BestvinaFeighn:FFCHyp} to make it work in our present setting of a group $\Gamma$, a free factor system~$\A$ of $\Gamma$. In ranks $n \ge 3$, $\FFC(F_n)$ is the subcomplex of $\CFFS(F_n)$ consisting of those simplices $\B_0 \sqsubset\cdots\sqsubset \B_K$ such that each $\B_k$ has just one component ($0 \le k \le K$). We define $\FFC(\Gamma;\A)$ as a subcomplex $\CFFS(\Gamma;\A)$ in exactly the same way \emph{except} that the phrase ``just one component'' is replaced with ``just one \emph{nonatomic} component'' (we shall apply this definition only under the hypotheses of Proposition~\ref{PropF_Into_FF_QI} below). Any free factor system~$\B$ rel~$\A$ with just one nonatomic component can be written as $\B = \A' \union \{[B]\}$ for some subset $\A' \subset \A$ and some nonatomic free factor $B$ rel~$\A$, and note that $\A'$ is determined by $[B]$ as the set of all components $[A]$ of $\A$ such that $A$ is \emph{not} conjugate to a subgroup of $B$. The corresponding $0$-simplex of $\FFC(\Gamma;\A)$ may therefore be denoted in a well-defined shorthand as $\V[B]$. Using this notation, the vertex set of every $K$-simplex of $\FFC(\Gamma;\A)$ may be written in the form $\V[B_0],\ldots,\V[B_K]$ for some strictly increasing chain of non-atomic, proper free factors $B_0 \subgroup \cdots < B_K$ rel~$\A$.

\begin{proposition} 
\label{PropF_Into_FF_QI}
If $(\Gamma;\A)$ is nonexceptional and if $\Gamma$ is not a rank~$2$ free group then the simplicial embedding $\FFC(\Gamma;\A) \inject \CFFS(\Gamma;\A)$ is an $\Out(\Gamma;\A)$-equivariant quasi-isometry.
\end{proposition}

\subparagraph{Remarks.} The hypothesis of Proposition~\ref{PropF_Into_FF_QI} is concocted so that \emph{both} of the complexes $\FFC(\Gamma;\A)$ and $\CFFS(\Gamma;\A)$ are connected. 

It also follows from Proposition~\ref{PropF_Into_FF_QI} that $\CFFS(\Gamma;\A)$ is equivariantly quasi-isometric to the ``electrification definition'' of the free factor complex given in \cite{GuirardelHorbez:SubgroupClassification}. 

See also below for a discussion in one \emph{especially} exceptional case.

\begin{proof} Noting that $\CFFS(\Gamma;\A)$ is a bounded neighborhood of its subcomplex $\FFC(\Gamma;\A)$, it suffices to construct a Lipschitz retract $r$ from the $0$-skeleton of $\CFFS(\Gamma;\A)$ to the $0$-skeleton of $\FFC(\Gamma;\A)$. Given a $0$-simplex $\B$ of $\CFFS(\Gamma;\A)$, choose any non-atomic component $[B] \in \B$ and define \hbox{$r(\B) = \V[B] \in \FFC(\Gamma;\A)$.} When $\B$ is already in $\FFC(\Gamma;\A)$, i.e.\ when $\B$ already has just a single non-atomic component, then clearly $r(\B) = \B$, and so $r$ is a retract.

Since $(\Gamma;\A)$ is nonexceptional, i.e.\ $\DFF(\Gamma;\A) \ge 3$, it follows that any two vertices $\CFFS(\Gamma;\A)$ are connected by a chain of $1$-simplices (see Proposition~\ref{PropConnectedLipschitz}~\pref{ItemConnected_pi} below). It therefore suffices to consider any $1$-simplex $\B' \sqsubset \B$ in $\CFFS(\Gamma,\A)$ and to bound the distance between $r(\B')=\V[B']$ and $r(\B)=\V[B]$, where $[B']$, $[B]$ are the chosen nonatomic components of $\B'$, $\B$ respectively. Let $[B^\#] \in \B$ be the unique element such that $B'$ is conjugate to a subgroup of $B^\#$, equivalently $\V[B'] \sqsubset \V[B^\#]$, and so $\V[B']$ and $\V[B^\#]$ have distance $\le 1$ in $\FFC(\Gamma,\A)$.  It therefore suffices to bound the distance in $\FFC(\Gamma,\A)$ between $\V[B]$ and $\V[B^\#]$. If $[B]=[B^\#]$ we are done. In the other case, each of $[B] \ne [B^\#]$ are nonatomic components of~$\B$. We may rechoose $B',B^\#$ in their conjugacy classes so that each is a term in a realization of the free factor system $\B$ having the form 
$$\Gamma = H * B' * B^\# * C
$$ 
where $H$ is either trivial or is a free product of atomic free factors that represent a subset of $\A$, and $C$ is a possibly trivial cofactor. By moving each of $\V[B']$ and $\V[B^\#]$ a distance at most one in $\F(\Gamma;\A)$, we may assume that $B'$ and $B^\#$ are both minimal amongst nonatomic free factors of $\Gamma$ rel~$\A$, and it follows that $B'$ and $B^\#$ are both rank~$1$ cofactors of~$\A$. If either $H$~or~$C$ is nontrivial then $B' * B^\#$ is a proper, nonatomic free factor rel~$\A$ representing a vertex $\V[B' * B^\#]$ at distance~$\le 1$ from each of $\V[B']$ and $\V[B^\#]$. If on the other hand both $H$ and $C$ are both trivial then $\Gamma = B' * B^\#$ is free of rank~$2$, contradicting the hypothesis.
\end{proof}

The cases that are not covered by Proposition~\ref{PropF_Into_FF_QI} include two interesting special cases of Proposition~\ref{PropExceptionalFFS} that we shall consider separately, namely cases (2b) and (3a) --- the remaining cases of Proposition~\ref{PropExceptionalFFS}, namely (1) and (2a), are precisely those for which the relative free splitting complex has bounded diameter.

\subparagraph{The case of $\Gamma = A_1 * A_2 * A_3$ and $\A=\{[A_1],[A_2],[A_3]\}$ (exceptional case (2b)).} This is the only exceptional case whose relative free splitting complex $\FS(\Gamma;\A)$ is ``interesting'' in that it has infinite diameter. But $\CFFS(\Gamma;\A)$ is disconnected in this case --- it is infinite and $0$-dimensional, with simplices of the form $\{[B],[A_i]\}$ where $B$ is a proper, nonatomic free factor rel~$\A$. In this case the function $\{[B],[A_i]\} \mapsto [B]$ defines a bijection with the conjugacy classes of such free factors, and there is an $\Out(\Gamma;\A)$ equivariant embedding $\CFFS(\Gamma;\A) \to \FS(\Gamma;\A)$ taking $[B]$ to the unique free splitting $T[B]$ that has a vertex with stabilizer group $B$; this free splitting has a one-edge fundamental domain having one endpoint stabilized by $B$ and opposite endpoint stabilized by an atomic free factor conjugate to $A_i$. 

One can make a special definition of $\CFFS(\Gamma;\A)$, attaching edges of the form \break \hbox{$\{[B],[A_i]\}$---$\{[B'],[A_j]\}$} whenever $A_k$ is conjugate to a subgroup of both $B$ and $B'$. This occurs in two cases, one of which is $[B]=[B']$ and so $T[B]=T[B']$. The other case occurs when there is a length~$2$ path in $\FS(\Gamma;\A)$ of the form $T[B] \expandsto U \collapsesto T[B']$ where $U$ has a 2-edge fundamental domain with vertex stabilizers of the form $A'_i$, $A_k$, $A'_j$ such that $A'_i * A_k$ is conjugate to $B$ and $A_k * A'_j$ is conjugate to $B'$. The effect of attaching these edges is thus to obtain a subcomplex of $\FS(\Gamma;\A)$ whose inclusion into $\FS(\Gamma;\A)$ is a quasi-isometry.

One can also apply the ``electrification'' definition of $\mathcal \F(\Gamma;\A)$ given in \cite{GuirardelHorbez:SubgroupClassification}: starting with $\FS(\Gamma;\A)$, add an edge between two free splittings $T,T'$ rel~$\A$ whenever there exists $\gamma \in \Gamma$ fixing points in both $T$ and $T'$ such that $\gamma$ is not contained in any atomic free factor rel~$\A$. This happens precisely when $T=T[B]$ and $T'=T[B']$ where $B,B'$ are proper nonatomic free factors each containing $\gamma$; but it then $B=B'$ and $T=T'$ because, in case (2b), the intersection of any two \emph{distinct} non-atomic proper free factors of $\Gamma$ \relA\ is either trivial or atomic. 

To summarize, in the exceptional case (2b), two different attempts at formulating an exceptional definition for $\CFFS(\Gamma;\A)$ have both led to the same outcome, namely a complex that is $\Out(\Gamma;\A)$-equivariantly quasi-isometric to $\FS(\Gamma;\A)$.

\subparagraph{The case of $\Gamma = F_2$ and $\A=\emptyset$ (nonexceptional case (3a)).} When the definition of $\FFC(\Gamma;\A)$ given above is applied to the nonexceptional case of $\FFC(F_2)$, the outcome is nonetheless special in a way: one obtains a $0$-dimensional complex with simplices of the form $\{[B]\}$ where $B \subgroup F_2$ is a rank~$1$ free factor. But it is common for $\FFC(F_2)$ to be re-defined in its own extra special way, by attaching a 1-simplex to each pair of $0$-simplices of the form $[B],[B']$ for which there exists a free factorization $F_2 = B * B'$. By representing the barycenter of each such $1$-simplex using the free factor system $\{[B_1],[B_2]\}$, one sees that $\CFFS(F_2)$ becomes the first barycentric subdivision of $\FFC(F_2)$, and so again we obtain an $\Out(F_2)$-equivariant quasi-isometry $\FFC(F_2) \mapsto \CFFS(F_2)$.

\subsection{Connectivity of $\CFFS(\Gamma;\A)$; a Lipschitz map $\FS(\Gamma;\A) \mapsto \CFFS(\Gamma;\A)$.}
\label{SectionFFConnected}

Consider a group $\Gamma$ and a non-exceptional free factor system $\A$ of $\Gamma$, and so $\DFF(\A) \ge 3$ and $\CFFS(\Gamma;\A)$ has dimension~$\ge 1$ (by Proposition~\ref{PropTempDimFF}). We shall kill two birds (Proposition~\ref{PropConnectedLipschitz}~\pref{ItemConnected_pi} and~\pref{ItemLipschitz_pi}) with one stone (Lemma~\ref{LemmaFSToFFProps}): we prove connectivity of $\CFFS(\Gamma;\A)$; and we describe a map $\FS(\Gamma;\A) \mapsto \CFFS(\Gamma;\A)$ which is Lipschitz with respect to simplicial metrics. 

Define the \emph{projection set map} $\Pi$ from the $0$-skeleton of $\FS(\Gamma;\A)$ to finite subsets of the $0$-skeleton of $\CFFS(\Gamma;\A)$, as follows. Consider a free splitting $T$ of $\Gamma$ rel~$\A$ representing a $0$-simplex $[T] \in \FS(\Gamma;\A)$. Consider also a $0$-simplex of $\CFFS(\Gamma;\A)$, namely a free factor system~$\B$ of $\Gamma$ rel~$\A$ such that $\A \ne \B \ne \{[\Gamma]\}$. To say that $\B$ is \emph{visible} in $T$ means that there exists a collapse map $T \mapsto U$ such that $\B = \Fell U$ (for equivalent reformulations of visibility see \cite[Definition 2.5]{\STLOneTag}). Define $\Pi[T] \subset \CFFS(\Gamma;\A)$ to be the set of all $0$-simplices of $\CFFS(\Gamma;\A)$ that are visible in~$T$.
Here are a few properties of the set map $\Pi$ that we will use without comment in what follows:
\begin{itemize}
\item $\Pi[T]$ is well-defined within the equivalence class of $T$. 
\item The sets $\Pi[T]$ cover the entire $0$-skeleton of $\CFFS(\Gamma;\A)$, as $[T]$ varies over the $0$-skeleton of $\FS(\Gamma;\A)$.
\item The inverted equivariance property: $\Pi\bigl([T] \cdot \phi \bigr) = \phi^\inv \bigl(\Pi[T]\bigr)$ for each $\phi \in \Out(\Gamma;\A)$. 
\item $\Pi[T] \ne \emptyset$. 
\end{itemize}
The first item is evident, the second follows from Lemma~\ref{LemmaFSToFFOnto}, and the third from Lemma~\ref{LemmaFofTEquivariance}. For the fourth item, by applying Proposition~\ref{PropMaximizingSimplices} we obtain a simplex $[T_0] \expandsto [T_1] \expandsto \cdots \expandsto [T_D]$ of maximal dimension $D=\DFF(\A)$ in $\CFFS(\Gamma;\A)$ such that $T_i = T$ for some $i = 0,\ldots,D$. From this simplex we obtain a collapse map $T \mapsto T_0=U$. Applying Proposition~\ref{PropMaximizingSimplices}~\pref{ItemSmallSteps}(c) it follows that $U$ is a one-edge free splitting rel~$\A$, and so $U$ has either one or two vertex orbits. It follows that $\Fell U$ is exceptional: if~$U$~has two vertex orbits then $\Fell U = \{[A'_1],[A'_2]\}$ with $\Gamma  = A'_1 * A'_2$ and so $\DFF(\Fell U) = 1$; if on the other hand $U$ has one vertex orbit then $\Fell U = \{[A']\}$ with $\Gamma = \A' * B$ and $B$ of rank~$1$ and so $\DFF(\Fell U) = 2$. In either case we conclude that $\DFF(\Fell U) \le 2 < 3 \le \DFF(\A)$, implying that $\Fell U \ne \A$ and so $\Fell U \in \Pi[T]$.

Using Bass-Serre theory, we next translate the definition of $\Pi[T]$ into the language of graphs of groups. In the quotient graph of groups $T / \Gamma$, by a \emph{subgraph} $G \subset T / \Gamma$ we mean a subgraph in the ordinary graph theory sense, but having the extra requirement that $G$ contains every vertex of $T/\Gamma$ labelled with a nontrivial vertex group. We say that $G$ is a \emph{core subgraph} if every vertex of $G$ that has valence~$0$ or~$1$ in $G$ is labelled with a nontrivial vertex group. In general any subgraph of $T/\Gamma$ contains a unique maximal core subgraph denoted $\core(G)$: inductively remove any vertex of $G$ having valence~$0$ or~$1$ that is labelled by the trivial group, together with any edge of $G$ incident to that vertex. 

Every subgraph $G \subset T / \Gamma$ represents a free factor system rel~$\A$ that we denote $[G]$: letting $\wt G \subset T$ be the total lift of $G$ via the Bass-Serre universal covering map $T \mapsto T/\Gamma$, define $[G]$ to be the set of conjugacy classes of stabilizers of components of~$\wt G$. We note a few properties of this construction. First, $[G] = \Fell U$ where $T \to U$ is the collapse map that collapses to a point each component of $\wt G$. Also, $[G]=[\core(G)]$. Also, for any nested pair of subgraphs $G \subset G' \subset T/\Gamma$ their associated free factor systems rel~$\A$ are also nested, meaning that $[G] \sqsubset [G']$; in addition, equality $[G]=[G']$ holds if and only if $\core(G)=\core(G')$. Note finally that $[G]=\{[\Gamma]\}$ if and only if $G$ is the improper subgraph $G=T/\Gamma$.

\smallskip

Denote $M = \abs{\A}$ (for the rest of Section~\ref{SectionFFConnected}).

\smallskip
 
To say a core subgraph $G \subset T/\Gamma$ is \emph{trivial} means that either $M=0$ and $G = \emptyset$, or $M \ge 1$ and $G$ consists of $M$ vertices of $T/\Gamma$ labelled with nontrivial vertex groups $A_1,\ldots,A_M$ such that $\A=\{[A_1],\ldots,[A_M]\}$. A trivial core subgraph of $T/\Gamma$ exists if and only if $\Fell T=\A$. The triviality property for core subgraphs is extended to arbitrary subgraphs $G \subset T/\Gamma$ by requiring that $\Core(G)$ be trivial, and so $[G]=\A$ if and only if $G$ is trivial, if and only if $\core(G)$ is trivial.

To complete the Bass-Serre translation, the set $\Pi[T]$ is equal to the following set of $0$-simplices in $\CFFS(\Gamma;\A)$: 
$$\Pi[T] = \{[G] \in \CFFS(\Gamma;\A) \suchthat G \subgroup T / \Gamma \,\, \text{is a proper, nontrivial, core subgraph}\}
$$
To prove this, in the discussion above we have already proved the inclusion $\supset$. For the opposite inclusion $\subset$, given $\Fell U \in \Pi[T]$ and a natural collapse map $T \xrightarrow{[\sigma]} U$, let $\sigma'$ be the union of $\sigma$ with all vertices of $T$ having nontrivial stabilizer, let $G$ be the image of $\sigma'$ under the quotient map $T \mapsto T/\Gamma$, and it follows that $G$ is a proper, nontrivial core subgraph and that $[G]=\Fell U$. 


\begin{lemma} 
\label{LemmaFSToFFProps}
The set map $\Pi$ has the following properties:
\begin{enumerate}
\item\label{ItemFFProjCollapse}
For any collapse of free splittings $S \collapses T$ we have $\Pi[T] \subset \Pi[S]$.
\item\label{ItemFFProjDiamBound}
If $\DFF(\A) \ge 3$ then for each $[T] \in \FS(\Gamma;\A)$ the set $\Pi[T]$ is contained in a connected subcomplex of $\CFFS(\Gamma;\A)$ of simplicial diameter $\le 4$.
\end{enumerate}
\end{lemma}

Before proving Lemma~\ref{LemmaFSToFFProps} we apply it as follows. A function $\pi$ from the $0$-skeleton of $\FS(\Gamma;\A)$ to the $0$-skeleton of $\CFFS(\Gamma;\A)$ is called a \emph{projection map} if $\pi[T] \in \Pi[T]$ for each $[T] \in \FS(\Gamma;\A)$. Projection maps always exist by simply choosing $\pi[T] \in \Pi[T] \ne \emptyset$. If it is so desired, perhaps because of an aversion to wearing out the Axiom of Choice \cite{Weiner:Choice}, for a concretely given group such as $\Gamma = F_n$ there is an explicit construction of a projection map, based on an explicit enumeration of the $0$-skeleta of $\FS(\Gamma;\A)$ and of $\CFFS(\Gamma;\A)$ and an explicit computation of the set map~$\Pi$. 


\begin{proposition}
\label{PropConnectedLipschitz}
Assuming $\DFF(\A) \ge 3$ the following hold:
\begin{enumerate}
\item\label{ItemConnected_pi}
$\CFFS(\Gamma;\A)$ is connected.
\item\label{ItemLipschitz_pi}
For any projection map $\pi$ from the $0$-skeleton on $\FS(\Gamma;\A)$ to the $0$-skeleton of $\CFFS(\Gamma;\A)$ we have:
\begin{enumerate}
\item\label{ItemLipschitz_pi_a}
$\pi$ is $4$-Lipschitz: for any free splitting $T$ of $\Gamma$ rel~$\A$, its set of visible free factor systems $\Pi[T]$ has diameter $\le 4$ in $\CFFS(\Gamma;\A)$.
\item $\pi$ satisfies the ``inverted coarse equivariance property'': $d(\pi[T \cdot \phi], \phi^\inv(\pi[T]))$ has an upper bound depending only on $\corank(\A)$ and $\abs{\A}$, for $[T] \in \FS(\Gamma;\A)$ and $\phi \in \Out(\Gamma;\A)$.
\end{enumerate}
\end{enumerate}
\end{proposition}

\begin{proof} To prove connectivity of $\CFFS(\Gamma;\A)$ it suffices to prove connectivity of its $1$-skeleton. For each $0$-simplex $[T] \in \FS(\Gamma;\A)$ let $\Pi^1[T]$ be the union of all edge paths having endpoints in $\Pi[T]$ and having length $\le 4$.  Connectivity of $\Pi^1[T]$ follows from Lemma~\ref{LemmaFSToFFProps}~\pref{ItemFFProjDiamBound}. This is the basis step of an inductive proof of the following statement: for each edge path $S_0$---$S_1$---$\ldots$---$S_L$ in $\FS(\Gamma;\A)$ the set $\Pi^1[S_0] \cup \cdots \cup \Pi^1[S_L]$ is connected. For the induction step one uses that either $S_{L-1} \collapsesto S_L$ or $S_{L-1} \expandsto S_L$ and therefore by Lemma~\ref{LemmaFSToFFProps}~\pref{ItemFFProjCollapse} the set $\Pi[S_{L-1}] \union \Pi[S_L]$ equals either $\Pi[S_{L-1}]$ or $\Pi[S_L]$; it follows $\Pi^1[S_{L-1}] \union \Pi^1[S_L]$ equals either $\Pi^1[S_{L-1}]$ or $\Pi^1[S_L]$ and so is connected. The union $\union\{\Pi^1[T] \suchthat [T] \in \FS(\Gamma;\A)\}$ is therefore connected, and this union includes the entire $0$-skeleton of $\CFFS(\Gamma;\A)$; the entire $1$-skeleton of $\CFFS(\Gamma;\A)$ is therefore connected, because any 1-cells that are not yet included in this union already have endpoints in the union and so may be freely added.

To prove the Lipschitz bound for $\pi$, for any $1$-simplex $[S] \prec [T]$ in $\FS(\Gamma;\A)$ the set $\Pi[S] \union \Pi[T] = \Pi[T]$ has diameter~$\le 4$, and so $d(\pi[S],\pi[T]) \le 4$.

Inverted coarse equivariance for $\pi$ follows from inverted equivariance for $\Pi$ combined with Lemma~\ref{LemmaFSToFFProps}~\pref{ItemFFProjDiamBound}.
\end{proof}

\begin{proof}[Proof of Lemma~\ref{LemmaFSToFFProps}] To prove conclusion~\pref{ItemFFProjCollapse} choose a collapse map $S \mapsto T$. Each element of $\Pi[T]$ has the form $\Fell U$ for some collapse map $T \mapsto U$, and since the composition $S \mapsto T \mapsto U$ is a collapse map it follows that $\Fell U \in \Pi[S]$.

\smallskip
Turning now to the proof of conclusion \pref{ItemFFProjDiamBound}, we reduce to a special case as follows: 

\smallskip
\noindent
\textbf{Assumption: The free splitting $T$ is generic.} To justify making this assumption, consider an arbitrary free splitting $T$ of $\Gamma$ rel~$\A$. By Proposition~\ref{PropMaximizingSimplices} there exists a generic free splitting $S$ such that $S \collapsesto T$. By applying conclusion~\pref{ItemFFProjCollapse} of Lemma~\ref{LemmaFSToFFProps} (proved just above) it follows that $\Pi[T] \subset \Pi[S]$. So once our desired diameter bound is known to hold under the above assumption, applying that to $S$ we get $\diam(\Pi[S]) \le 4$, and it follows that $\diam(\Pi[T]) \le 4$.

\smallskip
Continuing with the proof under the above assumption, in the quotient graph of groups $T/\Gamma$, the valence~$1$ vertices $\{q_1,\ldots,q_M\}$ are precisely those vertices of $T/\Gamma$ labelled by nontrivial vertex groups $A_1,\ldots,A_M$, those subgroups being free factors of a realization $\Gamma = A_1 * \cdots * A_M * B$ of the free factor system~$\A = \{[A_1],\ldots,[A_M]\}$. Clearly we may assume that $T$ has its natural simplicial structure, and so every vertex of the graph of groups $T/\Gamma$ is either of valence~$1$ and labelled by a nontrivial group  (namely a conjugate of one of $A_1,\ldots,A_M$) or of valence~$3$ and labelled by the trivial group. Also, there must be at least one valence~$3$ vertex, for otherwise one of two cases holds: $M=0$ and so $\A=\emptyset$, the graph $T/\Gamma$ is a circle, and the group $\Gamma$ is infinite cyclic; or $M=2$ and so $\A=\{[A_1],[A_2]\}$, the graph $T/\Gamma$ is an arc, and $\Gamma = A_1 * A_2$. But in each of those two cases we have $\DFF(\A) = 1$ which is ruled out by hypothesis. 

\smallskip
The proof that $\diam(\Pi[T]) \ge 4$ continues by breaking into two cases:

\smallskip
\noindent
\textbf{Case 1: For every pair $G_1 \ne G_2$ of maximal, proper, core subgraphs of $T/\Gamma$ we have $d(G_1,G_2) \le 2$.} Consider any pair of proper, nontrivial core subgraphs $G'_1,G'_2 \subset T/\Gamma$. Since the finite graph $T/\Gamma$ has only finitely many proper, nontrivial core subgraphs, and so each $G'_i$ is contained in some maximal proper, nontrivial core subgraph $G_i \subset T/\Gamma$, and there is a path of length~$\le 1$ connecting $[G'_i]$ to $[G_i]$. From the Case~1 hypothesis we get a path of length $\le 2$ from $[G_1]$ to $[G_2]$, and by concatenation we get a path of length~$\le 4$ from $[G'_1]$ to~$[G'_2]$.

\smallskip 
\noindent
\textbf{Case 2: There exists a pair $G_1 \ne G_2$ of maximal, proper, core subgraphs of $T/\Gamma$ such that $d(G_1,G_2) \ge 3$.} We proceed in two steps. Step 1 will use the Case 2 hypothesis to show that $T/\Gamma$ with its natural cell structure is isomorphic one of three special graphs of low complexity:
\begin{description}
\item[A clam:] the rank~$2$ graph having two valence~$3$ vertices and three edges each with its endpoints at distinct vertices. In this case $\abs{\A}=0$ and $\corank(\Gamma;\A)=2$.
\item[A spindle:] the rank~$1$ graph having two valence~$3$ vertices and two valence~$1$ vertices, a circle consisting of two edges each having endpoints on the two vertices of valence~$3$, and two more edges each with one valence~$3$ endpoint and one valence~$1$ endpoint. In this case $\abs{\A}=2$, and $\corank(\Gamma;\A))=1$.
\item[A clam with an antenna:] the rank~$2$ graph obtained from the clam by attaching one endpoint of an edge to an interior point of one of the edges of the clam. In this case $\abs{\A}=1$, and $\corank(\Gamma;\A)=2$.
\end{description}
Step 2 will verify, for each of these three special cases, that $\Pi[T]$ has diameter~$\le 4$.

\smallskip
\noindent
\textbf{Step 1.} We start by proving that a proper core subgraph $G \subset T/\Gamma$ is maximal if and only if it is obtained from $T/\Gamma$ by removing the interior of a single natural edge having its two ends at distinct vertices. 

The ``if'' direction of this statement is clear. For the ``only if'' direction, given a maximal, proper core subgraph $G$, there are two cases to consider. In the first case $G$ is obtained from $T/\Gamma$ by removing the interior of a single natural edge $e$ having both ends at the \emph{same} vertex. That vertex has valence~$3$ in $T/\Gamma$ and so it has valence~$1$ in $G$, contradicting that $G$ is a core subgraph. In the second case $G$ is missing the interiors of two or more edges of $T/\Gamma$; consider two such edges $e_1,e_2$. If some $e_i$ has distinct endpoints then $T/\Gamma - \interior(e_i)$ is a proper core subgraph larger than $G$, contradicting maximality of~$G$. If no $e_i$ has distinct endpoints then, for each $i=1,2$, the edge $e_i$ forms a loop with both ends at some natural vertex $v_i$ of valence~$3$, and so $v_1 \ne v_2$. Let $e'_i$ be the edge distinct from $e_i$ incident to $v_i$, so $e'_i$ is distinct from both $e_1$ and $e_2$, and $e'_i$ has its ends at two distinct vertices. If $e'_i$ is in $G$ then, as above, $v_i$ has valence~$3$ in $T/\Gamma$ and valence~$1$ in $G$, again contradicting that $G$ is a core graph. If on the other hand $e'_i$ is not in $G$ then $T/\Gamma - \interior(e'_i)$ is a proper core subgraph larger than~$G$, also a contradiction.

Applying the Case~2 hypothesis, choose two maximal, proper, core subgraphs $G_1 \ne G_2 \subset T / \Gamma$ such that \hbox{$d([G_1],[G_2]) \ge 3$} in $\CFFS(\Gamma;\A)$. From Step 1 we have $G_i = T/\Gamma - \interior(e_i)$ for two natural edges $e_1 \ne e_2 \subset T/\Gamma$, and each endpoint set $\bdy e_1$, $\bdy e_2$ consists of two points. The union $\bdy e_1 \union \bdy e_2$ therefore has four, three, or two points, and we handle those cases separately, with various subcases. In each subcase we find that $T/\Gamma$ is a spindle, or a clam maybe with an antenna, or we find a contradiction.

\textbf{Case 1: $\bdy e_1 \union \bdy e_2$ is four points.} In this case $G_1 \intersect G_2$ is a core subgraph, because each of the four points $\bdy e_1 \union \bdy e_2$ has valence~$2$ in $G_1 \intersect G_2$ or valence~$1$ in $T/\Gamma$. If the core graph $G_1 \intersect G_2$ is nontrivial then we have a path of length~$2$ in $T/\Gamma$ of the form $[G_1] \leftarrow [G_1 \intersect G_2] \rightarrow [G_2]$ (the directions of arrows show the direction of inclusions of graphs or, equivalently, direction of nesting of free factor system), which contradicts the choice of $G_1$ and $G_2$. If $G_1 \intersect G_2$ is trivial, all four endpoints must have valence~$1$ in $T/\Gamma$, and it follows that $G_1 \intersect G_2 = \bdy e_1 \union \bdy e_2$. The graph $T/\Gamma$ is therefore the disjoint union of $e_1$ and $e_2$, but that graph is disconnected, a contradiction.


\textbf{Case 2: $\bdy e_1 \union \bdy e_2$ is three points.} Let $\bdy  e_1 \intersect \bdy e_2 = \{p\}$, a single point. In this case $G_1 \intersect G_2$ is not a core graph, because $p$ has valence~$1$ in $G_1 \intersect G_2$ but valence~$3$ in $T / \Gamma$. Let $e_3$ be the edge of $G_1 \intersect G_2$ incident to $p$, and let $H = (G_1 \intersect G_2) - (\{p\} \union \interior(e_3))$. It follows that $\core(H)=\core(G_1 \intersect G_2)$, and so $H$ is trivial but we cannot yet conclude that $H$ itself is a core subgraph. Let $q_i \ne p$ be the endpoint of $e_i$ opposite $p$. There are two subcases, depending on whether $q_3$ equals one of $q_1,q_2$. 

If $q_3$ is distinct from both $q_1$ and $q_2$ then each of $q_1,q_2,q_3$ has valence $2$ in $H$ or valence~$1$ in $T/\Gamma$ and so $H$ is a core subgraph. But $H$ is trivial and so $H = V_{nt} = \{q_1,q_2,q_3\}$ implying that $\abs{\A}=3$, and implying that $T/\Gamma = e_1 \union e_2 \union e_3$ is a tree and so $\corank(\A)=0$. But then $\DFF(\A)=2$, a contradiction.

Suppose that $q_3$ equals one of $q_1$ or $q_2$, say $q_3=q_1$, a point of valence~$3$ in $T/\Gamma$ and of valence~$1$ in $H$, and so $H$ is not a core subgraph. Let $e'$ be the edge of $H$ incident to $q_3$ and let $H' = H - (\{q_3\} \union \interior(e'))$, so $\core(H')=\core(H)=\core(G_1 \intersect G_2)$ and $H'$ is trivial, but again $H'$ need not be a core subgraph. Let $q'$ be the endpoint of $e'$ opposite $q_3$. Depending on whether $q_2=q'$ we will see that $T/\Gamma$ is either a spindle or a clam with an antenna. If $q_2 \ne q'$ then each has valence~$1$ in $T/\Gamma$ or valence~$2$ in $H'$ and so $H'$ is a core subgraph, but $H'$ is trivial and so both $q_2,q'$ have valence~$1$ in $T/\Gamma$, and in this case $T/\Gamma$ is a spindle. If $q_2=q'$ then that point has valence~$1$ in $H'$ and valence~$3$ in $T/\Gamma$, and so $H'$ is not a core subgraph. Letting $e''$ be the edge of $H'$ with endpoint $q''$ opposite $q'$ it follows that $H'' = H' - (\{q'\} \union \interior(e''))$ satisfies $\core(H'') = \core(H') = \core(G_1 \intersect G_2)$ and so $H''$ is trivial. Also, $q''$ has either valence~$2$ in $H''$ or valence~$1$ in $T/\Gamma$ so $H''$ is, at last, a core subgraph. By triviality it follows that $q''$ has valence $1$ in $T/\Gamma$ and that $T/\Gamma$ is a clam with an antenna.

\textbf{Case 3: $\bdy e_1 \union \bdy e_2 = \{p,q\}$ is two points.} These two points each have valence~$1$ in $G_1 \intersect G_2$ and valence~$3$ in $T/\Gamma$. Let $e_p,e_q \subset T/\Gamma$ be the natural edges incident to $p,q$ respectively. If $e_p=e_q$ then $G_1 \intersect G_2 = e_p$, and so $\core(G_1 \intersect G_2) = \emptyset$ and $T/\Gamma$ is a clam. 

We may therefore assume that $e_p \ne e_q$. Let $H' = (G_1 \intersect G_2) - (\{p,q\} \union \interior(e_p) \union \interior(e_q)\}$, so $\core(H')=\core(G_1 \intersect G_2)$, implying that $H'$ is trivial. Let $p',q'$ be the endpoints of $e_p,e_q$ opposite $p,q$ respectively. Depending on whether $p'=q'$ the graph $T/\Gamma$ is either a spindle or a clam with an antenna, which is proved exactly as in Case~2 but with the notation changed to replace $q_2$ in Case~2 with $p'$ in Case~3. 

This completes Step~$1$. 

\smallskip

\textbf{Step 2.} Knowing that $T/\Gamma$ is the clam, a spindle, or a clam with an antenna, we now consider these three graphs one at a time, in each case proving that the set of visible free factor systems $\Pi[T]$ has diameter $\le 4$ in $\CFFS(\Gamma;\A)$.  

\smallskip

\textbf{$T / \Gamma$ is a clam.} This graph has two vertices of valence~3 and three edges denoted $e_1,e_2,e_3$ each with distinct endpoints. The group $\Gamma$ is free of rank~$2$ and $\A=\emptyset$. The three loops in this graph represent three rank~$1$ free factors $Z_1,Z_2,Z_3 \subgroup \Gamma$, with $Z_i$ representing $e_j \union e_k$ for each cyclic permutation $(i,j,k)$ of the indices $(1,2,3)$. The three visible free factor systems of $T$ form the set projection $\Pi[T]=\{[Z_1]\}$, $\{[Z_2]\}$, $\{[Z_3]\}$. In the $1$-skeleton of $\CFFS(\Gamma;\A)$ (which is isomorphic to the barycentric subdivision of the Farey graph), we thus obtain a loop of length~$6$ in the form of a geodesic triangle, using which we see that the diameter of the $\Pi[T]$ is~$\le 2$ (actually $=2$ on the nose).
$$\xymatrix{
      		&  				& [Z_1] \ar[dl] \ar[dr]				\\
      		& \{[Z_1],[Z_2]\} 	& 	& \{[Z_3],[Z_1]\} 	\\
 \{[Z_2]\} \ar[ur] \ar[rr]	&      		& \{[Z_2],[Z_3]\}	&&\{[Z_3]\} \ar[ll] \ar[ul]		\\
}$$
In this diagram, for each $1 \le i \ne j \le 3$ we also include the $0$-cell $\{[Z_i],[Z_j]\}$ that lies between each of the two $0$-cells $\{[Z_i]\}$, $\{[Z_j]\}$, and  that is associated to the free factorization $\Gamma = Z_i * Z_j$. In this and the diagrams below, arrows $X \to Y$ indicate the nesting relation $X \sqsubset Y$.

\smallskip

\textbf{$T / \Gamma$ is a spindle.} In $T/\Gamma$ denote vertices $p,q$ labelled by subgroups with the same notations $P,Q \subgroup \Gamma$. Denote a circle subgraph $C \subset T/\Gamma$ with edges $e_1,e_2$ and endpoints $\bdy e_1=\bdy e_2 = \{u,v\}$, with notation chosen so that in $T/\Gamma$ there are edges $\overline{pu}$ and $\overline{vq}$. We may also choose a rank~$1$ cofactor $Z = \<z\> \subgroup \Gamma$ represented by the subgraph $C$, so that in $\Gamma$ there are two realizations of $\A$ of the form 
$$\Gamma = P * Q * Z \qquad \text{and} \qquad \Gamma = P * Q^z * Z
$$
where $Q^z = z Q z^\inv$. With these notations, the five nontrivial, proper free factor systems that are visible in $T$, associated to the five nontrivial core subgraphs of $T/\Gamma$, can be listed as follows:
\begin{align*}
\mathcal{PZQ} =  \{[P],[Z],[Q]\} = \{[P],[Z],[Q^z]\} & = [\{p\} \union C \union \{q\}]  \\
 \{[P],[Z * Q]\} &= [\{p\} \union C \union \overline{vq})] \\
\{[P * Z],[Q]\}  &= [\overline{pu} \union C \union \{q\}] \\
 \{[P * Q]\} &= [\overline{pu} \union e_1 \union \overline{vq}] \\
 \{[P * Q^z)]\} &= [\overline{pu} \union e_2 \union \overline{vq}]
\end{align*} 
Including two additional free factor systems $\{[P * Q],[Z]\}$ and \hbox{$\{[P * Q^z],[Z]\}$} (which are not visible in $T/\Gamma$) we obtain a connected subgraph of diameter $4$ in the $1$-skeleton of~$\CFFS(\Gamma;\A)$:
$$\xymatrix{
& 	& \{[P],[Z * Q]\}  & & \\
\{[P \! * \! Q]\} \ar[r]
	& \{[P \! * \! Q],[Z]\}  
		& \mathcal{PZQ} \ar[d] \ar[u] \ar[l] \ar[r] 
			& \{[P \! * \! Q^z],[Z]\} 
				& \{[P \! * \! Q^z]\} \ar[l] \\
& 	& \{[P* Z],[Q]\}
}$$

\textbf{$T/\Gamma$ is a clam with an antenna.} In $T/\Gamma$ there is one valence $1$ vertex that we denote~$r$, labelled with a subgroup $R \subgroup \Gamma$ such that $\A=\{[R]\}$, and with incident edge $\overline{ru}$ having opposite valence~3 vertex denoted $u$. There are two other valence~3 vertices $p,q$ and edges $\overline{up}$, $\overline{uq}$, and there are two edges $e_1,e_2$ each having endpoints $p,q$. The two loops $\overline{uq} \union e_1 \union \overline{pu}$ and $\overline{uq} \union e_2 \union \overline{pu}$ represent rank~$1$ free factors $Z_1 = \<z_1\>$, $Z_2 = \<z_2\>$, and the loop $e_1 \union e_2$ represents a free factor $Z_3 = \<z_1 z_2^\inv\>$. The free factor system $\A$ has a realization $\Gamma = R * F$ with a rank~$2$ cofactor $F$ that has several different internal free factorizations we will need:
$$F = Z_1 * Z_2 = Z_2 * Z_3 = Z_3 * Z_1 = \underbrace{(z_1 Z_3 z_1^\inv)}_{Z'_3}  * Z_1
$$
The free splitting $T$ has eight visible free factor systems, given by the eight nontrivial core subgraphs of $T/\Gamma$, as follows. The four such subgraphs do not contain $\overline{ru}$, and the remaining four do contain $\overline{ru}$:
\begin{align*}
\{[R],[F]\} &= \{r\} \union (\overline{pu} \union \overline{uq} \union e_1 \union e_2) \\
\{[R],[Z_1]\} &= \{r\} \union (\overline{pu} \union \overline{uq} \union e_1) \\
\{[R],[Z_2]\} &= \{r\} \union (\overline{pu} \union \overline{uq} \union e_2) \\
\{[R],[Z_3]\} = \{[R],[Z'_3]\} &= \{r\} \union (e_1 \union e_2) \\
\{[R * Z_1]\} &= \overline{ru} \union \overline{pu} \union \overline{uq} \union e_1 \\
\{[R * Z_2]\} &= \overline{ru} \union \overline{pu} \union \overline{uq} \union e_2 \\
\{[R * Z_3]\} &= \overline{ru} \union \overline{pu} \union e_1 \union e_2 \\
\{[R * Z_3']\} &= \overline{ru} \union \overline{qu} \union e_1 \union e_2
\end{align*}
The inclusion lattice of these eight core subgraphs, equivalently the nesting lattice of the eight visible free factor systems, forms a connected graph of diameter~$4$ in the $1$ skeleton of $\FS(\Gamma;\A)$:
$$\xymatrix{
						& \{[R],[F]\}  			& 					&  \\
 \{[R],[Z_1]\} \ar[ur] \ar[d]	& \{[R],[Z_2]\} \ar[u] \ar[d] 	& \{[R],[Z_3]\} \ar[ul] \ar[d] \ar[dr]	& \\
\{[R * Z_1]\} & \{[R * Z_2]\} &  \{[R * Z_3]\} & \{[R * Z_3']\} &
}$$
\end{proof}

\subsection{Proof of hyperbolicity of $\CFFS(\Gamma;\A)$}
\label{SectionFFHyperbolicProof}
We shall prove an enhanced version of Theorem~\ref{TheoremRelFFGammaHyp} by applying the following result of I.\ Kapovich and K.\ Rafi, which is in turn based on results of Bowditch \cite{Bowditch:CurveComplex}:

\begin{proposition}[\protect{\cite[Proposition 2.5]{KapovichRafi:HypImpliesHyp}}]
\label{TheoremKR}
Let $X$ be a connected simplicial complex which is $\delta$-hyperbolic with respect to the simplicial metric and let $Y$ be a connected simplicial complex. Suppose that the following hypothesis holds: there exists a map of $1$-skeleta $\pi \from X^{(1)} \to Y^{(1)}$ such that
\begin{enumerate}
\item $\pi$ is a $K$-Lipschitz graph map, meaning:
\begin{enumerate}
\item\label{ItemKRSurjective}
$\pi$ restricts to a surjection of vertex sets $X^{(0)} \mapsto Y^{(0)}$;
\item\label{ItemKRLipschitz}
The restriction of $\pi$ to each edge of $X$ is an edge path of length~$\le K$ in~$Y^{(1)}$.
 \end{enumerate}
\item\label{ItemKRSmall}
There exists a constant $D$ such that for all $v,w \in X^{(0)}$, if $d_Y(\pi(v),\pi(w)) \le 1$, and if $v=v_0,v_1,\ldots,v_L=w$ are the vertices along a geodesic in the $1$-skeleton $X^{(1)}$ between $v$ and $w$, then $\diam_Y\{\pi(v_0),\ldots,\pi(v_L)\} \le D$.
\end{enumerate}
It follows that $Y$ is $\delta_1$-hyperbolic with respect to the simplicial metric, and the following conclusion also holds:
\begin{itemize}
\item
If $v_0,v_1,\ldots,v_L$ is the sequence of vertices along some geodesic in $X^{(1)}$ then the subset $\{\pi(v_0),\pi(v_1),\ldots,\pi(v_L)\} \subset Y$ is $C$-Hausdorff close to a geodesic in $Y$. 
\end{itemize}
The constants $\delta_1$ and $C$ depend only on $\delta$, $K$, and $D$.
\qed
\end{proposition}

\begin{theorem}[Enhanced version of Theorem~\ref{TheoremRelFFGammaHyp}]
\label{TheoremRelFFHypEnchanced}
For any group $\Gamma$ and any nonexceptional free factor system $\A$ of $\Gamma$, the complex $\CFFS(\Gamma;\A)$ is nonempty, connected, and hyperbolic. Furthermore, the image under $\pi \from \FS(\Gamma;\A) \to \CFFS(\Gamma;\A)$ of any geodesic in $\FS(\Gamma;\A)$ is uniformly Hausdorff close to a geodesic in $\CFFS(\Gamma;\A)$.
\end{theorem}

After the proof, we will record a further enchancement, in which the conclusion is strengthened to get a reparameterized quasigeodesic in $\CFFS(\Gamma;\A)$.

\begin{proof}
Choose a special projection map $\pi \from \FS(\Gamma;\A) \to \CFFS(\Gamma;\A)$ having the property that for each $0$-simplex $[T] \in \FS(\Gamma;\A)$, if $\Fell T \ne \A$ then $\pi[T]=\Fell T$; such a map exists since clearly $\Fell T \in \Pi[T]$. Surjectivity of this special $\pi$ follows from Lemma~\ref{LemmaFSToFFOnto}, and $\pi$ is Lipschitz by Proposition~\ref{PropConnectedLipschitz}~\pref{ItemLipschitz_pi_a}. These are hypotheses~\pref{ItemKRSurjective} and~\pref{ItemKRLipschitz} of Theorem \ref{TheoremKR} above. It remains to verify hypothesis~\pref{ItemKRSmall}.

Consider $0$-simplices $[S],[T] \in \FS(\Gamma;\A)$ such that $d(\pi(S),\pi(T)) \le 1$. 

We first reduce to the case that $\Fell S=\pi(S)$ and that $\Fell T=\pi(T)$; by the special choice of $\pi$ this is equivalent to reducing to the case $\Fell S \ne \A$ and $\Fell T \ne \A$. From the requirement that $\pi[S] \in \Pi[S]$ and $\pi[T] \in \Pi[T]$ it follows that there exist collapse maps $S \collapses S'$ and $T \collapses T'$ such that $\pi(S)=\Fell S' \ne \A$ and $\pi(T)=\Fell T' \ne \A$. From the special choice of $\pi$ it then follows that $\Fell S'=\pi(S')$ and $\Fell T'=\pi(T')$. Noting that in $\FS(\Gamma;\A)$ we have $d(S,S') \le 1$ and $d(T,T') \le 1$, and applying hyperbolicity of $\FS(\Gamma;\A)$, any geodesic connecting $S$ to $T$ stays uniformly Hausdorff close to any geodesic connecting $S'$ to $T'$. Since $\pi$ is Lipschitz it follows that the $\pi$-images of these geodesics are uniformly Hausdorff close in $\CFFS(\Gamma;\A)$. Once we have verified that hypothesis~\pref{ItemKRSmall} holds for an $S',T'$ geodesic, it holds as well for an $S,T$ geodesic, completing the reduction.

Henceforth we assume $\Fell S=\pi(S)$ and $\Fell T=\pi(T)$. Since $d(\Fell S,\Fell T) \le 1$, up to transposing notation we may assume $\Fell S \subset \Fell T$. Combining Lemma~\ref{LemmaFoldableExists} and Lemma~\ref{ThmFoldPathExists}, there exists a collapse map $S \collapses S''$ and a fold sequence from $S''$ to $T$. Since $d(S,S'') \le 1$ it follows, just as in the previous paragraph, that once we have verified the desired conclusions for an $S'',T$ geodesic, the conclusions for an $S,T$ geodesic follow. 

We may also assume that there exists a fold sequence from $S$ to $T$, denoted 
$$S = S_0 \xrightarrow{f_1} S_1 \xrightarrow{f_2} \cdots \xrightarrow{f_L} S_L = T
$$
By Theorem~\ref{TheoremRelFSUParams} the sequence $S_0,S_1,\ldots,S_L$ can be reparameterized as a uniform quasigeodesic. By hyperbolicity of $\FS(\Gamma;\A)$ this quasigeodesic is uniformly Hausdorff close in $\FS(\Gamma;\A)$ to any $S,T$ geodesic. And by the Lipschitz property for $\pi$ the images of the quasigeodesic and the geodesic are uniformly Hausdorff close in $\CFFS(\Gamma;\A)$. It therefore suffices to bound the diameter of the set $\{\pi(S_0),\pi(S_1),\ldots,\pi(S_L)\}$. By Lemma~\ref{LemmaRealCollapse}\pref{ItemRCFFtoFS} we have $\Fell S_0 \sqsubset \Fell S_1 \sqsubset \cdots \sqsubset \Fell S_L$ and so the set $\{\Fell S_0,\Fell S_1,\ldots,\Fell S_L\}$ has diameter~$\le 1$ in $\CFFS(\Gamma;\A)$. Since $\A$ is properly contained in $\Fell S=\Fell S_0$, it follows that $\A$ is properly contained in each of $\Fell S_0,\Fell S_1,\ldots,\Fell S_L$, and so $\pi(S_i)=\Fell S_i$ for $0 \le i \le L$. The set $\{\pi(S_0),\pi(S_1),\ldots,\pi(S_L)\}$ therefore has diameter~$\le 1$.
\end{proof}


\paragraph{Remarks on the exceptional cases of Theorem~\ref{TheoremRelFFHypEnchanced}.}These cases --- where $\DFF(\A) = 2$ and $\CFFS(\Gamma;\A)$ is an infinite discrete set --- can be incorporated into the conclusion if one attaches cells to $\CFFS(\Gamma;\A)$ in an exceptional but still natural manner. By analogy, the curve complex of the torus $T^2$ has a $0$-cell for every isotopy class of essential simple closed curves, and an exceptional rule for attaching a $1$-cell connecting each pair of $0$-cells that are represented by two curves that intersect transversely at one point; the ordinary rule for attaching a $1$-cell, on a surface more complicated than a torus, requires that the two curves be disjoint.

There are two cases to consider when $\DFF(\A)=2$ (see Section~\ref{SectionDFF}). The more interesting case is when $\Gamma = A_1 * A_2 * A_3$ with $\A=\{[A_1],[A_2],[A_3]\}$. The $0$-cells of $\CFFS(\Gamma;\A)$ can all be written in the form $\{[A'_i * A'_j],[A'_k]\}$ for some free factorization $\Gamma = A'_1 * A'_2 * A'_3$ such that $[A_i]=[A'_i]$ and some triple of indices $\{i,j,k\}=\{1,2,3\}$. For any such choice of free factorization and index triple one can then attach a $1$-cell --- in an exceptional but still natural manner --- having as endpoints the two $0$-cells $\{[A'_i * A'_j],[A'_k]\}$ and $\{[A'_i],[A'_j * A'_k]\}$. 

The less interesting case of $\DFF(\A)=2$ occurs when $\A = \{[A]\}$ with $\corank(\A)=1$. In this case the points of $\CFFS(\Gamma;\A)$ can all be written in the form $\{[A],[Z]\}$ where $\Gamma = A * Z$ with rank~$1$ cofactor $Z$. By examining the description of the free splitting complex $\FS(\Gamma;\A)$ given in Proposition~\ref{PropOneCoedgeOneComp}, one obtains an $\Out(\Gamma;\A)$-equivariant embedding $\CFFS(\Gamma;\A) \hookrightarrow \FS(\Gamma;\A)$, taking $\{[A],[Z]\}$ to the unique free splitting of $\Gamma$ rel~$\A$ having one edge orbit, represented by an edge whose two endpoints are stabilized by $A$ and $Z$ respectively. In $\FS(\Gamma;\A)$ there is just one other $0$-cell, namely the unique free splitting having one edge orbit represented by an edge whose two endpoints are stabilized by distinct conjugates of $A$; furthermore, $\FS(\Gamma;\A)$ is obtained by coning from this one other $0$-cell to each of the $0$-cells in the embedded image of $\CFFS(\Gamma;\A)$. In summary, after exceptional cell attachments, i.e.\ taking the cone of $\FS(\Gamma;\A)$, one obtains $\CFFS(\Gamma;\A)$ whose diameter is finite. 

\paragraph{Remarks on reparameterized quasigeodesics in $\CFFS(\Gamma;\A)$.}  We record here one more simple enhancement to Theorem~\ref{TheoremRelFFHypEnchanced}, which we use in Part~III to motivate a conjecture regarding quasigeodesics in $\CFFS(\Gamma;\A)$. 

In a simplicial complex $Y$, a sequence of vertices $V_0,\ldots,V_L$ is a \emph{reparameterized $K,C$ quasigeodesic} if there exists an integer $n \ge 0$ and a nondecreasing, surjective step function \hbox{$r(t) \in \{0,\ldots,L\}$} defined for all $0 \le t \le n$ such that the function defined by $t \mapsto V_{r(t)}$ is a $K,C$ quasigeodesic. 

\begin{theorem}[Addendum to Theorem~\ref{TheoremRelFFHypEnchanced}]
\label{TheoremEnchancedAddendum}
If $v_0,\ldots,v_L$ is the sequence of vertices along a geodesic in the $1$-skeleton of $\FS(\Gamma;\A)$ then the sequence of vertices $\pi(v_0),\ldots,\pi(v_L)$ is a reparameterized quasigeodesic in $\CFFS(\Gamma;\A)$, with uniform quasigeodesic constants.
\end{theorem}

\begin{proof} The conclusions of \cite[Proposition 2.5]{KapovichRafi:HypImpliesHyp} are derived from the conclusions \cite[Proposition 2.3]{KapovichRafi:HypImpliesHyp} which (as said in \cite{KapovichRafi:HypImpliesHyp} just before Proposition 2.3) is a ``slightly restated special case'' of \cite[Proposition 3.1]{Bowditch:CurveComplex}. To enchance all of these conclusions in the desired manner one can then apply \cite[Lemma 3.2]{Bowditch:CurveComplex} which, in our present context, says that the sequence $0,\ldots,L$ has a subsequence $0=m(0) < m(1) < \cdots < m(n) = L$ satisfying the following conditions, with a uniform constant $K$:
\begin{itemize}
\item $\diam\{v_{m(i-1)},v_{m(i-1)+1},\ldots,v_{m(i)}\} \le K$ for all $i=1,\ldots,n$;
\item $\abs{i-j} \le d(v_{m(i)},v_{m(j)}) + 2$ for all $i,j$.
\end{itemize}
\end{proof}

\bibliographystyle{amsalpha} 
\bibliography{/Users/Lee/Dropbox/Handel_Lyman_Mosher_shared/Lee_bibtex_file/mosher.bib} 

\end{document}